\documentclass{amsart}
\usepackage{mathrsfs}
\usepackage{epsfig}
\usepackage{amsmath}
\usepackage{amssymb}
\usepackage{amscd}
\usepackage{graphicx}
\usepackage{color}

\newtheorem{thm}{Theorem}[section]
\newtheorem{lem}[thm]{Lemma}
\newtheorem{prop}[thm]{Proposition}

\newtheorem{ques}[thm]{Question}
\newtheorem{cor}[thm]{Corollary}
\newtheorem{de}[thm]{Definition}
\newtheorem{rem}[thm]{Remark}

\def \N {\mathbb N}

\def \Z {\mathbb Z}
\def \R {\mathbb R}
\def \E {\mathbb E}

\numberwithin{equation}{section}

\begin{document}

\title{Local entropy theory for a \\ countable discrete amenable group action}

\author{Wen Huang, Xiangdong Ye and Guohua Zhang}

\address{Department of Mathematics, University of Science and
Technology of China, Hefei, Anhui 230026, China}

\email{wenh@mail.ustc.edu.cn, yexd@ustc.edu.cn}

\address{School of Mathematical Sciences and LMNS, Fudan University, Shanghai 200433, China}

\email{chiaths.zhang@gmail.com}


\begin{abstract}
The local properties of entropy for a countable discrete amenable
group action are studied. For such an action, a local variational
principle for a given finite open cover is established, from which
the variational relation between the topological and
measure-theoretic entropy tuples is deduced. While doing this it is
shown that two kinds of measure-theoretic entropy for finite Borel
covers are coincide. Moreover, two special classes of such an
action: systems with uniformly positive entropy and completely
positive entropy are investigated.
\end{abstract}

\maketitle

\markboth{Local Entropy Theory for a Countable Discrete Amenable
Group Action}{Wen Huang, Xiangdong Ye and Guohua Zhang}


\section{Introduction}

Rohlin and Sinai \cite{RS} introduced the notion of completely
positive entropy (c.p.e.) for $\Z$-actions on a Lebesgue space. It
is also known as $K$-actions of $\Z$. $K$-actions played an
important role in the classic ergodic theory. In 1992, Blanchard
introduced the notions of uniformly positive entropy (u.p.e.) and
c.p.e. as topological analogues of the $K$-actions in topological
dynamics of $\Z$-actions \cite{B1}. By localizing the concepts of
u.p.e. and c.p.e., he defined the notion of entropy pairs, and used
it to show that a u.p.e. system is disjoint from all minimal zero
entropy systems \cite{B2} and to obtain the maximal zero entropy
factor for any topological dynamical system of $\Z$-actions (namely
the topological Pinsker factor) \cite{BL}. From then on, the local
entropy theory of $\Z$-actions have been made great achievements
\cite{B1, B2, BGH, BHMMR, BL, DYZ, G, HMRY, HY, HYZ1, HYZ2, KL, R,
YZ}, see also the relevant chapters in \cite{G1} and the survey
papers \cite{GW4, GY}. A key point in the local entropy theory of
$\Z$-actions is the local variational principle for finite open
covers.

Note that for each dynamical system $(X, T)$ of $\Z$-actions (or
call it TDS), there always exist $T$-invariant Borel probability
measures on $X$ so that the classic ergodic theory involves the
study of the entropy theory of $(X, T)$. Whereas, there are
some groups $G$ such that there exists no any invariant Borel
probability measures on some compact metric space with $G$-actions,
for example the rank two free group $F_2$. It is well known that,
for a dynamical system of group actions, the amenability of the
group ensures the existence of invariant Borel probability measures,
which includes all finite groups, solvable groups and compact
groups.

Comparing to dynamical systems of $\Z$-actions, the level of
development of dynamical systems of an amenable group action lagged
behind. However, this situation is rapidly changing in recent years.
A turning point occurred with Ornstein and Weiss's pioneering paper
\cite{OW} in 1987 which laid a foundation of an amenable group
action. In 2000, Rudolph and Weiss \cite{RW} showed that $K$-actions
for a countable discrete amenable group is mixing of all orders (an
open important question for years) by using methods from orbit
equivalence. Inspired by this, Danilenko \cite{D} pushed further the
idea used by Rudolph and Weiss providing new short proofs of results
in \cite{GTW, OW, RW, WZ}. Meanwhile, based on the result of
\cite{RW} and with the help of the results from \cite{CFW}, Dooley
and Golodets in \cite{DG(2002)} proved that every free ergodic
actions of a countable discrete amenable group with c.p.e. has a
countable Lebesgue spectrum. Another long standing open problem is
the generalization of pointwise convergence results, even such basic
theorems as the $L^1$-pointwise ergodic theorem and the
Shannon-McMillan-Breiman (SMB) Theorem for general amenable groups,
for related results see for example \cite{E(1974), Ki(1975),
OW(1992)}. In \cite{L} Lindenstrauss gave a satisfactory answer to
the question by proving the pointwise ergodic theorem for general
locally compact amenable groups along F\o lner sequences obeying
some restrictions (such sequences must exist for all amenable
groups) and obtaining a generalization of the SMB Theorem to all
countable discrete amenable groups (see also the survey \cite{We}
written by Weiss). Moreover, using the tools built in \cite{L}
Lindenstrauss also proved other pointwise results, for example
\cite{OW(1992)} and so on.

\medskip

Along with the development of the local entropy theory for
$\Z$-actions, a natural question arises: to what extends the theory
can be generalized to an amenable group action? In \cite{KL} Kerr
and Li studied the local entropy theory of an amenable group action
for topological dynamics via independence. In this paper we try to
study systematically the local properties of entropy for actions of
a countable discrete amenable group both in topological and measure
theoretical settings.

\medskip

First, we shall establish a local variational principle for a given
finite open cover of a countable discrete amenable group action.
Note that the classical variational principle of a countable
discrete amenable group action (see \cite{OP2, ST}) can be deduced
from our result by proceeding some simple arguments. In the way to
build the local variational principle, we also introduce two kinds
of measure-theoretic entropy for finite Borel covers following the
ideas of \cite{R}, prove the upper semi-continuity (u.s.c.) of them
(when considering a finite open cover) on the set of invariant
measures, and show that they are coincide. We note that completely
different from the case of $\Z$-actions, in our proving of the
u.s.c. we need a deep convergence lemma related to a countable
discrete amenable group; and in our proving of the equivalence of
these two kinds of entropy, we need the result that they are
equivalent for $\Z$-actions, and Danilenko's orbital approach method
(since we can't obtain a universal Rohlin Lemma and a result similar
to Glasner-Weiss Theorem \cite{GW4} in this setting).
Meanwhile, inspired by \cite[Lemma 5.11]{We} we shall give a local
version of the well-known Katok's result \cite[Theorem
I.I]{Kat(1980.b)} for a countable discrete amenable group action.

Then we introduce entropy tuples in both topological and
measure-theoretic settings.
The set of measure-theoretic entropy tuples for an invariant measure
is characterized, the variational relation between these two kinds
of entropy tuples is obtained as an application of the local
variational principle for a given finite open cover. Based on the
ideas of topological entropy pairs, we discuss two classes of
dynamical systems: having u.p.e. and having c.p.e. Precisely
speaking, for a countable discrete amenable group action, it is
proved: u.p.e. and c.p.e. are both preserved under a finite
production; u.p.e. implies c.p.e.; c.p.e. implies the existence of
an invariant measure with full support; u.p.e. implies mild mixing;
and minimal topological $K$ implies strong mixing if the group
considered is commutative.

We note that when we finished our writing of the paper, we received
a preprint by Kerr and Li \cite{KL1}, where the authors investigated
the local entropy theory of an amenable group action for
measure-preserving systems via independence. They obtained the
variational relation between these two kinds of entropy tuples
defined by them, and stated the local variational principle for a
given finite open cover as an open question, see \cite[Question
2.10]{KL1}. Moreover, the results obtained in this paper have been
applied to consider the co-induction of dynamical systems in
\cite{DZ}.

The paper is organized as following. In section 2, we introduce the
terminology from \cite{OW, WZ} that we shall use, and obtain some
convergence lemmas which play key roles in the following sections.
In section 3, for a countable discrete amenable group action we
introduce the entropy theory of it, including two kinds of
measure-theoretic entropy for a finite Borel cover, and establish some
basic properties of them, such as u.s.c., affinity and so on. Then
in section 4 we prove the equivalence of those two kinds of entropy
introduced for a finite Borel cover, and give a local version of the
well-known Katok's result \cite[Theorem I.I]{Kat(1980.b)} for a
countable discrete amenable group action. In section 5, we aim to
establish the local variational principle for a finite open cover. In
section 6, we introduce entropy tuples in both topological and
measure-theoretic settings and establish the variational relation
between them. Based on the ideas of
topological entropy pairs, in section 7 we discuss two special
classes of dynamical systems: having u.p.e. and having c.p.e., respectively.

\section{Backgrounds of a countable discrete amenable group}

Let $G$ be a countable discrete infinite group and $F (G)$ the
set of all finite non-empty subsets of $G$. $G$ is called {\it amenable}, if
for each $K\in F (G)$ and $\delta>0$ there exists $F\in F (G)$
such that
\begin{equation*}
\frac{|F\Delta KF|}{|F|}<\delta,
\end{equation*}
where $|\cdot|$ is the counting measure, $K F= \{k f: k\in K, f\in
F\}$ and $F\Delta KF= (F\setminus K F)\cup (K F\setminus F)$.
Let $K\in F (G)$ and $\delta>0$. Set $K^{- 1}= \{k^{- 1}: k\in K\}$.
$A\in F (G)$ is {\it $(K,\delta)$-invariant} if
\begin{equation*}
\frac{|B (A, K)|}{|A|}<\delta,
\end{equation*}
where $B (A, K)\doteq \{g\in G: Kg\cap A\neq \emptyset \text{ and }
Kg\cap (G\setminus A)\neq \emptyset\}= K^{- 1} A\cap K^{- 1}
(G\setminus A)$. A sequence $\{ F_n\}_{n\in \mathbb{N}}\subseteq F
(G)$ is called a {\it F\o lner sequence}, if for each $K\in F (G)$
and $\delta>0$, $F_n$ is $(K,\delta)$-invariant when $n$ is large
enough. It is not hard to obtain the following asymptotic invariance
property that $G$ is amenable if and only if  $G$ has a F\o lner
sequence $\{ F_n \}_{n\in \mathbb{N}}$. For example, for $\Z$ we may
take F\o lner sequence $F_n=\{ 0,1,\cdots,n-1\}$, or for that matter
$\{ a_n,a_n+1,\cdots,a_n+n- 1 \}$ for any sequence $\{a_n\}_{n\in
\mathbb{N}}\subseteq \Z$.

Throughout the paper, any amenable group considered is assumed to be
a countable discrete amenable infinite group, and $G$ will always be
such a group with the unit $e_G$.

\subsection{Quasi-tiling for an amenable group}

The following terminology and results are due to Ornstein and Weiss
\cite{OW} (see also \cite{RW, WZ}). Let $\{A_1, \cdots,
A_k\}\subseteq F (G)$ and $\epsilon\in (0, 1)$. Subsets $A_1,
\cdots, A_k$ are {\it $\epsilon$-disjoint} if there are $\{B_1,
\cdots, B_k\}\subseteq F (G)$ such that
\begin{enumerate}

\item $B_i\subseteq A_i$ and $\frac{|B_i|}{|A_i|}>1-\epsilon$ for $i=1, \cdots,
k$,

\item $B_i\cap B_j= \emptyset$ if $1\le i\neq j\le k$.
\end{enumerate}
For $\alpha\in (0,1]$, we say that $\{A_1, \cdots, A_k\}$ {\it
$\alpha$-covers} $A\in F (G)$ if
$$\frac{|A\cap (\bigcup_{i=1}^k A_i)|}{|A|}\ge \alpha.$$
For $\delta\in [0,1)$, $\{ A_1, \cdots,A_k\}$ is called a {\it
$\delta$-even cover} of $A\in F (G)$ if
\begin{enumerate}

\item $A_i\subseteq A$ for $i=1,\cdots,k$,

\item there is $M\in \N$ such that $\sum_{i=1}^k 1_{A_i}(g)\le M$
for each $g\in G$ and $\sum_{i=1}^k |A_i|\ge (1-\delta) M |A|$.
\end{enumerate}
We say that $A_1, \cdots, A_k$ {\it $\epsilon$-quasi-tile} $A\in F
(G)$ if there exists $\{C_1, \cdots, C_k\}\subseteq F (G)$ such that
\begin{enumerate}

\item for $i= 1, \cdots, k$, $A_i C_i\subseteq A$ and $\{A_i c: c\in C_i\}$
forms an $\epsilon$-disjoint family,

\item $A_i C_i\cap A_j C_j= \emptyset$ if $1\le i\neq j\le k$,

\item $\{A_i C_i: i= 1, \cdots, k\}$ forms a $(1- \epsilon)$-cover
of $A$.
\end{enumerate}
The subsets $C_1, \cdots, C_k$ are called the {\it tiling centers}.

The following lemmas are proved in \cite[\S1.2]{OW}.

\begin{lem} \label{ow-lemma1}
Let $\delta\in [0, 1), e_G\in S\in F (G)$ and $A\in F (G)$ satisfy
that $A$ is $(S S^{- 1}, \delta)$-invariant. Then the right
translates of $S$ that lie in $A$, $\{S g: g\in G, S g\subseteq
A\}$, form a $\delta$-even cover of $A$.
\end{lem}

\begin{lem} \label{ow-lemma2}
Let $\delta\in [0, 1)$ and $\mathcal{A}\subseteq F (G)$ a
$\delta$-even cover of $A\in F (G)$. Then for each $\epsilon\in (0,
1)$ there is an $\epsilon$-disjoint sub-collection of $\mathcal{A}$
which $\epsilon (1- \delta)$-covers $A$.
\end{lem}

Then we can claim the following proposition (see \cite{OW} or
\cite[Theorem 2.6]{WZ}).

\begin{prop} \label{ow-prop}
Let $\{ F_n\}_{n\in \mathbb{N}}$ with $e_G\in F_1\subseteq
F_2\subseteq \cdots$ and $\{F_n'\}_{n\in \mathbb{N}}$ be two F\o
lner sequences of $G$. Then for any $\epsilon\in (0, \frac{1}{4})$
and $N\in \mathbb{N}$, there exist integers $n_1, \cdots, n_k$
with $N\le n_1< \cdots< n_k$ such that $F_{n_1}, \cdots, F_{n_k}$
$\epsilon$-quasi-tile $F_m'$ when $m$ is large enough.
\end{prop}
\begin{proof} We follow the arguments in the proof of \cite[Theorem 2.6]{WZ}. Fix
$\epsilon\in (0, \frac{1}{4})$ and $N\in \mathbb{N}$.

Let $k\in \mathbb{N}$ and $\delta> 0$ such that $(1-
\frac{\epsilon}{2})^k< \epsilon$ and $6^k \delta<
\frac{\epsilon}{2}$. We can choose integers $n_1, \cdots, n_k$ with
$N\le n_1< \cdots< n_k$ such that $F_{n_{i+ 1}}$ is $(F_{n_i}
F_{n_i}^{- 1}, \delta)$-invariant and $\frac{|F_{n_i}|}{|F_{n_{i+
1}}|}< \delta$, $i=1, \cdots, k- 1$.

Now for each enough large $m$, $F_m'$ is $(F_{n_k} F_{n_k}^{- 1},
\delta)$-invariant and $\frac{|F_{n_k}|}{|F_m'|}< \delta$, thus by
Lemma \ref{ow-lemma1} the right translates of $F_{n_k}$ that lie
in $F_m'$ form a $\delta$-even cover of $F_m'$, and so by Lemma
\ref{ow-lemma2} there exists $C_k\in F (G)$ such that $F_{n_k}
C_k\subseteq F_m'$ and the family $\{F_{n_k} c: c\in C_k\}$ is
$\epsilon$-disjoint and $\epsilon (1- \delta)$-covers $F_m'$. Let
$c_k\in C_k$. Without loss of generality assume
that $|F_{n_k} C_k\setminus F_{n_k} c_k|< \epsilon (1- \delta)
|F_m'|$ (if necessity we may take a subset of $C_k$ to replace
with $C_k$). Then $(1- \epsilon) |F_{n_k}| |C_k|< |F_m'|$ and
\begin{equation} \label{estimate 1}
1- \epsilon (1- \delta)\ge \frac{|F_m'\setminus F_{n_k}
C_k|}{|F_m'|}= 1-\frac{|F_{n_k} C_k\setminus
F_{n_k}c_k|+|F_{n_k}c_k|}{|F_m'|}\ge 1-\epsilon(1-\delta)-\delta.
\end{equation}
Set $A_{k- 1}= F_m'\setminus F_{n_k} C_k$, $K_{k- 1}= F_{n_{k- 1}}
F_{n_{k- 1}}^{- 1}$. We have
\begin{eqnarray*}
B (A_{k- 1}, K_{k- 1})& =& K_{k- 1}^{- 1} (F_m'\setminus F_{n_k}
C_k)\cap K_{k- 1}^{- 1}
((G\setminus F_m')\cup F_{n_k} C_k) \\
&\subseteq & B (F_m', K_{k- 1})\cup \bigcup_{c\in C_k} B (F_{n_k} c,
K_{k- 1}) \\
&\subseteq & B (F_m', F_{n_k} F_{n_k}^{- 1})\cup \bigcup_{c\in C_k}
B (F_{n_k}, K_{k- 1}) c\ (\text{as}\ K_{k- 1}\subseteq F_{n_k}
F_{n_k}^{- 1}),
\end{eqnarray*}
which implies
\begin{eqnarray*}
\frac{|B (A_{k- 1}, K_{k- 1})|}{|A_{k- 1}|}
&\le & \frac{|B (F_m',
F_{n_k} F_{n_k}^{- 1})|}{|A_{k- 1}|}+ |C_k| \frac{|B (F_{n_k},
K_{k- 1})|}{|A_{k- 1}|} \\
&< & \frac{\delta}{|F_m'\setminus F_{n_k} C_k|} (|F_m'|+ |C_k|
|F_{n_k}|) \\
&< & \delta \left(1+ \frac{1}{1- \epsilon}\right)\frac{|F_m'|
}{|F_m'\setminus F_{n_k} C_k|}\ (\text{as}\ (1- \epsilon) |F_{n_k}|
|C_k|<
|F_m'|) \\
&\le & \delta \left(1+\frac{1}{1-\epsilon}\right) \frac{1}{1-
\epsilon (1-
\delta)-\delta}\ \ \ \ (\text{by \eqref{estimate 1}})\\
&<& 6 \delta\ \left(\text{as}\ \epsilon\in \left(0,
\frac{1}{4}\right)\right).
\end{eqnarray*}
That is, $A_{k-1}$ is $(F_{n_{k- 1}} F_{n_{k- 1}}^{- 1}, 6
\delta)$-invariant. Moreover, using \eqref{estimate 1} one has
\begin{equation*}
\frac{|F_{n_{k- 1}}|}{|A_{k-1}|}= \frac{|F_{n_{k-
1}}|}{|F_{n_k}|}\cdot \frac{|F_{n_k}|}{|F_m'|}\cdot
\frac{|F_m'|}{|F_m'\setminus F_{n_k} C_k|}< \frac{\delta^2}{1-
\epsilon (1- \delta)-\delta}<\delta.
\end{equation*}

By the same reasoning there exists $C_{k- 1}\in F (G)$ such that
$F_{n_{k- 1}} C_{k- 1}\subseteq A_{k-1}$, the family $\{F_{n_{k- 1}}
c: c\in C_{k- 1}\}$ is $\epsilon$-disjoint and $\epsilon (1- 6
\delta)$-covers $A_{k-1}$ and
\begin{equation} \label{07-03-05-01}
1-\epsilon (1- 6 \delta)\ge \frac{|A_{k-1}\setminus F_{n_{k- 1}}
C_{k- 1}|}{|A_{k-1}|}\ge 1-\epsilon (1- 6 \delta)-6\delta.
\end{equation}
Moreover, by \eqref{estimate 1} and \eqref{07-03-05-01} we have
\begin{eqnarray*}
\frac{|A_{k-1}\setminus F_{n_{k- 1}} C_{k- 1}|}{|F_m'|}&= &
\frac{|A_{k-1}\setminus F_{n_{k- 1}} C_{k- 1}|}{|A_{k-1}|}\cdot
\frac{|F_m'\setminus
F_{n_k} C_k|}{|F_m'|} \\
&\le& (1- \epsilon (1- 6 \delta)) (1- \epsilon (1-\delta))< \left(1-
\frac{\epsilon}{2}\right)^2.
\end{eqnarray*}

Inductively, we get $\{C_k, \cdots, C_1\}\subseteq F (G)$ such that
if $1\le i\neq j\le k$ then $F_{n_i} C_i\cap F_{n_j} C_j=
\emptyset$, and if $i= 1, \cdots, k$ then $F_{n_i} C_i\subseteq
F_m'$ and the family $\{F_{n_i} c: c\in C_i\}$ is
$\epsilon$-disjoint. Moreover,
\begin{equation*}
\frac{|F_m'\setminus \bigcup_{i= 1}^k F_{n_i} C_i|}{|F_m'|}<
\left(1- \frac{\epsilon}{2}\right)^k< \epsilon.
\end{equation*}
Thus, $\{F_{n_i} C_i: i= 1, \cdots, k\}$ forms a $(1-
\epsilon)$-cover of $F_m'$. This ends the proof.
\end{proof}

\subsection{Convergence key lemmas}

Let $f: F (G)\rightarrow \R$ be a function. We say that f is
\begin{enumerate}

\item
{\it monotone}, if $f(E)\le f(F)$ for any $E, F\in F(G)$ satisfying
$E\subseteq F$;

\item {\it non-negative}, if $f(F)\ge 0$ for any $F\in F(G)$;

\item
{\it $G$-invariant}, if $f (F g)= f (F)$ for any $F\in F(G)$ and
$g\in G$;

\item
{\it sub-additive}, if $f(E\cup F)\le f(E)+f(F)$ for any $E, F\in
F(G)$.
\end{enumerate}

The following lemma is proved in \cite[Theorem 6.1]{LW}.
\begin{lem} \label{convergent}
Let $f: F(G)\rightarrow \R$ be a monotone non-negative
$G$-invariant sub-additive (m.n.i.s.a.) function. Then
for any F\o lner sequence $\{F_n\}_{n\in \mathbb{N}}$ of $G$, the
sequence $\{\frac{f(F_n)}{|F_n|}\}_{n\in \N}$ converges and the
value of the limit is independent of the selection of the F\o lner
sequence $\{F_n\}_{n\in \mathbb{N}}$.
\end{lem}
\begin{proof} We give a proof for the completion. Since $f$ is $G$-invariant, there exists $M\in \R_+$ such that $f
(\{g\})= M$ for all $g\in G$.

Now first we claim that if $\{ F_n\}_{n\in \mathbb{N}}$ with
$e_G\in F_1\subseteq F_2\subseteq \cdots$ and $\{F_n'\}_{n\in
\mathbb{N}}$ are two F\o lner sequences of $G$ then
\begin{eqnarray} \label{key-convergence}
\limsup_{n\rightarrow +\infty} \frac{f (F_n')}{|F_n'|}\le
\limsup_{n\rightarrow +\infty} \frac{f (F_n)}{|F_n|}.
\end{eqnarray}
Let $\epsilon\in (0, \frac{1}{4})$ and $N\in \mathbb{N}$. By
Proposition \ref{ow-prop} there exist integers $n_1, \cdots, n_k$
with $N\le n_1< \cdots< n_k$ such that when $n$ is large enough
then $F_{n_1}, \cdots, F_{n_k}$ $\epsilon$-quasi-tile $F_n'$ with
tiling centers $C_1^n, \cdots, C_k^n$. Thus, when $n$ is large
enough then
\begin{equation} \label{condition}
F_n'\supseteq \bigcup_{i= 1}^k F_{n_i} C_i^n\ \text{and}\
|\bigcup_{i= 1}^k F_{n_i} C_i^n|\ge \max \{(1- \epsilon) |F_n'|, (1-
\epsilon) \sum_{i= 1}^k |C_i^n|\cdot |F_{n_i}|\},
\end{equation}
which implies
\begin{eqnarray} \label{07-03-05-02}
\frac{f (F_n')}{|F_n'|} &\le & \frac{f (F_n'\setminus \bigcup_{i=
1}^k F_{n_i} C_i^n)+ f (\bigcup_{i= 1}^k F_{n_i} C_i^n)}{|F_n'|} \nonumber \\
&\le & M \frac{|F_n'\setminus \bigcup_{i= 1}^k F_{n_i}
C_i^n|}{|F_n'|}+ \frac{f (\bigcup_{i= 1}^k F_{n_i}
C_i^n)}{|\bigcup_{i= 1}^k F_{n_i} C_i^n|} \nonumber\\
&\le & M \epsilon+ \frac{f (\bigcup_{i= 1}^k F_{n_i}
C_i^n)}{|\bigcup_{i= 1}^k F_{n_i} C_i^n|} \nonumber \\
&\le & M \epsilon+ \sum_{i= 1}^k \frac{|C_i^n| f (F_{n_i})}{(1-
\epsilon) \sum_{i= 1}^k
|C_i^n|\cdot |F_{n_i}|}\ (\text{using}\ \eqref{condition}) \nonumber \\
&\le & M \epsilon+ \frac{1}{1- \epsilon}\max_{1\le i\le k} \frac{f
(F_{n_i})}{|F_{n_i}|}\le M \epsilon+ \frac{1}{1- \epsilon}\sup_{m\ge
N} \frac{f (F_m)}{|F_m|}.
\end{eqnarray}
Now letting $\epsilon\rightarrow 0+$ and $N\rightarrow +\infty$, we
conclude the inequality \eqref{key-convergence}.

Now let $\{ H_n\}_{n\in \mathbb{N}}$ with $e_G\in H_1\subseteq
H_2\subseteq \cdots$ be a F\o lner sequence of $G$. Clearly, there
is a sub-sequence $\{H_{n_m}\}_{m\in \mathbb{N}}$ of $\{
H_n\}_{n\in \mathbb{N}}$ such that
\begin{eqnarray} \label{liminf}
\lim_{m\rightarrow +\infty} \frac{f(H_{n_m})}{|H_{n_m}|}=
\liminf_{n\rightarrow +\infty} \frac{f(H_n)}{|H_n|}.
\end{eqnarray}
Applying the above claim to F\o lner sequences $\{ H_{n_m}\}_{m\in
\mathbb{N}}$ and $\{H_n\}_{n\in \mathbb{N}}$ (see
\eqref{key-convergence}), we obtain
\begin{eqnarray*}
\limsup_{n\rightarrow +\infty} \frac{f (H_n)}{|H_n|}\le
\limsup_{m\rightarrow +\infty} \frac{f (H_{n_m})}{|H_{n_m}|}=
\liminf_{n\rightarrow +\infty} \frac{f(H_n)}{|H_n|}\ (\text{by
\eqref{liminf}}).
\end{eqnarray*}
Thus, the sequence $\{\frac{f (H_n)}{|H_n|}\}_{n\in \N}$ converges
(say $N(f)$ to be the value of the limit). Then for any F\o lner
sequence $\{ F_n\}_{n\in \mathbb{N}}$ with $e_G\in F_1\subseteq
F_2\subseteq \cdots$ of $G$, the sequence $\{\frac{f
(F_n)}{|F_n|}\}_{n\in \N}$ converges to $N(f)$ (by
\eqref{key-convergence}).

Finally, in order to complete the proof, we only need to check
that, for any given F\o lner sequence $\{F_n\}_{n\in \mathbb{N}}$
of $G$, if $\{F_n'\}_{n\in \mathbb{N}}$ is any sub-sequence of
$\{F_n\}_{n\in \mathbb{N}}$ such that the sequence $\{\frac{f
(F_n')}{|F_n'|}\}_{n\in \N}$ converges, then it converges to $N
(f)$, which implies $\{\frac{f (F_n)}{|F_n|}\}_{n\in \N}$
converges to $N (f)$. Let $\{F_n'\}_{n\in \mathbb{N}}$ be such a
sub-sequence. With no loss of generality we assume $\lim\limits_{n\rightarrow +\infty}
\frac{|F_n^*|}{|F_{n+ 1}'|}=0$ (if necessity we take a
sub-sequence of $\{ F_n'\}_{n\in \mathbb{N}}$), where $F_n^*=
\{e_G\}\cup \bigcup_{i= 1}^n F_i'$ for each $n$. It is easy to
check that $e_G\in F_1^*\subseteq F_2^*\subseteq \cdots$ forms a
F\o lner sequence of $G$ and so the sequence
$\{\frac{f(F_{n}^*)}{|F_{n}^*|}\}_{n\in \N}$ converges to $N(f)$
from the above discussion. Note that, for each $n\in \N$,
\begin{eqnarray*}
|\frac{f(F_{n+ 1}^*)}{|F_{n+ 1}^*|}- \frac{f(F_{n+ 1}')}{|F_{n+
1}'|}|&\le & \frac{f(F_n^*)}{|F_{n+ 1}^*|}+ |\frac{f(F_{n+
1}')}{|F_{n+ 1}^*|}- \frac{f(F_{n+ 1}')}{|F_{n+
1}'|}| \\
&\le & M \left(\frac{|F_n^*|}{|F_{n+1}^*|}+|F_{n+1}'|\cdot|\frac{1}
{|F_{n+1}^*|}-\frac{1}{|F_{n+1}'|}|\right) \\
&\le & M \left(\frac{|F_n^*|}{|F_{n+
1}'|}+1-\frac{1}{1+\frac{|F_n^*|}{|F_{n+ 1}'|}}\right).
\end{eqnarray*}
By letting $n\rightarrow +\infty$ one has $\lim_{n\rightarrow
+\infty} \frac{f(F_{n+ 1}^*)}{|F_{n+ 1}^*|}=\lim_{n\rightarrow
+\infty} \frac{f(F_{n+ 1}')}{|F_{n+ 1}'|}=N(f)$, that is, the
sequence $\{\frac{f(F_{n}')}{|F_{n}'|}\}_{n\in \N}$ converges also
to $N(f)$.
\end{proof}

\begin{rem}
Recall that we say a set $T$ tiles $G$ if there is a subset $C$
such that $\{T c: c\in C\}$ is a partition of $G$. It's proved
that if $G$ admits a F\o lner sequence $\{F_n\}_{n\in \N}$ of
tiling sets then for each $f$ as in Lemma \ref{convergent} the
sequence $\{\frac{f (F_n)}{|F_n|}\}_{n\in \N}$ converges to
$\inf_{n\in \N} \frac{f (F_n)}{|F_n|}$ and the value of the limit
is independent of the choice of such a F\o lner sequence, which is
stated as \cite[Theorem 5.9]{We}.
\end{rem}

The following useful lemma is an alternative version of
\eqref{07-03-05-02} in the proof of Lemma \ref{convergent}.

\begin{lem} \label{conv-lem}
Let $e_G\in F_1\subseteq F_2\subseteq \cdots$ be a F\o lner sequence
of $G$. Then for any $\epsilon\in (0, \frac{1}{4})$ and $N\in
\mathbb{N}$ there exist integers $n_1, \cdots, n_k$ with $N\le n_1<
\cdots< n_k$ such that if $f: F (G)\rightarrow \R$ a m.n.i.s.a.
function with $M = f (\{g\})$ for all $g\in G$ then
\begin{eqnarray*}
\lim_{n\rightarrow +\infty} \frac{f (F_n)}{|F_n|}\le M \epsilon+
\frac{1}{1- \epsilon}\max_{1\le i\le k} \frac{f
(F_{n_i})}{|F_{n_i}|}\le M \epsilon \left(1+ \frac{1}{1-
\epsilon}\right)+ \max_{1\le i\le k} \frac{f (F_{n_i})}{|F_{n_i}|}.
\end{eqnarray*}
\end{lem}

\section{Entropy of an amenable group action}

Let $\{ F_n\}_{n\in \mathbb{N}}$ be a F\o lner sequence of $G$ and
fix it in the section. In this section, we aim to introduce the
entropy theory of a $G$-system. By a {\it $G$-system} $(X, G)$ we
mean that $X$ is a compact metric space and $\Gamma: G\times X\rightarrow X, (g, x)\mapsto g x$ is a continuous mapping
 satisfying
\begin{enumerate}

\item $\Gamma (e_G, x)= x$ for each $x\in X$,

\item $\Gamma (g_1, \Gamma (g_2, x))= \Gamma (g_1 g_2, x)$ for each
$g_1, g_2\in G$ and $x\in X$.
\end{enumerate}
Moreover, if a non-empty compact subset $W\subseteq X$ is $G$-invariant (i.e. $g W= W$ for any $g\in G$) then $(W, G)$ is called a {\it sub-$G$-system} of $(X, G)$.

From now on, we let $(X, G)$ always be a $G$-system if there is no
any special statement. Denote by $\mathcal{B}_X$ the collection of
all Borel subsets of $X$. A cover of $X$ is a finite family of
Borel subsets of $X$, whose union is $X$. A partition of $X$ is a
cover of $X$ whose elements are pairwise disjoint. Denote by
$\mathcal{C}_X$ (resp. $\mathcal{C}_X^{o}$) the set of all covers
(resp. finite open covers) of $X$. Denote by $\mathcal{P}_X$ the
set of all partitions of $X$. Given two covers $\mathcal{U},
\mathcal{V}\in \mathcal{C}_X$, $\mathcal{U}$ is said to be {\it
finer} than $\mathcal{V}$ (denoted by $\mathcal{U}\succeq
\mathcal{V}$ or $\mathcal{V}\preceq \mathcal{U}$) if each element
of $\mathcal{U}$ is contained in some element of $\mathcal{V}$;
set $\mathcal{U}\vee \mathcal{V}= \{U\cap V: U\in \mathcal{U},
V\in \mathcal{V}\}$.

\subsection{Topological entropy}

Let $\mathcal{U}\in \mathcal{C}_X$. Set $N (\mathcal{U})$ to be
the minimum among the cardinalities of all sub-families of
$\mathcal{U}$ covering $X$ and denote by $\# (\mathcal{U})$ the cardinality of $\mathcal{U}$. Define $H (\mathcal{U})= \log N
(\mathcal{U})$. Clearly, if $\mathcal{U}, \mathcal{V}\in
\mathcal{C}_X$, then $H(\mathcal{U}\vee
\mathcal{V})\le H(\mathcal{U})+ H(\mathcal{V})$ and
$H(\mathcal{V})\ge H(\mathcal{U})$ when $\mathcal{V}\succeq
\mathcal{U}$.

Let $F\in F (G)$ and $\mathcal{U}\in \mathcal{C}_X$, set $\mathcal{U}_F= \bigvee_{g\in F} g^{-1}
\mathcal{U}$ (letting $\mathcal{U}_\emptyset= \{X\}$). It is not
hard to check that $F\in F(G)\mapsto H(\mathcal{U}_F)$ is a
m.n.i.a.s. function, and so by Lemma \ref{convergent}, the quantity
\begin{equation*}
h_{\text{top}}(G,\mathcal{U})\doteq \lim_{n\rightarrow
+\infty}\frac{1}{|F_n|}H(\mathcal{U}_{F_n})
\end{equation*}
exists and $h_{\text{top}}(G,\mathcal{U})$ is independent of the
choice of $\{ F_n\}_{n\in \mathbb{N}}$.
$h_{\text{top}}(G,\mathcal{U})$ is called the {\it topological
entropy of $\mathcal{U}$}. It is clear that
$h_{\text{top}}(G,\mathcal{U})\le H (\mathcal{U})$. Note that if
$\mathcal{U}_1, \mathcal{U}_2\in \mathcal{C}_X$, then
$h_{\text{top}}(G, \mathcal{U}_1\vee \mathcal{U}_2)\le
h_{\text{top}}(G, \mathcal{U}_1)+ h_{\text{top}}(G,
\mathcal{U}_2)$ and $h_{\text{top}}(G, \mathcal{U}_2)\ge
h_{\text{top}}(G, \mathcal{U}_1)$ when $\mathcal{U}_2\succeq
\mathcal{U}_1$. The {\it topological entropy of $(X,G)$} is
defined by
\begin{equation*}
h_{\text{top}} (G, X)=\sup_{\mathcal{U}\in \mathcal{C}_X^o}
h_{\text{top}}(G,\mathcal{U}).
\end{equation*}

\subsection{Measure-theoretic entropy}

Denote by $\mathcal{M}(X)$ the set of all Borel probability measures
on $X$. For $\mu\in \mathcal{M} (X)$, denote by $\text{supp} (\mu)$ the {\it support} of $\mu$, i.e. the smallest closed subset $W\subseteq X$ such that $\mu (W)= 1$. $\mu\in \mathcal{M} (X)$ is called {\it $G$-invariant} if $g
\mu= \mu$ for each $g\in G$; $G$-invariant $\nu\in \mathcal{M} (X)$
is called {\it ergodic} if $\nu (\bigcup_{g\in G} g A)= 0$ or $1$
for any $A\in \mathcal{B}_X$. Denote by $\mathcal{M} (X, G)$ (resp.
$\mathcal{M}^e(X,G)$) the set of all $G$-invariant (resp. ergodic
$G$-invariant) elements in $\mathcal{M} (X)$. Note that the
amenability of $G$ ensures that $\emptyset\neq \mathcal{M}^e (X, G)$
and both $\mathcal{M}(X)$ and $\mathcal{M}(X, G)$ are convex compact
metric spaces when they are endowed with the weak$^*$-topology.

Given $\alpha\in \mathcal{P}_X$, $\mu\in \mathcal{M}(X)$ and
 a sub-$\sigma$-algebra $\mathcal{A}\subseteq \mathcal{B}_X$,
 define
\begin{equation*}
H_{\mu}
(\alpha | \mathcal{A})= \sum_{A\in \alpha} \int_X
 -\E (1_A| \mathcal{A}) \log \E (1_A| \mathcal{A}) d \mu,
\end{equation*}
where $\E (1_A| \mathcal{A})$ is the expectation of $1_A$ with
respect to (w.r.t.) $\mathcal{A}$. One standard fact is
that $H_{\mu} (\alpha | \mathcal{A})$ increases w.r.t. $\alpha$
and decreases w.r.t. $\mathcal{A}$. Set $\mathcal{N}= \{\emptyset,
X\}$. Define
\begin{equation*}
H_\mu (\alpha)= H_\mu (\alpha| \mathcal{N})= \sum_{A\in \alpha}
-\mu(A) \log \mu(A).
\end{equation*}
Let $\beta \in \mathcal{P}_X$. Note that $\beta$ generates
naturally a sub-$\sigma$-algebra $\mathcal{F} (\beta)$ of
$\mathcal{B}_X$, define
\begin{equation*}
H_{\mu}(\alpha|\beta)= H_\mu (\alpha| \mathcal{F} (\beta))=
H_{\mu}(\alpha\vee \beta)- H_{\mu}(\beta).
\end{equation*}
Now let $\mu\in \mathcal{M}(X,G)$, it is not hard to see that $F\in
F(G)\mapsto H_{\mu} (\alpha_F)$ is a m.n.i.a.s. function. Thus by
Lemma \ref{convergent} we can define the {\it measure-theoretic
$\mu$-entropy of $\alpha$} as
\begin{equation}
\label{eq-p-e} h_\mu(G,\alpha)=\lim_{n\rightarrow +\infty}
\frac{1}{|F_n|}H_{\mu}(\alpha_{F_n}) \left(=\inf_{F\in F(G)}
\frac{1}{|F|}H_{\mu}(\alpha_F)\right),
\end{equation}
where the last identity is to be proved in Lemma
\ref{lem-436}~(4). In particular, $h_\mu(G,\alpha)$ is independent
of the choice of F\o lner sequence $\{ F_n\}_{n\in \mathbb{N}}$.
The {\it measure-theoretic $\mu$-entropy of $(X, G)$} is defined
by
\begin{equation} \label{hmu}
h_\mu(G, X)= \sup_{\alpha \in \mathcal{P}_X} h_\mu(G,\alpha).
\end{equation}

\subsubsection{The proof of the second identity in \eqref{eq-p-e}}

\begin{lem}\label{lem-436}
Let $\alpha\in \mathcal{P}_X$, $\mu\in \mathcal{M}(X)$, $m\in \N$
and $E, F, B, E_1, \cdots, E_k\in F (G)$. Then
\begin{description}

\item[1] $H_\mu (\alpha_{E\cup
F})+H_\mu(\alpha_{E\cap F})\le H_\mu(\alpha_E)+H_\mu(\alpha_F)$.

\item[2] If $1_E(g)=\frac{1}{m}\sum_{i=1}^k 1_{E_i}(g)$ holds
for each $g\in G$, then $H_\mu(\alpha_E)\le
\frac{1}{m}\sum_{i=1}^k H_\mu(\alpha_{E_i})$.

\item[3]
$$H_\mu(\alpha_F)\le \sum_{g\in
F}\frac{1}{|B|}H_\mu(\alpha_{Bg})+|F\setminus \{ g\in
G:B^{-1}g\subseteq F\}|\cdot \log \# (\alpha).$$

\item[4] If in addition $\mu\in \mathcal{M}(X,G)$, then
$h_\mu(G,\alpha)=\inf_{B\in F(G)}\frac{H_\mu(\alpha_B)}{|B|}$.
\end{description}
\end{lem}
\begin{proof}
1. The conclusion follows directly from the following simple
observation:
\begin{eqnarray*}
H_\mu(\alpha_{E\cup F})+ H_\mu(\alpha_{E\cap F})& =
&H_\mu(\alpha_{E})+H_\mu(\alpha_{F}|\alpha_E)+H_\mu(\alpha_{E\cap
F}) \\
&\le & H_\mu(\alpha_{E})+H_\mu(\alpha_{F}|\alpha_{E\cap
F})+H_\mu(\alpha_{E\cap F})\\
&= & H_\mu(\alpha_E)+H_\mu(\alpha_F).
\end{eqnarray*}

2. Clearly, $\bigcup_{i= 1}^k E_i= E$. Say $\{A_1, \cdots, A_n\}=
\bigvee_{i= 1}^k \{E_i, E\setminus E_i\}$ (neglecting all empty
elements). Set $K_0= \emptyset$, $K_i= \bigcup_{j= 1}^i A_j$, $i=
1, \cdots, n$. Then $\emptyset= K_0\subsetneq K_1\subsetneq \cdots
\subsetneq K_n= E$. Moreover, if for some $i= 1, \cdots, n$ and
$j= 1, \cdots, k$ with $E_j\cap (K_{i}\setminus K_{i- 1})\neq
\emptyset$ then $K_{i}\setminus K_{i- 1}\subseteq E_j$ and so
$K_i= K_{i- 1}\cup (K_i\cap E_j)$, thus $H_\mu (\alpha_{K_i})+
H_\mu (\alpha_{K_{i- 1}\cap E_j})\le H_\mu(\alpha_{K_{i- 1}})+
H_\mu(\alpha_{K_i\cap E_j})$ (using 1), i.e.
\begin{equation} \label{inequality}
H_\mu (\alpha_{K_i})- H_\mu(\alpha_{K_{i- 1}})\le
H_\mu(\alpha_{K_i\cap E_j})- H_\mu (\alpha_{K_{i- 1}\cap E_j}).
\end{equation}
Now for each $i= 1, \cdots, n$ we select $k_i\in K_{i}\setminus
K_{i- 1}$, one has
\begin{eqnarray*}
H_\mu(\alpha_E) &= & \sum_{i= 1}^n \left(\frac{1}{m} \sum_{j= 1}^k
1_{E_j}(k_i)\right) (H_\mu
(\alpha_{K_i})- H_\mu(\alpha_{K_{i- 1}}))\ (\text{by assumptions}) \\
&= & \frac{1}{m} \sum_{j= 1}^k \sum_{1\le i\le n: k_i\in E_j}
(H_\mu
(\alpha_{K_i})- H_\mu(\alpha_{K_{i- 1}})) \\
&\le & \frac{1}{m} \sum_{j= 1}^k \sum_{1\le i\le n: k_i\in E_j}
(H_\mu(\alpha_{K_i\cap E_j})- H_\mu (\alpha_{K_{i- 1}\cap
E_j})) \ (\text{using}\ \eqref{inequality}) \\
&\le & \frac{1}{m} \sum_{j= 1}^k \sum_{i= 1}^n
(H_\mu(\alpha_{K_i\cap E_j})- H_\mu (\alpha_{K_{i- 1}\cap E_j}))=
\frac{1}{m} \sum_{j=1}^k H_\mu(\alpha_{E_j}).
\end{eqnarray*}

3. Note that $1_{\{ h\in BF: B^{-1}h\subseteq F\}}(f)= \frac{1}{|B|}
\sum_{g\in F} 1_{\{h\in B g: B^{-1} h\subseteq F\}}(f)$ for each
$f\in G$. By 2, one has
\begin{align}\label{3-eq-key}
H_\mu(\alpha_{\{ h\in BF:B^{-1}h\subseteq F\}})\le \frac{1}{|B|}
\sum_{g\in F} H_\mu(\alpha_{\{ h\in Bg:B^{-1}h\subseteq F\}})\le
\frac{1}{|B|} \sum_{g\in F} H_\mu(\alpha_{Bg}),
\end{align}
which implies
\begin{eqnarray*}
H_\mu(\alpha_F)&\le & H_\mu(\alpha_{\{ h\in BF:B^{-1}h\subseteq
F\}})+H_\mu(\alpha_{F\setminus \{ h\in BF:B^{-1}h\subseteq F\}})\\
&\le & \frac{1}{|B|} \sum_{g\in F} H_\mu(\alpha_{Bg})+|F\setminus \{
h\in BF:B^{-1}h\subseteq F\}|\cdot \log \#\alpha   \ \text{(using
\eqref{3-eq-key})}\\
&= & \frac{1}{|B|} \sum_{g\in F} H_\mu(\alpha_{Bg})+|F\setminus \{
h\in G:B^{-1}h\subseteq F\}|\cdot \log \#\alpha.
\end{eqnarray*}

4. If in addition $\mu$ is $G$-invariant, then by 3, for each
$n\in \mathbb{N}$ we have
\begin{eqnarray} \label{07-03-05-03}
\frac{1}{|F_n|}H_\mu(\alpha_{F_n})&\le &\frac{1}{|F_n|}\sum_{g\in
F_n}\frac{1}{|B|}H_\mu(\alpha_{Bg})+\frac{1}{|F_n|}|F_n\setminus
\{ g\in G:B^{-1}g\subseteq F_n\}|\cdot \log \# \alpha \nonumber \\
&= & \frac{1}{|F_n|}\sum_{g\in
F_n}\frac{1}{|B|}H_\mu(g^{-1}(\alpha_{B}))+\frac{1}{|F_n|}|F_n\setminus
\{ g\in G:B^{-1}g\subseteq F_n\}|\cdot \log \#\alpha \nonumber \\
&= & \frac{1}{|B|}H_\mu(\alpha_B)+\frac{1}{|F_n|}|F_n\setminus \{
g\in G:B^{-1}g\subseteq F_n\}|\cdot \log \#\alpha.
\end{eqnarray}
Set $B'= B^{- 1}\cup \{e_G\}$. Note that for each $\delta>0$, $F_n$
is $(B', \delta)$-invariant if $n$ is large enough and
\begin{eqnarray*}
F_n\setminus \{g\in G: B^{-1} g\subseteq F_n\}= F_n\cap B
(G\setminus F_n)\subseteq (B')^{- 1} F_n\cap (B')^{- 1} (G\setminus
F_n)= B (F_n, B'),
\end{eqnarray*}
letting $n\rightarrow +\infty$ we get
\begin{equation} \label{need}
\lim_{n\rightarrow +\infty} \frac{1}{|F_n|} |F_n\setminus \{g\in G:
B^{-1} g\subseteq F_n\}|= \lim_{n\rightarrow +\infty} \frac{|B (F_n,
B')|}{|F_n|}= 0,
\end{equation}
and so $h_\mu(G,\alpha)\le \frac{1}{|B|}H_\mu(\alpha_B)$ (using
\eqref{07-03-05-03} and \eqref{need}). Since $B$ is arbitrary, 4 is
proved.
\end{proof}

\begin{rem}
In \cite{O}, Lemma \ref{lem-436} (1) is called the strong
sub-additivity of entropy. In his treatment of entropy for amenable
group actions \cite[Chapter 4]{O}, Ollaginer used the property
rather heavily.
\end{rem}

\subsubsection{Measure-theoretic entropy for covers}

Following Romagnoli's ideas \cite{R}, we define a new notion that
extends definition \eqref{eq-p-e} to covers. Let $\mu\in
\mathcal{M} (X)$ and $\mathcal{A}\subseteq \mathcal{B}_X$ be a
sub-$\sigma$-algebra. For $\mathcal{U}\in \mathcal{C}_X$, we define
$$H_{\mu}(\mathcal{U}| \mathcal{A})= \inf_{\alpha \in \mathcal{P}_X: \alpha \succeq
\mathcal{U}} H_{\mu}(\alpha| \mathcal{A})\ \text{and}\ H_\mu
(\mathcal{U})= H_\mu (\mathcal{U}| \mathcal{N}).$$

Many properties of the function $H_{\mu}(\alpha)$ are extended to
$H_{\mu}(\mathcal{U})$ from partitions to covers.

\begin{lem} \label{lem-2-2}
Let $\mu \in \mathcal{M}(X)$, $\mathcal{A}\subseteq \mathcal{B}_X$
be a sub-$\sigma$-algebra, $g\in G$ and $\mathcal{U}_1,
\mathcal{U}_2\in \mathcal{C}_X$. Then
\begin{description}

\item[1] $0\le H_{\mu}(g^{ -1} \mathcal{U}_1| g^{- 1} \mathcal{A})= H_{g \mu}
(\mathcal{U}_1| \mathcal{A})\le H (\mathcal{U}_1)$.

\item[2] If $\mathcal{U}_1\succeq \mathcal{U}_2$, then
$H_{\mu}(\mathcal{U}_1| \mathcal{A})\ge H_{\mu}(\mathcal{U}_2|
\mathcal{A})$.

\item[3] $H_{\mu}(\mathcal{U}_1\vee \mathcal{U}_2| \mathcal{A})\le
H_{\mu}(\mathcal{U}_1| \mathcal{A})+ H_{\mu}(\mathcal{U}_2|
\mathcal{A})$.
\end{description}
\end{lem}

Using Lemma \ref{lem-2-2}, one gets easily that if $\mu\in
\mathcal{M}(X,G)$ then $F\in F(G)\mapsto H_{\mu}(\mathcal{U}_F)$ is
a m.n.i.s.a. function. So we may define the {\it measure-theoretic
$\mu^-$-entropy of $\mathcal{U}$} as
$$h_\mu^-(G,\mathcal{U})=\lim_{n\rightarrow +\infty}
\frac{1}{|F_n|}H_{\mu}(\mathcal{U}_{F_n})$$ and
$h_\mu^-(G,\mathcal{U})$ is independent of the choice of F\o lner
sequence $\{ F_n\}_{n\in \mathbb{N}}$ (see Lemma
\ref{convergent}). At the same time, we define the {\it
measure-theoretic $\mu$-entropy of $\mathcal{U}$} as
\[
h_\mu(G,\mathcal{U})=\inf_{\alpha\in \mathcal{P}_X: \alpha\succeq
\mathcal{U}} h_\mu(G,\alpha).
\]

We obtain directly the following easy facts.

\begin{lem} \label{basic}
Let $\mu\in \mathcal{M}(X,G)$ and $\mathcal{U}, \mathcal{V}
\in\mathcal{C}_X$. Then
\begin{description} \label{props}

\label{props3}
\item[1] $h^-_\mu(G,\mathcal{U})\le
h_\mu(G,\mathcal{U})$ and $h^-_\mu(G,\mathcal{U})\le
h_{\text{top}} (G, \mathcal{U})$.

\label{props1}
\item[2] $h_\mu(G,\mathcal{U}\vee \mathcal{V})\leq
h_\mu(G,\mathcal{U})+ h_\mu(G,\mathcal{V})$ and
$h_\mu^-(G,\mathcal{U}\vee \mathcal{V})\leq
h_\mu^-(G,\mathcal{U})+ h_\mu^-(G,\mathcal{V})$.

\label{newnewbasic}
\item[3] If $\mathcal{U} \succeq \mathcal{V}$ then $h_\mu(G,\mathcal{U})\ge
h_\mu(G,\mathcal{V})$ and $h_\mu^-(G,\mathcal{U})\ge
h_\mu^-(G,\mathcal{V})$. \label{props9}
\end{description}
\end{lem}

\subsubsection{An alternative formula for \eqref{hmu}}

Let $\mu\in \mathcal{M}(X,G)$. Since $\mathcal{P}_X\subseteq
\mathcal{C}_X$, we have
\begin{equation} \label{07-03-05-04}
h_\mu(G, X)=\sup_{\mathcal{U}\in \mathcal{C}_X}
h_\mu^-(G,\mathcal{U})=\sup_{\mathcal{U}\in \mathcal{C}_X}
h_\mu(G,\mathcal{U}).
\end{equation}
In fact, the above extension of local measure-theoretic entropy from
partitions to covers allows us to give another alternative formula
for \eqref{hmu}.

\begin{thm} \label{ope}
Let $\mu\in \mathcal{M}(X,G)$. Then
\begin{align}\label{PU}
h_\mu(G, X)= \sup_{\mathcal{U}\in \mathcal{C}^{o}_X}
h_\mu^-(G,\mathcal{U})= \sup_{\mathcal{U}\in \mathcal{C}^{o}_X}
h_\mu (G,\mathcal{U}).
\end{align}
\end{thm}
\begin{proof}
By \eqref{07-03-05-04}, $h_\mu(G, X)\ge \sup_{\mathcal{U}\in
\mathcal{C}^{o}_X} h_\mu (G,\mathcal{U})$. For the other
direction, let $\alpha=\{ A_1,\cdots,A_k \}\in \mathcal{P}_X$ and
$\epsilon>0$.

\medskip

\noindent{\bf Claim.} There exists $\mathcal{U}\in \mathcal{C}_X^{
o}$ such that $H_{\mu} (g^{-1}\alpha| \beta )\le \epsilon$ if
$g\in G$ and $\beta\in \mathcal{P}_X$ satisfy $\beta\succeq g^{-1}
\mathcal{U}$.

\begin{proof}[Proof of Claim]
By \cite[Lemma 4.15]{W}, there exists $\delta_1= \delta_1 (k,
\epsilon)>0$ such that if
 $\beta_i=\{ B^i_1,\cdots,B^i_k \}\in \mathcal{P}_X, i= 1, 2$ satisfy
 $\sum_{i=1}^k
\mu(B^1_i\Delta B^2_i)<\delta_1$ then $H_{\mu} (\beta_1|\beta_2)\le
\epsilon$. Since $\mu$ is regular, we can take closed subsets
$B_i\subseteq A_i$ with $\mu(A_i\setminus
B_i)<\frac{\delta_1}{2k^2}$, $i=1,\cdots,k$. Let $B_0=X\setminus
\bigcup_{i=1}^kB_i$, $U_i=B_0\cup B_i,i=1,\cdots,k$. Then
$\mu(B_0)<\frac{\delta_1}{2k}$ and $\mathcal{U}=\{ U_1,\cdots,U_k
\}\in \mathcal{C}^{o}_X$.

Let $g\in G$. If $\beta\in \mathcal{P}_X$ is finer than $g^{-1}
\mathcal{U}$, we can find $\beta' =\{ C_1,\cdots,C_k \}\in
\mathcal{P}_X$ satisfying
 $C_i\subseteq g^{-1} U_i, \ i=1,\cdots,k$ and $\beta\succeq
\beta'$, and so $H_{\mu} (g^{-1}\alpha| \beta)\le H_{\mu}
(g^{-1}\alpha| \beta')$. For each $i= 1, \cdots, k$, as
$g^{-1}U_i\supseteq C_i \supseteq X\setminus \bigcup_{l\neq i}
g^{-1}U_l=g^{-1}B_i$ and $g^{-1}A_i\supseteq g^{-1}B_i$, one has
\[
\mu(C_i\Delta g^{-1}A_i)\le \mu(g^{-1}A_i\setminus
g^{-1}B_i)+\mu(g^{-1}B_0)=\mu(A_i\setminus
B_i)+\mu(B_0)<\frac{\delta_1}{2k}+\frac{\delta_1}{2k^2}\le
\frac{\delta_1}{k}.
\]
Thus $\sum_{i=1}^k \mu(C_i\Delta g^{-1}A_i)<\delta_1$.  It follows
that $H_{\mu}(g^{-1}\alpha|\beta')\le \epsilon$ and  hence
$H_{\mu}(g^{-1}\alpha|\beta)\le \epsilon$.
\end{proof}

Let $F\in F(G)$. If $\beta\in \mathcal{P}_X$ is finer than
$\mathcal{U}_F$, then $\beta\succeq g^{-1}\mathcal{U}$ for each
$g\in F$, and so using the above Claim one has
\begin{equation*}
H_{\mu}(\alpha_F) \le H_{\mu}(\beta)+ H_{\mu}(\alpha_F|\beta) \le
H_{\mu}(\beta)+\sum_{g\in F}H_{\mu}(g^{-1}\alpha|\beta) \le
H_{\mu}(\beta)+|F|\epsilon.
\end{equation*}
Moreover, $H_\mu (\alpha_F)\le H_\mu (\mathcal{U}_F)+ |F|\epsilon$.
Now letting $F$ range over $\{ F_n \}_{n\in \mathbb{N}}$ one has
\begin{eqnarray*}
h_\mu(G,\alpha)&= & \lim_{n\rightarrow +\infty} \frac{1}{|F_n|}
H_{\mu}(\alpha_{F_n})\le \limsup_{n\rightarrow +\infty}
\frac{1}{|F_n|}H_{\mu}(\mathcal{U}_{F_n})+\epsilon
\\
&= & h^-_\mu(G,\mathcal{U})+\epsilon\le \sup_{\mathcal{V}\in
\mathcal{C}^{o}_X} h_\mu^-(G,\mathcal{V})+\epsilon.
\end{eqnarray*}
Since $\alpha$ and $\epsilon$ are arbitrary, $h_\mu(G, X)\le \sup
\limits_{\mathcal{V}\in \mathcal{C}^{o}_X} h_\mu^-(G,\mathcal{V})$
and so
\begin{equation*}
h_\mu(G, X)\le \sup_{\mathcal{V}\in \mathcal{C}^{o}_X} h_\mu^-
(G,\mathcal{V})\le \sup_{\mathcal{V}\in \mathcal{C}^{o}_X} h_\mu
(G,\mathcal{V})\ (\text{by Lemma \ref{basic}~(1)}).
\end{equation*}
\end{proof}

\subsubsection{U.s.c. of measure-theoretic entropy of open covers}

A real-valued function $f$ defined on a compact metric space $Z$
is called {\it upper semi-continuous} (u.s.c) if one of
the following equivalent conditions holds:
\begin{enumerate}

\item[(A1)] $\limsup_{z'\rightarrow z} f(z')\le f(z)$ for each
$z\in Z$;

\item[(A2)] for each $r\in \R$, the set $\{ z\in Z:f(z)\ge r \}$ is
closed.
\end{enumerate}
Using (A2), the infimum of any family of u.s.c. functions is again a
u.s.c. one; both the sum and the supremum of finitely many u.s.c.
functions are u.s.c. ones.

In this sub-section, we aim to prove that those two kinds entropy of
open covers over $\mathcal{M} (X, G)$ are both u.s.c. First, we need

\begin{lem} \label{observation}
Let $\mathcal{U}= \{U_1, \cdots, U_M\}\in \mathcal{C}^{o}_X$ and
$F\in F(G)$. Then the function $\psi: \mathcal{M} (X)\rightarrow
\R_+$ with $\psi(\mu)=\inf_{\alpha\in \mathcal{P}_X: \alpha\succeq
\mathcal{U}} H_\mu (\alpha_F)$ is u.s.c.
\end{lem}
\begin{proof}
Fix $\mu\in \mathcal{M}(X)$ and $\epsilon>0$. It is sufficient to
prove that
\begin{equation} \label{07-03-06-01}
\limsup_{\mu'\rightarrow \mu: \mu'\in \mathcal{M}(X)}
\psi(\mu')\le \psi(\mu)+\epsilon.
\end{equation}

We choose $\alpha\in \mathcal{P}_X$ such that $\alpha\succeq
\mathcal{U}$ and $H_\mu(\alpha_F)\le \psi(\mu)+ \frac{\epsilon}{2}$.
With no loss of generality we assume $\alpha= \{A_1, \cdots, A_M\}$ with $A_i\subseteq
U_i$, $1\le i\le M$. Then there exists $\delta= \delta (M, F,
\epsilon)> 0$ such that if $\beta^i= \{B_1^i, \cdots, B_M^i\}\in
\mathcal{P}_X$, $i= 1, 2$ satisfy $\sum _{i= 1}^M \sum_{g\in F} g\mu
(B_i^1\Delta B_i^2)< \delta$ then $H_\mu(\beta^1_F|\beta_F^2)\le
\sum_{g\in F} H_{g\mu} (\beta^1| \beta^2)< \frac{\epsilon}{2}$
(\cite[Lemma 4.15]{W}). Set $\mathcal{U}^*_{\mu, F}= \{\beta\in
\mathcal{P}_X: \beta\succeq \mathcal{U}, \mu (\bigcup_{B\in \beta_F}
\partial B)= 0\}$.

\medskip

\noindent{\bf Claim.} There exists $\beta= \{B_1, \cdots, B_M\}\in
\mathcal{U}^*_{\mu,F}$ such that $H_\mu (\beta_F| \alpha_F)<
\frac{\epsilon}{2}$.

\begin{proof}[Proof of Claim]
Let $\delta_1\in (0,\frac{\delta}{2M})$. By the regularity of $\mu$,
there exists compact $C_j\subseteq A_j$ such that
\begin{align}\label{AOU1}
\sum_{g\in F} g\mu(A_j\setminus C_j)<\frac{\delta_1}{M}, j=
1,\cdots, M.
\end{align}
For $j\in \{ 1,\cdots,M\}$, set $O_j=U_j\cap (X \setminus
\bigcup_{i\neq j} C_i)$, then $O_j$ is an open subset of $X$
satisfying
\begin{align}\label{AOU}
A_j\subseteq O_j\subseteq U_j \text{ and } \sum_{g\in F}
g\mu(O_j\setminus A_j)\le \sum_{i\neq j} \sum_{g\in F}
g\mu(A_i\setminus C_i)<\delta_1,\ \text{as}\ O_j\setminus
A_j\subseteq \bigcup_{i\neq j} A_i\setminus C_i.
\end{align}
Note that if $x\in X$ then there exist at most countably many
$\gamma> 0$ such that $\{y\in X: d (x, y)= \gamma\}$ has positive $g
\mu$-measure for some $g\in F$. Moreover, as $O_1, \cdots, O_M$ are
open subsets of $X$ and $\bigcup_{i= 1}^M O_i= X$, it is not hard to
obtain Borel subsets $C_1^*, \cdots, C_M^*$ such that
$C_i^*\subseteq O_i, 1\le i\le M, \bigcup_{i= 1}^M C_i^*= X$ and
$\sum_{i= 1}^M \sum_{g\in F} g \mu (\partial C_i^*)= 0$.

Set $B_1=C_1^*$, $B_j= C_j^*\setminus (\bigcup_{i=1}^{j-1} C_i^*)$,
$2\le j\le M$. Then $\beta\doteq \{ B_1,\cdots,B_M\}\in
\mathcal{P}_X$ and $\beta\succeq \mathcal{U}$. As $g^{-1}(\partial
D)=
\partial (g^{-1}D)$ for each $g\in F$ and $D\subseteq X$, by the
construction of $C_1^*, \cdots, C_M^*$ it's easy to check that $\mu
(\bigcup_{B\in \beta_F}
\partial B)= 0$ and so $\beta\in \mathcal{U}^*_{\mu,F}$. Note that if
$1\le j\neq i\le M$ then $B_j\cap C_i\subseteq O_j\cap C_i=
\emptyset$, which implies $C_i\subseteq B_i\subseteq O_i$ for all
$1\le i\le M$. By \eqref{AOU1} and \eqref{AOU},
$$\sum_{i=1}^M\sum_{g\in F} g\mu(A_i\Delta B_i)\le \sum_{i=1}^M \sum_{g\in F}(g\mu(A_i\setminus
C_i)+g\mu(O_i\setminus A_i))\le \sum_{i=1}^M 2\delta_1<\delta.$$
Thus $H_\mu(\beta_F|\alpha_F)<\frac{\epsilon}{2}$ (by the selection
of $\delta$). This finishes the proof of the claim.
\end{proof}

Now, note that $\beta\in \mathcal{P}_X$ satisfies $\beta\succeq
\mathcal{U}$ and $\mu(\bigcup_{B\in \beta_F}\partial B)=0$, one has
\begin{eqnarray*}
\limsup_{\mu'\rightarrow \mu: \mu'\in \mathcal{M}(X)} \psi(\mu')
&\le & \limsup_{\mu'\rightarrow \mu, \mu'\in
\mathcal{M}(X)}H_{\mu'}(\beta_F)=H_\mu(\beta_F) \\
&\le & H_\mu(\alpha_F)+H_\mu(\beta_F|\alpha_F)\le \psi(\mu)+\epsilon
 \ \text{(by Claim)}.
\end{eqnarray*}
This establishes \eqref{07-03-06-01} and so completes the proof of
the lemma.
\end{proof}

\begin{lem} \label{07-01-02-01}
Let $\mu\in \mathcal{M} (X, G)$, $M\in \mathbb{N}$ and $\epsilon>
0$. Then there exists $\delta> 0$ such that if $\mathcal{U}=\{
U_1,\cdots,U_M\}\in \mathcal{C}_X$, $\mathcal{V}=\{
V_1,\cdots,V_M\}\in \mathcal{C}_X$ satisfy $\mu(\mathcal{U}\Delta
\mathcal{V})\doteq \sum_{m=1}^M \mu(U_m\Delta V_m)< \delta$ then
$|h_\mu (G, \mathcal{U})- h_\mu (G, \mathcal{V})|\le \epsilon$.
\end{lem}
\begin{proof} We follow the arguments in the proof of \cite[Lemma 5]{HMRY}. Fix $M
\in\mathbb{N}$ and $\epsilon>0$. Then there exists
$\delta'=\delta'(M,\varepsilon)>0$ such that for $M$-sets
partitions $\alpha,\beta$ of $X$, if $\mu(\alpha\Delta
\beta)<\delta'$ then $H_\mu(\beta|\alpha)<\epsilon$ (see for
example \cite[Lemma 4.15]{W}). Let $\mathcal{U}=\{
U_1, \cdots,U_M\}$ and $\mathcal{V}=\{ V_1, \cdots,V_M\}$ be
any two $M$-sets covers of $X$ with
$\mu(\mathcal{U}\Delta\mathcal{V})<{\delta'\over M}=\delta$.

\medskip
\noindent {\bf Claim.} for every finite partition $\alpha \succeq
\mathcal{U}$ there exists a finite partition
$\beta\succeq\mathcal{V}$ with $H_\mu(\beta|\alpha)<\epsilon$.

\begin{proof}[Proof of Claim] Since $\alpha\succeq \mathcal{U}$, there exists a partition
$\alpha'=\{ A_1,\cdots,A_M \}$ with $A_i\subseteq U_i$,
$i=1, \cdots,M$ and $\alpha\succeq \alpha'$, where $A_i$ may be
empty. Let
\begin{equation*}
B_1=V_1\setminus \bigcup_{k>1}(A_k\cap V_k),
\end{equation*}
\begin{equation*}
B_i=V_i\setminus \left(\bigcup_{k>i}(A_k\cap V_k)\cup \bigcup
\limits_{j<i}B_j\right),\ i  \in \{ 2,\cdots, M \}.
\end{equation*}
Then $\beta=\{B_1,\cdots,B_M\}\in \mathcal{P}_X$ which satisfies
$B_m\subseteq V_m$ and $A_m\cap V_m\subseteq B_m$ for $m\in
\{1,\cdots,M\}$. It is clear that $A_m \setminus B_m \subseteq U_m
\setminus V_m$ and
\begin{eqnarray*}
B_m\setminus A_m&=& \left(X\setminus \bigcup_{k\neq
m}B_k\right)\setminus A_m= \bigcup_{j\neq m}
A_j\setminus \bigcup_{k\neq m}B_k \\
& \subseteq &
\bigcup \limits_{k\ne m} (A_k\setminus B_k)\subseteq \bigcup
\limits_{k\ne m} (U_k\setminus V_k).
\end{eqnarray*}
Hence for every $m\in\{1,\cdots,M\}$, $A_m\Delta B_m\subseteq
\bigcup\limits_{k=1}^M (U_k\Delta V_k)$ and
$\mu(\alpha'\Delta\beta)\le M\cdot
\mu(\mathcal{U}\Delta\mathcal{V})<\delta'$. This implies that
$H_\mu(\beta|\alpha')<\epsilon$. Moreover, $H_\mu(\beta|\alpha)\le
H_\mu(\beta|\alpha')<\epsilon$.
\end{proof}

Fix $n\in \mathbb{N}$. For any $\alpha\in \mathcal{P}_X$ with
$\alpha\succeq \mathcal{U}_{F_n}$, we have $g\alpha\succeq
\mathcal{U}$ for $g\in F_n$. By the above Claim, there exists
$\beta_g\in \mathcal{P}_X$ such that $\beta_g\succeq \mathcal{V}$
and $H_\mu(\beta_g|g\alpha)<\epsilon$, i.e.,
$H_\mu(g^{-1}\beta_g|\alpha)<\epsilon$. Let $\beta=\bigvee_{g\in
F_n}g^{-1}\beta_g$. Then $\beta\in \mathcal{P}_X$ with
$\beta\succeq \mathcal{V}_{F_n}$. Now
\begin{eqnarray*}
H_\mu(\mathcal{V}_{F_n})&\le & H_\mu(\beta)\le H_\mu(\beta\vee
\alpha)=H_\mu(\alpha)+H_\mu(\beta|\alpha)\\
&\le & H_\mu(\alpha)+\sum_{g\in
F_n}H_\mu(g^{-1}\beta_g|\alpha)<H_\mu(\alpha)+n\epsilon.
\end{eqnarray*}
Since this is true for any $\alpha\in
\mathcal{P}_X$ with $\alpha\succeq \mathcal{U}_{F_n}$, we get
$\frac{1}{|F_n|}H_\mu(\mathcal{V}_{F_n})\le
\frac{1}{|F_n|}H_\mu(\mathcal{U}_{F_n})+\epsilon$.

Exchanging the roles of $\mathcal{U}$ and $\mathcal{V}$ we get
$$\frac{1}{|F_n|}H_\mu(\mathcal{U}_{F_n})\le
\frac{1}{|F_n|}H_\mu(\mathcal{V}_{F_n})+\epsilon.$$ This shows
$\frac{1}{|F_n|}|H_\mu(\mathcal{U}_{F_n})-H_\mu(\mathcal{V}_{F_n})|\le
\epsilon$. Letting $n\rightarrow +\infty$, one has
$|h_\mu(G,\mathcal{U})-h_\mu(G,\mathcal{V})|\le \epsilon$.
\end{proof}

Now we can prove the u.s.c. property of those two kinds of
measure-theoretic entropy of open covers over $\mathcal{M} (X, G)$.

\begin{prop}\label{usc-em_1}
Let $\mathcal{U}\in \mathcal{C}_X^o$. Then $h_{\{\cdot\}}
(G,\mathcal{U}): \mathcal{M}(X, G)\rightarrow \R_+$ is u.s.c. on
$\mathcal{M} (X, G)$.
\end{prop}
\begin{proof}
Note that
\begin{eqnarray*}
h_\mu(G,\mathcal{U})&= &\inf_{\alpha\in \mathcal{P}_X: \alpha
\succeq \mathcal{U}} h_\mu(G,\alpha)=\inf_{\alpha\in
\mathcal{P}_X: \alpha \succeq \mathcal{U}} \inf_{B\in F(G)}
\frac{H_\mu(\alpha_B) }{|B|}\  \text{(by Lemma \ref{lem-436}~(4))} \\
&=&\inf_{B\in F(G)} \inf_{\alpha\in \mathcal{P}_X: \alpha \succeq
\mathcal{U}}
 \frac{H_\mu(\alpha_B )}{|B|}.
\end{eqnarray*}
Since $\mu\mapsto \inf_{\alpha\in \mathcal{P}_X: \alpha \succeq
\mathcal{U}} H_\mu(\alpha_B )$ is u.s.c. (see Lemma
\ref{observation}) and the infimum of any family of u.s.c.
functions is again u.s.c., one has $h_{\{\cdot\}} (G,\mathcal{U}):
\mathcal{M}(X, G)\rightarrow \R_+$ is u.s.c. on $\mathcal{M} (X,
G)$.
\end{proof}

\begin{prop}\label{usc-em_2}
Let $\mathcal{U}\in \mathcal{C}_X^o$. Then $h_{\{\cdot\}}^-
(G,\mathcal{U}): \mathcal{M}(X,G)\rightarrow \R_+$ is u.s.c. on
$\mathcal{M} (X, G)$.
\end{prop}
\begin{proof}
With no loss of generality we assume $e_G\in F_1\subseteq F_2\subseteq \cdots$ by
Lemma \ref{convergent}. Let $\mu\in \mathcal{M} (X, G)$ and
$\epsilon\in (0,\frac{1}{4})$. Then there exists $N\in \mathbb{N}$
with
\begin{equation} \label{semi-esti-1}
\sup_{n\ge N} \frac{H_\mu (\mathcal{U}_{F_n})}{|F_n|}\le h_\mu^-
(G, \mathcal{U})+ \frac{\epsilon}{2}.
\end{equation}
 By Lemma \ref{conv-lem}, there exist
integers $n_1, \cdots, n_k$ with $N\le n_1< \cdots< n_k$ such that
\begin{eqnarray}
\label{semi-esti-02}
h_\nu^- (G, \mathcal{U})&=&
\lim_{n\rightarrow +\infty} \frac{H_\nu
(\mathcal{U}_{F_n})}{|F_n|}\le \max_{1\le i\le k} \frac{H_\nu
(\mathcal{U}_{F_{n_i}})}{|F_{n_i}|}+ \frac{\epsilon
H_\nu(\mathcal{U})}{2\log (N(\mathcal{U})+1)}\nonumber \\
&\le&\max_{1\le i\le k} \frac{H_\nu
(\mathcal{U}_{F_{n_i}})}{|F_{n_i}|}+ \frac{\epsilon}{2} \text{ for
each } \nu\in \mathcal{M} (X, G).
\end{eqnarray}
Then we have
\begin{eqnarray} \label{07-03-06-02}
\limsup_{\mu'\rightarrow \mu, \mu'\in \mathcal{M} (X, G)} h_{\mu'}^-
(G, \mathcal{U}) &\le & \frac{\epsilon}{2}+ \limsup_{\mu'\rightarrow
\mu, \mu'\in \mathcal{M} (X, G)} \max_{1\le i\le k} \frac{H_{\mu'}
(\mathcal{U}_{F_{n_i}})}{|F_{n_i}|}\ (\text{using
\eqref{semi-esti-02}})\ \nonumber \\
&= & \frac{\epsilon}{2}+ \max_{1\le i\le k} \limsup_{\mu'\rightarrow
\mu, \mu'\in \mathcal{M} (X, G)} \frac{H_{\mu'}
(\mathcal{U}_{F_{n_i}})}{|F_{n_i}|}\ \nonumber \\
&\le & \frac{\epsilon}{2}+ \max_{1\le i\le k} \frac{H_\mu
(\mathcal{U}_{F_{n_i}})}{|F_{n_i}|}\ (\text{Lemma
\ref{observation}})\ \nonumber \\
&\le & \frac{\epsilon}{2}+ \sup_{n\ge N} \frac{H_\mu
(\mathcal{U}_{F_n})}{|F_n|}\le h_\mu^- (G, \mathcal{U})+ \epsilon\
(\text{using \eqref{semi-esti-1}}).
\end{eqnarray}
Thus, we claim the conclusion from the arbitrariness of $\mu\in
\mathcal{M} (X, G)$ and $\epsilon\in (0,\frac{1}{4})$ in
\eqref{07-03-06-02}.
\end{proof}

\subsubsection{Affinity of measure-theoretic entropy of covers}

Let $\mu= a \nu+ (1- a)\eta$, where $\nu,\eta\in \mathcal{M} (X, G)$
and $0< a <1$. Using the concavity of $\phi(t)=-t \log t$ on $[0,
1]$ with $\phi (0)= 0$ (fix it in the remainder of the paper), one
has if $\beta\in \mathcal{P}_X$ and $F\in F(G)$ then
$0\le
H_{\mu}(\beta_F)-a H_{\nu}(\beta_F)-(1-a)H_{\eta}(\beta_F)\le
\phi(a)+\phi(1-a)$ (see for example the proof of
\cite[Theorem 8.1]{W}) and so
\begin{align}\label{eq-36}
h_\mu(G,\beta)=ah_\nu(G,\beta)+(1-a)h_\eta(G,\beta),
\end{align}
i.e. the function $h_{\{\cdot\}} (G, \beta): \mathcal{M}(X,
G)\rightarrow \mathbb{R}_+$ is affine. In the following, we shall
show the affinity of $h_{\{\cdot\}} (G,\mathcal{U})$ and
$h_{\{\cdot\}}^-(G,\mathcal{U})$ on $\mathcal{M} (X, G)$ for each
$\mathcal{U}\in \mathcal{C}_X$.

Let $\mu\in \mathcal{M}(X,G)$ and $\mathcal{B}_X^\mu$ be the
completion of $\mathcal{B}_X$ under $\mu$. Then $(X,
\mathcal{B}_X^\mu,\mu,G)$ is a Lebesgue system. If $\{
\alpha_i\}_{i\in I}$ is a countable family in $\mathcal{P}_X$, the
partition $\alpha= \bigvee_{i\in I} \alpha_i\doteq \{\bigcap_{i\in
I} A_i: A_i\in \alpha_i, i\in I\}$ is called a {\it measurable
partition}. Note that the sets $A\in \mathcal{B}_X^\mu$, which are
unions of atoms of $\alpha$, form a sub-$\sigma$-algebra of
$\mathcal{B}_X^\mu$, which is denoted by $\widehat{\alpha}$ or
$\alpha$ if there is no ambiguity. In fact, every
sub-$\sigma$-algebra of $\mathcal{B}_X^\mu$ coincides with a
$\sigma$-algebra constructed in this way in the sense of mod $\mu$
\cite{Rohlin}. We consider the sub-$\sigma$-algebra $I_\mu= \{
A\in \mathcal{B}_X^\mu: \mu(gA\Delta A)=0 \text{ for each }g\in
G\}$. Clearly, $I_\mu$ is $G$-invariant since $G$ is countable.
Let $\alpha$ be the measurable partition of $X$ with
$\widehat{\alpha}= I_\mu$ (mod $\mu$). With no loss of generality we may require that
$\alpha$ is $G$-invariant, i.e. $g\alpha= \alpha$ for any $g\in
G$. Let $\mu= \int_X \mu_x d \mu(x)$ be the disintegration of
$\mu$ over $I_\mu$, where $\mu_x\in \mathcal{M}^e (X, G)$ and
$\mu_x (\alpha(x))= 1$ for $\mu$-a.e. $x\in X$, here $\alpha (x)$
denotes the atom of $\alpha$ containing $x$. This disintegration
is known as the {\it ergodic decomposition} of $\mu$ (see for
example \cite[Theorem 3.22]{G1}).

The disintegration is characterized by properties \eqref{meas1}
and \eqref{meas3} below:
\begin{eqnarray}
&& \label{meas1}  \text{for every } f \in
L^1(X,\mathcal{B}_X,\mu),  f \in L^1(X,\mathcal{B}_X,\mu_x) \text{
for $\mu$-a.e. } x\in X,\\
&& \text{and the map }x \mapsto \int_X  f(y)\,d\mu_x(y)\text{ is in
}L^1(X,I_\mu,\mu) \nonumber;
\end{eqnarray}
\begin{equation}
\label{meas3} \text{for every } f\in
L^1(X,\mathcal{B}_X,\mu),\mathbb{E}_{\mu}(f|I_\mu )(x)=\int_X
f\,d\mu_{x} \ \ \text{for $\mu$-a.e. } x\in X.
\end{equation}
Then for $f \in L^1(X,\mathcal{B}_X,\mu)$,
\begin{equation} \label{07-03-06-04}
\int_X \left(\int_X f\,d\mu_x \right)\, d\mu(x)=\int_X f \,d\mu.
\end{equation}

Note that the disintegration exists uniquely in the sense that if
$\mu= \int_X \mu_x d \mu (x)$ and $\mu= \int_X \mu'_x d \mu (x)$
are both the disintegrations of $\mu$ over $I_\mu$, then $\mu_x=
\mu'_x$ for $\mu$-a.e. $x\in X$. Moreover, there exists a
$G$-invariant subset $X_0\subseteq X$ such that $\mu(X_0)=1$ and
if for $x\in X_0$ we define $\Gamma_x= \{y\in X_0: \mu_x=\mu_y \}$
then $\Gamma_x= \alpha(x)\cap X_0$ and $\Gamma_x$ is
$G$-invariant.

\begin{lem} \label{lem4-2}
Let $\mu\in \mathcal{M}(X,G)$ with $\mu=\int_X \mu_x d \mu(x)$ the
ergodic decomposition of $\mu$ and $\mathcal{V}\in \mathcal{C}_X$.
Then $H_\mu(\mathcal{V}|I_\mu)=\int_X H_{\mu_x}(\mathcal{V})
d\mu(x)$.
\end{lem}
\begin{proof}Let $\mathcal{V}=\{ V_1,\cdots,V_n\}$. For any
$s=(s(1),\cdots,s(n))\in { \{ 0,1 \} }^n$, set
$V_s=\bigcap_{i=1}^{n}V_i(s(i))$, where $V_i(0)=V_i$ and
$V_i(1)=X\setminus V_i$. Let $\alpha=\{ V_s: s\in { \{ 0,1 \} }^n
\}$. Then $\alpha$ is the Borel partition generated by
$\mathcal{V}$ and put $P(\mathcal{V})=\{ \beta\in
\mathcal{P}_X:\alpha \succeq \beta \succeq \mathcal{V}\}$, which
is a finite family of partitions. It is well known that, for each
$\theta \in \mathcal{M}(X)$ one has
\begin{align}\label{keyeq1}
H_\theta(\mathcal{V})=\min_{\beta\in P(\mathcal{V})}H_\theta(\beta),
\end{align}
see for example the proof of \cite[Proposition 6]{R}. Now denote
$P(\mathcal{V})=\{ \beta_1,\cdots, \beta_l \}$ and put
$$A_i=\left\{ x\in X:
H_{\mu_x}(\beta_i)= \min_{\beta\in P(\mathcal{V})} H_{\mu_x}(\beta)
\right\}, i\in \{1,\cdots, l\}.$$ Let $B_1= A_1$, $B_2= A_2\setminus
B_1$, $\cdots$, $B_l= A_l\setminus \bigcup_{i=1}^{l-1} B_i$ and
$B_0=X\setminus \bigcup_{i=1}^l A_i$. By \eqref{keyeq1},
$\mu(B_0)=0$.

Set $\beta^*= \{ B_0 \cap \beta_1\} \cup \{ B_i\cap \beta_i: i=1,
\cdots, l\}\in \mathcal{P}_X$ (mod $\mu$). Then $\beta^*\succeq
\mathcal{V}$. Clearly, for $i\in \{ 1,\cdots,l \}$ and $\mu$-a.e.
$x\in B_i$, $H_{\mu_x}(\beta^*)= H_{\mu_x} (\beta_i)= \min_{\beta
\in P(\mathcal{V})} H_{\mu_x}(\beta)= H_{\mu_x}(\mathcal{V})$
where the last equality follows from \eqref{keyeq1}. Combining
this fact with $\mu(B_0)=0$ one gets
$H_{\mu_x}(\beta^*)=H_{\mu_x}(\mathcal{V})$ for $\mu$-a.e. $x\in
X$. This implies
\begin{eqnarray*}
H_\mu(\mathcal{V}| I_\mu) &\le & H_\mu(\beta^*|I_\mu)=
\int_X H_{\mu_x}(\beta^*) d \mu(x) \ \text{(using \eqref{meas3})} \\
&= & \int_X H_{\mu_x}(\mathcal{V}) d \mu(x)\le \inf_{\beta\in
\mathcal{P}_X: \beta \succeq \mathcal{V}} \int_X
H_{\mu_x}(\beta) d \mu(x)\\
&= & \inf_{\beta\in \mathcal{P}_X: \beta \succeq \mathcal{V}}
H_\mu(\beta|I_\mu)=H_\mu(\mathcal{V}|I_\mu).
\end{eqnarray*}
Thus $H_\mu(\mathcal{V}| I_\mu)= \int_X H_{\mu_x}(\mathcal{V}) d
\mu(x)$. This finishes the proof.
\end{proof}

Then we have

\begin{prop}\label{entr-imu}
Let $\mathcal{U}\in \mathcal{C}_X$ and $\mu\in \mathcal{M}(X,G)$. If
$e_G\in F_1\subseteq F_2\subseteq \cdots$ then
$$h_\mu^-(G,\mathcal{U})=\lim_{n\rightarrow +\infty}
\frac{1}{|F_n|}H_\mu(\mathcal{U}_{F_n}|I_\mu).$$
\end{prop}
\begin{proof}
It is easy to check that $F\in F (G)\mapsto H_\mu (\mathcal{U}_F|
I_\mu)$ is a m.n.i.s.a. function, and so the sequence
$\{\frac{1}{|F_n|}H_\mu(\mathcal{U}_{F_n}|I_\mu)\}_{n\in \N}$
converges, say it converges to $f_\mathcal{U}$ (see Lemma
\ref{convergent}). Clearly $h_\mu^-(G,\mathcal{U})\ge
f_\mathcal{U}$.

Now we aim to prove $h_\mu^-(G,\mathcal{U})\le f_\mathcal{U}$. Let
$\epsilon\in (0,\frac{1}{4})$ and $N\in \mathbb{N}$. By Proposition
\ref{ow-prop} there exist integers $n_1, \cdots, n_k$ with $N\le
n_1< \cdots< n_k$ such that if $n$ is large enough then $F_{n_1},
\cdots, F_{n_k}$ $\epsilon$-quasi-tile the set $F_n$ with tiling
centers $C_1^n, \cdots, C_k^n$ and so
\begin{equation} \label{condition1}
F_n\supseteq \bigcup_{i= 1}^k F_{n_i} C_i^n\ \text{and}\
|\bigcup_{i= 1}^k F_{n_i} C_i^n|\ge \max \left\{(1- \epsilon) |F_n|,
(1- \epsilon) \sum_{i= 1}^k |C_i^n|\cdot |F_{n_i}|\right\}.
\end{equation}

Thus if $\alpha\in \mathcal{P}_X$ and $n$ is large enough then
\begin{eqnarray} \label{07-03-06-03}
H_\mu(\mathcal{U}_{F_n}|\alpha_{F_n}) &\le &
H_\mu(\mathcal{U}_{F_n\setminus \bigcup_{i= 1}^k F_{n_i}
C_i^n}|\alpha_{F_n})+\sum_{i=1}^k
H_\mu(\mathcal{U}_{F_{n_i}C_i^n}|\alpha_{F_n})\ \nonumber \\
&\le & |F_n\setminus \bigcup_{i= 1}^k F_{n_i}C_i^n|\cdot \log
N(\mathcal{U}) +\sum_{i=1}^k
H_\mu(\mathcal{U}_{F_{n_i}C_i^n}|\alpha_{F_{n_i}C_i^n}).
\end{eqnarray}
This implies
\begin{eqnarray*}
& & \limsup_{n\rightarrow
+\infty}\frac{1}{|F_n|}H_\mu(\mathcal{U}_{F_n}|\alpha_{F_n}) \\
&\le & \epsilon \log N(\mathcal{U})+\limsup_{n\rightarrow +\infty}
\frac{1}{|F_n|}\sum_{i=1}^k \sum_{g\in C_i^n}
H_\mu(\mathcal{U}_{F_{n_i}g}|\alpha_{F_{n_i} g})\ (\text{using}\
\eqref{condition1}\ \text{and}\ \eqref{07-03-06-03}) \\
&\le & \epsilon \log N(\mathcal{U})+\limsup_{n\rightarrow +\infty}
\frac{\sum_{i=1}^k |F_{n_i}||C_i^n|}{|F_n|} \max_{1\le i\le k}
\frac{1}{|F_{n_i}|}H_\mu(\mathcal{U}_{F_{n_i}}|\alpha_{F_{n_i}})\\
&\le & \epsilon \log N(\mathcal{U})+\frac{1}{1-\epsilon} \max_{1\le
i\le k}
\frac{1}{|F_{n_i}|}H_\mu(\mathcal{U}_{F_{n_i}}|\alpha_{F_{n_i}})\
(\text{using}\ \eqref{condition1}).
\end{eqnarray*}
Thus
\begin{eqnarray}\label{le-key-eq}
h_\mu^-(G,\mathcal{U})&\le & \limsup_{n\rightarrow +\infty}
\frac{1}{|F_n|} \left( H_\mu(\mathcal{U}_{F_n}|\alpha_{F_n})+
H_\mu(\alpha_{F_n}) \right)\ \nonumber\\
&\le & h_\mu(G,\alpha)+\epsilon \log
N(\mathcal{U})+\frac{1}{1-\epsilon} \max_{1\le i\le k}
\frac{1}{|F_{n_i}|}H_\mu(\mathcal{U}_{F_{n_i}}|\alpha_{F_{n_i}}).
\end{eqnarray}
Note that if $\alpha\in \mathcal{P}_X$ satisfies $\alpha\subseteq
I_\mu$ then $h_\mu (G, \alpha) = 0$. In particular, in
\eqref{le-key-eq} we replace $\alpha$ by a sequence
$\{\alpha_i\}_{i\in \N}$ in $\mathcal{P}_X$ with $\alpha_i\nearrow
I_\mu$, then
$$h_\mu^-(G,\mathcal{U})\le \epsilon \log
N(\mathcal{U})+\frac{1}{1-\epsilon} \sup_{m\ge N}
\frac{1}{|F_{m}|}H_\mu(\mathcal{U}_{F_{m}}|I_\mu).$$ Since the above
inequality is true for any $\epsilon\in (0, \frac{1}{4})$ and $N\in
\N$, one has $h_\mu^-(G,\mathcal{U})\le f_\mathcal{U}$.
\end{proof}

\begin{lem} \label{refer}
Let $\mathcal{U}\in \mathcal{C}_X$ and $\mu\in \mathcal{M}(X,G)$
with $\mu=\int_X \mu_x d \mu(x)$ the ergodic decomposition of $\mu$.
Then
\begin{equation*}
h_\mu^- (G, \mathcal{U})= \int_{X} h_{\mu_x}^- (G, \mathcal{U}) d
\mu (x) \hbox{ and } h_\mu (G, \mathcal{U})= \int_{X} h_{\mu_x}
(G, \mathcal{U}) d \mu (x).
\end{equation*}
\end{lem}
\begin{proof}
With no loss of generality we assume $e_G\in F_1\subseteq F_2\subseteq \cdots$ (by
Lemma \ref{convergent}). Then we have
\begin{eqnarray*}
h_\mu^- (G, \mathcal{U})&= & \lim_{n\rightarrow +\infty}
\frac{1}{|F_n|} H_\mu (\mathcal{U}_{F_n}|I_\mu) \ \ \ \text{(by Proposition \ref{entr-imu})} \\
&= & \lim_{n\rightarrow +\infty} \frac{1}{|F_n|} \int_{X} H_{\mu_x}
(\mathcal{U}_{F_n}) d \mu(x)
\ \ \ (\text{by Lemma \ref{lem4-2}}) \\
&= & \int_{X} \lim_{n\rightarrow +\infty} \frac{1}{|F_n|}
H_{\mu_x} (\mathcal{U}_{F_n}) d \mu (x)\ \ \ (\text{by Dominant
Convergence Theorem}).
\end{eqnarray*}
That is, $h_\mu^- (G, \mathcal{U})= \int_{X} h_{\mu_x}^- (G,
\mathcal{U}) d \mu(x)$. In particular, if $\alpha\in \mathcal{P}_X$
then
\begin{align}\label{ergecmu}
h_\mu (G, \alpha)=\int_X h_{\mu_x}(G, \alpha) d \mu(x).
\end{align}

Next we follow the idea of the proof of \cite[Lemma 4.8]{HY} to prove $h_\mu
(G, \mathcal{U})= \int_{X} h_{\mu_x} (G, \mathcal{U}) d \mu(x)$.
Let $\mathcal{U}=\{U_1, \cdots, U_M\}$ and put
$\mathcal{U}^*=
\{\alpha= \{A_1, \cdots, A_M\}\in \mathcal{P}_X: A_m\subseteq U_m,
m=1, \cdots, M\}$.
 As $(X,\mathcal{B}_X)$ is a standard Borel
space, there exists a countable algebra $\mathcal{A}\subseteq
\mathcal{B}_X$ such that $\mathcal{B}_X$ is the $\sigma$-algebra
generated by $\mathcal{A}$. It is well known that if $\nu\in
\mathcal{M} (X)$ then
\begin{equation}
\label{fact11}
\mathcal{B}_X=\{ A\in \mathcal{B}_X: \forall \
\epsilon>0, \exists B\in \mathcal{A} \text{ such that }
\nu(A\Delta B)<\epsilon \}.
\end{equation}
Take $\mathcal{C}$ to be the countable algebra generated by
$\mathcal{A}$ and $\{ U_1,\cdots,U_M \}$, then $\mathcal{F}=\{
P\in \mathcal{U}^*: P\subseteq \mathcal{C} \}$ is a countable set
and for each $\alpha \in \mathcal{U}^*$, $\epsilon>0$ and $\nu\in
\mathcal{M} (X)$ there exists $\beta \in \mathcal{F}$ such that
$\nu(\alpha \Delta \beta)< \epsilon$ by \eqref{fact11}, i.e.
$\mathcal{F}$ is $L^1(X,{\mathcal B}_X,\nu)$-dense in
$\mathcal{U}^*$. In particular, say $\mathcal{F}= \{ \alpha_k:
k\in\mathbb{N} \}$ (denote $\alpha_k=\{A_1^k,\cdots,A_M^k\}$ for
each $k\in \mathbb{N}$), if $\nu\in \mathcal{M}(X,G)$ then
\begin{equation}
\label{dense} h_\nu(G,\mathcal{U})=\inf_{\alpha\in \mathcal{U}^*}
h_\nu(G,\alpha)=\inf_{k\in \mathbb{N}}h_\nu(G,\alpha_k).
\end{equation}

First, for one inequality one has
\begin{eqnarray*}
h_\mu (G, \mathcal{U})& =& \inf_{k\in \mathbb{N}} h_\mu (G,
\alpha_k)= \inf_{k\in \mathbb{N}} \int_{X} h_{\mu_x} (G, \alpha_k)
d \mu(x))\ \ \
(\text{by \eqref{ergecmu}}) \\
&\ge & \int_{X} \inf_{k\in \mathbb{N}} h_{\mu_x} (G, \alpha_k) d
\mu(x)= \int_{X} h_{\mu_x} (G, \mathcal{U}) d \mu (x)\ \ (\text{by
\eqref{dense}}).
\end{eqnarray*}
For the other inequality, let $\epsilon>0$. For each $n\in
\mathbb{N}$ define
$B_n^\epsilon= \{ x\in X: h_{\mu_x}(G, \alpha_n)< h_{\mu_x} (G,
\mathcal{U})+ \epsilon\}$.
Then $B_n^\epsilon$ is $G$-invariant and $\mu(\bigcup_{n\in
\mathbb{N}} B_n^\epsilon)= 1$ by \eqref{dense}, and so there
exists a measurable partition $\{ X_n: n\in \mathbb{N}\}$ of $X$
with $X_n\in I_\mu$ and $\mu(X_n)> 0$, and a sequence
$\{\alpha_{k_n}\}_{n\in \mathbb{N}}$ such that for each $n\in \N$
and $\mu$-a.e. $x\in X_n$ one has $h_{\mu_x} (G, \alpha_{k_n})<
h_{\mu_x} (G, \mathcal{U})+ \epsilon$. For every $n\in\mathbb{N}$
we define $\mu_n(\cdot)=\frac{1}{\mu(X_n)} \int_{X}
\mu_x(\cdot\cap X_n) \,d \mu(x)\in\mathcal{M}(X,G)$. We deduce
\begin{align*}
h_{\mu_n}(G,\alpha_{k_n})&=\frac{1}{\mu(X_n)} \int_{X_n}
h_{\mu_x}(G,\alpha_{k_n})\,d \mu(x)\nonumber\ \ \  (\text{by
\eqref{ergecmu}})\\
&\le \frac{1}{\mu(X_n)} \int_{X_n} h_{\mu_x}(G,\mathcal{U})\,d
\mu(x)+\epsilon.
\end{align*}
Note that, by definition, for every $n\in\mathbb{N}$, $\mu_n(X_n)=1$
and $\mu_n(X_k)=0$ if $k\not=n$. For $m\in\{1,\cdots,M\}$ define
$A_m=\bigcup_{n\in \mathbb{N}}(X_n\cap A_m^{k_n})$, then $\alpha=\{
A_1,\cdots,A_M \} \in\mathcal{U}^*$. We get,
\begin{eqnarray*}
h_\mu(G,\mathcal{U})&\le & h_\mu(G,\alpha)= \sum\limits_{n\in
\mathbb{N}} \mu(X_n) h_{\mu_n}(G,\alpha)\ \ \ \text{(by
\eqref{ergecmu})}\\
&= & \sum\limits_{n\in \mathbb{N}} \mu(X_n) h_{\mu_n}(G,\alpha_{k_n})
\le \int\limits_X h_{\mu_x}(G,\mathcal{U})\,d\mu(x)+\epsilon.
\end{eqnarray*}
Letting $\epsilon\rightarrow 0+$ we conclude $h_\mu(G,\mathcal{U})
\le \int\limits_X h_{\mu_x}(G,\mathcal{U})\,d\mu(x)$ and the desired
equality holds.
\end{proof}

Denote by $C (X; \R)$ the Banach space of the set of all
continuous real-valued functions on $X$ equipped with the maximal
norm $||\cdot||$. Note that the Banach space $C (X; \R)$ is
separable, let $\{f_n: n\in \N\}\subseteq C (X; \R)\setminus \{0\}$ be
a countable dense subset, where $0$ is the constant $0$ function on $X$, then a compatible metric on ${\mathcal
M} (X)$ is given by
\begin{equation*} \rho (\mu, \nu)= \sum_{n\in \N}
\frac{|\int_X f_n d \mu- \int_X f_n d \nu|}{2^n ||f_n||}, \text{ for
each }\mu, \nu\in {\mathcal M} (X).
\end{equation*}
Let $\mu\in \mathcal{M}(X,G)$ with $\mu=\int_X \mu_x d \mu(x)$ the
ergodic decomposition of $\mu$. Then there exists a $G$-invariant
subset $X_0\subseteq X$ with $\mu(X_0)=1$ such that the map $\Phi:
X_0\rightarrow \mathcal{M}^e(X,G)$ with $\Phi(x)= \mu_x$ is
well defined. We extend $\Phi$ to the whole space $X$ such that
$\Phi(x)\in \mathcal{M}^e(X,G)$ for each $x\in X$. For any $g_i\in
C(X;\mathbb{R})$, $\mu_i\in \mathcal{M}(X,G)$ and $\epsilon_i>0$,
$i=1,\cdots,k$, note that for any $f\in C(X;\mathbb{R})$, the
function $x\in X_0 \mapsto \int_X f d\mu_x$ is an element of
$L^1(X,I_\mu,\mu)$, we have $\Phi^{-1}(\bigcap_{i=1}^k \{\nu \in
\mathcal{M}(X,G): |\int_X g_i d \nu-\int_X g_i d
\mu_i|<\epsilon_i\})\in I_\mu$. Since all the sets having the form
of $\bigcap_{i=1}^k \{\nu \in \mathcal{M}(X,G): |\int_X g_i d
\nu-\int_X g_i d \mu_i|<\epsilon_i\}$ form a topological base of
$\mathcal{M}(X,G)$, the map $\Phi:(X,I_\mu)\rightarrow
(\mathcal{M}(X,G),\mathcal{B}_{\mathcal{M}(X,G)})$ is measurable,
i.e. $\Phi^{-1}(A)\in I_\mu$ for any $A\in
\mathcal{B}_{\mathcal{M}(X,G)}$. Now we define $m\in \mathcal{M}
(\mathcal{M}(X,G))$ as following: $m(A)=\mu(\Phi^{-1}(A))$ for any
$A\in \mathcal{B}_{\mathcal{M}(X,G)}$. Then if $g$ is a bounded
Borel function on $\mathcal{M}(X,G)$ then $g\circ \Phi\in
L^1(X,I_\mu,\mu)$ and
\begin{align}\label{eq-1111}
\int_X g\circ \Phi(x) d \mu(x)=\int_{\mathcal{M}(X,G)}g(\theta) d
m(\theta).
\end{align}
Now if $f\in C(X;\mathbb{R})$, let $L_f:\theta\in
\mathcal{M}(X,G)\mapsto \int_X f d \theta$, then $L_f$ is a
continuous function, and so
$$\int_X \left(\int_X f d \mu_x \right) \, d \mu(x)=\int_X L_f \circ \Phi(x) d \mu(x)=
\int_{\mathcal{M}(X,G)} L_f(\theta) d m(\theta)\ (\text{using}\
\eqref{eq-1111}),$$ moreover,
\begin{align}\label{fj-er}
\int_X f(x) d \mu(x)=\int_{\mathcal{M}(X,G)} \left(\int_X f(x) d
\theta(x) \right) \, d m(\theta) \text{ for any }f\in
C(X;\mathbb{R})\ (\text{using}\ \eqref{07-03-06-04}).
\end{align}
Note that $m(\mathcal{M}^e(X,G))\ge \mu(X_0)=1$, $m$ can be viewed
as a Borel probability measure on $\mathcal{M}^e(X,G)$. So
\eqref{fj-er} can also be written as
\begin{align}\label{fj-er1}
\int_X f(x) d \mu(x)=\int_{\mathcal{M}^e(X,G)}  \left( \int_X f(x)
d \theta(x) \right) \, d m(\theta) \text{ for any }f\in
C(X;\mathbb{R}),
\end{align}
which is denoted by $\mu=\int_{\mathcal{M}^e(X,G)} \theta d
m(\theta)$ (also called the {\it ergodic decomposition of $\mu$}).
Finally, it is not hard to check that if $m'$ is another Borel
probability measure on $\mathcal{M}(X,G)$ satisfying
$m'(\mathcal{M}^e(X,G))=1$ and \eqref{fj-er1} then $m'=m$. That is,
for any given $\mu\in \mathcal{M}(X,G)$ there exists uniquely a
Borel probability measure $m'$ on $\mathcal{M}(X,G)$ with $m'
(\mathcal{M}^e(X,G))=1$ satisfying \eqref{fj-er1}.

\begin{thm} \label{en-decom}
Let $\mathcal{U}\in \mathcal{C}_X$. Then the function $\eta\in
\mathcal{M} (X, G)\mapsto h_\eta(G,\mathcal{U})$ and the function
$\eta\in \mathcal{M} (X, G)\mapsto h_\eta^{-}(G,\mathcal{U})$ are
both bounded affine Borel functions on $\mathcal{M}(X,G)$.
Moreover, if we let $\mu\in \mathcal{M} (X, G)$ with $\mu=
\int_{\mathcal{M}^e(X,G)} \theta \, d m(\theta)$ the ergodic
decomposition of $\mu$, then
\begin{equation}\label{lab-eq}
h_\mu(G,\mathcal{U})=\int_{\mathcal{M}^e(X,G)}
h_\theta(G,\mathcal{U})\, d m(\theta) \text{ and
}h_\mu^-(G,\mathcal{U})=\int_{\mathcal{M}^e(X,G)}
h^-_\theta(G,\mathcal{U})\, d m(\theta).
\end{equation}
\end{thm}
\begin{proof}
First we aim to establish \eqref{lab-eq}. Similar to the proof of
Lemma \ref{refer}, there exists $\{ \alpha_k\}_{k\in
\mathbb{N}}\subseteq \mathcal{P}_X$ such that $\alpha_k\succeq
\mathcal{U}$ for each $k\in \N$ and $H_\eta (\mathcal{U})=
\inf_{k\in \mathbb{N}} H_\eta (\alpha_k), h_\eta (G,\mathcal{U})=
\inf_{k\in \mathbb{N}} h_\eta (G, \alpha_k)$ for each $\eta\in
\mathcal{M}(X,G)$. Note that, for any $A\in \mathcal{B}_X$, the
function $\eta\in \mathcal{M}(X,G)\mapsto \eta(A)$ is Borel
measurable and hence if $\alpha\in \mathcal{P}_X$ then the
function $\eta\in \mathcal{M}(X,G)\mapsto H_\eta(\alpha)$ and the
function $\eta\in \mathcal{M}(X,G)\mapsto h_\eta(G,\alpha)$ are
both Bounded Borel functions. Moreover, the function $\eta\in
\mathcal{M}(X,G)\mapsto H_\eta(\mathcal{U})$ is a bounded Borel
function. Thus, the function $\eta\in \mathcal{M} (X, G)\mapsto
h_\eta(G,\mathcal{U})$ and the function $\eta\in \mathcal{M} (X,
G)\mapsto h_\eta^{-}(G,\mathcal{U})$ are both bounded Borel
functions. In particular, \eqref{lab-eq} follows directly from
Lemma \ref{refer} and \eqref{eq-1111}.

Now let $\mu_1,\mu_2\in \mathcal{M}(X,G)$ and $\lambda\in (0,1)$.
For $i= 1, 2$, let $\mu_i=\int_{\mathcal{M}^e(X,T)} \theta d
m_i(\theta)$ be the ergodic decomposition of $\mu_i$, where $m_i$ is
a Borel propobility measure on $\mathcal{M}^e(X,G)$. Consider
$\mu=\lambda \mu_1+(1-\lambda)\mu_2$ and $m=\lambda m_1
+(1-\lambda)m_2$. Then $m$ is a Borel probability measure on
$\mathcal{M}^e(X,G)$ and $\mu=\int_{\mathcal{M}^e(X,G)} \theta d
m(\theta)$ is the ergodic decomposition of $\mu$. By \eqref{lab-eq},
we have
\begin{eqnarray*}
h_\mu(G,\mathcal{U})&= & \int_{\mathcal{M}^e(X,G)}
h_{\theta}(G,\mathcal{U})
d m(\theta) \\
&= & \lambda \int_{\mathcal{M}^e(X,G)} h_{\theta}(G,\mathcal{U}) d
m_1(\theta)+(1-\lambda) \int_{\mathcal{M}^e(X,G)}
h_{\theta}(G,\mathcal{U}) d m_2(\theta) \\
&= & \lambda
h_{\mu_1}(G,\mathcal{U})+(1-\lambda)h_{\mu_2}(G,\mathcal{U}).
\end{eqnarray*}
This shows the affinity of $h_{\{\cdot \}}(G,\mathcal{U})$. We can
obtain similarly the affinity of $h^-_{\{\cdot\}}(G,\mathcal{U})$.
\end{proof}

\section{The equivalence of measure-theoretic entropy of covers}

In the section,  following arguments of Danilenko in
\cite{D}, we will develop an orbital approach to local entropy
theory for actions of an amenable group. Then combining it with the
equivalence of measure-theoretic entropy of covers in the case of
$G= \Z$, we will establish the equivalence of those two kinds of
measure-theoretic entropy of covers for a general $G$.

\subsection{Backgrounds of orbital theory}

Let $(X,\mathcal{B}_X,\mu)$ be a Lebesgue space. Denote by
$Aut(X,\mu)$ the group of all $\mu$-measure preserving invertible
transformations of $(X, \mathcal{B}_X, \mu)$, which is endowed
with the weak topology, i.e. the weakest topology which makes
continuous the following unitary representation: $Aut(X,\mu)\ni \gamma\mapsto U_{\gamma}\in \mathcal{U}(L^2(X,\mu))\
\text{with}\ U_{\gamma} f=f\circ \gamma^{-1},$ where the unitary
group $\mathcal{U}(L^2(X,\mu))$ is the set of all unitary
operators on $L^2(X,\mu)$ endowed with the strong operator
topology. Let a Borel subset $\mathcal{R}\subseteq X\times X$ be
an equivalence relation on $X$. For each $x\in X$, we denote
$\mathcal{R}(x)= \{y\in X: (x, y)\in \mathcal{R}\}$. Following
\cite{FM}, $\mathcal{R}$ is called {\it measure preserving} if it
is generated by some countable sub-group $G\subseteq Aut(X,\mu)$,
in general, this generating sub-group is highly non-unique;
$\mathcal{R}$ is {\it ergodic} if $A$ belongs to the trivial
sub-$\sigma$-algebra of $\mathcal{B}_X$ when $A\in \mathcal{B}_X$
is $\mathcal{R}$-invariant (i.e. $A= \bigcup_{x\in A} \mathcal{R}
(x)$); $\mathcal{R}$ is {\it discrete} if $\# \mathcal{R}(x)\le \#
\mathbb{Z}$ for $\mu$-a.e. $x\in X$; $\mathcal{R}$ is of {\it type
I} if $\# \mathcal{R}(x)< +\infty$ for $\mu$-a.e. $x\in X$,
equivalently, there is a subset $B\in \mathcal{B}_X$ with $\# (B
\cap \mathcal{R}(x))= 1$ for $\mu$-a.e. $x\in X$, such a $B$ is
called a {\it $\mathcal{R}$-fundamental domain}; $\mathcal{R}$ is
{\it countable} if $\# \mathcal{R}(x)= +\infty$ for $\mu$-a.e.
$x\in X$, observe that if $\mathcal{R}$ is measure preserving then
it is countable iff it is {\it conservative}, i.e.
$\mathcal{R}\cap (B\times B)\setminus \Delta_2 (X)\neq \emptyset$
for each $B\in \mathcal{B}_X$ satisfying $\mu(B)>0$, where
$\Delta_2(X)=\{ (x,x):x\in X\}$; $\mathcal{R}$ is {\it
hyperfinite} if there exists a sequence $\mathcal{R}_1\subseteq
\mathcal{R}_2\subseteq \cdots$ of type I sub-relations of
$\mathcal{R}$ such that $\bigcup_{n\in \mathbb{N}}
\mathcal{R}_n=\mathcal{R}$, the sequence $\{\mathcal{R}_n\}_{n\in
\mathbb{N}}$ is called a {\it filtration} of $\mathcal{R}$. Note
that a measure preserving discrete equivalence relation is
hyperfinite iff it is generated by a single invertible
transformation \cite{Dye}, the orbit equivalence relation of a
measure preserving action of a countable discrete amenable group
is hyperfinite \cite{CFW, Zim}, any two ergodic hyperfinite measure
preserving countable equivalence relations are isomorphic in the
natural sense (i.e. there exists an isomorphism between the
Lebesgue spaces which intertwines the corresponding equivalent
classes) \cite{Dye}. Everywhere below $\mathcal{R}$ is a measure
preserving discrete equivalence relation on a Lebesgue space
$(X,\mathcal{B}_X,\mu)$.

The {\it full group} $[\mathcal{R}]$ of $\mathcal{R}$ and its {\it
normalizer} $N[\mathcal{R}]$ are defined, respectively, by
$$[\mathcal{R}]= \{ \gamma\in Aut(X,\mu): (x,\gamma x)\in
\mathcal{R}\ \text{for $\mu$-a.e. $x\in X$}\},$$
$$N[\mathcal{R}]= \{ \theta \in Aut(X,\mu): \theta \mathcal{R}(x)=
\mathcal{R}(\theta x)\ \text{for $\mu$-a.e. $x\in X$}\}.$$ Let $A$
be a Polish group. A Borel map $\phi: \mathcal{R}\rightarrow A$ is
called a {\it cocycle} if
$$\phi(x,z)=\phi(x,y)\phi(y,z) \text{ for all }(x,y),(y,z)\in \mathcal{R}.$$
Let $\theta\in N[\mathcal{R}]$, we define a cocycle $\phi\circ
\theta$ by setting $\phi \circ \theta (x,y)=\phi(\theta x,\theta
y)$ for all $(x, y)\in \mathcal{R}$. Let $(Y,\mathcal{B}_Y,\nu)$
be another Lebesgue space and $A$ be embedded continuously into
$Aut(Y,\nu)$. For each cocycle $\phi:\mathcal{R}\rightarrow A$, we
associate a measure preserving discrete equivalence relation
$\mathcal{R}(\phi)$ on $(X\times Y, \mathcal{B}_X\times
\mathcal{B}_Y, \mu\times \nu)$ by setting
$(x,y)\sim_{\mathcal{R}(\phi)} (x',y')$ if $(x,x')\in \mathcal{R}$
and $y'=\phi(x',x)y$. Then an one-to-one group homomorphism
$[\mathcal{R}]\ni \gamma\mapsto \gamma_\phi\in [\mathcal{R}_\phi]$
is well defined via the formula
$$\gamma_\phi (x,y)=(\gamma x, \phi(\gamma x,x)y) \text{ for each }(x,y)\in X\times Y.$$
The transformation $\gamma_\phi$ is called the {\it $\phi$-skew
product extension} of $\gamma$, and the equivalence relation
$\mathcal{R}(\phi)$ is called the {\it $\phi$-skew product
extension} of $\mathcal{R}$.

\subsection{Local entropy for a cocycle of a
discrete measure preserving equivalence relation}

Denote by $I (\mathcal{R})$ the set of all type I sub-relations of
$\mathcal{R}$. Let $\epsilon>0$ and $\mathcal{T}, \mathcal{S}\in I
(\mathcal{R})$. We write $\mathcal{T}\subseteq_{\epsilon}
\mathcal{S}$ if there is $A\in \mathcal{B}_X$ such that
$\mu(A)>1-\epsilon$ and
$$\# \{ y\in \mathcal{S}(x): \mathcal{T}(y)\subseteq
\mathcal{S}(x)\}>(1-\epsilon) \# \mathcal{S}(x) \text{ for each }
x\in A.
$$
Replacing, if necessity, $A$ by $\bigcup_{x\in A} \mathcal{S}(x)$
we may (and so shall) assume that $A$ is $\mathcal{S}$-invariant.
Let $A_0=\{ x\in A: \mathcal{T}(x)\subseteq \mathcal{S}(x)\}$. The
following two lemmas are proved in \cite{D}.

\begin{lem}\label{lem-c-1}
$A_0$ is $\mathcal{T}$-invariant, $\mu(A_0)>1-2\epsilon$ and $\#
(\mathcal{S}(x)\cap A_0)>(1-\epsilon)\# \mathcal{S}(x)$ for each
$x\in A_0$.
\end{lem}

\begin{lem} \label{lem-c-2}
Let $\epsilon> 0$ and $\mathcal{R}$ be hyperfinite with
$\{\mathcal{R}_n\}_{n\in \mathbb{N}}$ a filtration of
$\mathcal{R}$.
\begin{description}

\item[1] If $\Gamma\subseteq [\mathcal{R}]$ is a countable subset
satisfying $\# (\Gamma x)< +\infty$ for $\mu$-a.e. $x\in X$ then
for each sufficiently large $n$ there is a
$\mathcal{R}_n$-invariant subset $A_n$ such that
$\mu(A_n)>1-\epsilon$ and
$$\# \{ y\in \mathcal{R}_n(x):\Gamma y\subseteq \mathcal{R}_n(x)\}
>(1-\epsilon) \# \mathcal{R}_n(x) \text{ for each } x\in A_n.$$

\item[2] If $\mathcal{S}\in I (\mathcal{R})$ then
$\mathcal{S}\subseteq_\epsilon \mathcal{R}_n$ if $n$ is large
enough.
\end{description}
\end{lem}

Let $(Y,\mathcal{B}_Y,\nu)$ be a Lebesgue space and
$\phi:\mathcal{R}\rightarrow Aut(Y,\nu)$ a cocycle. For
$\mathcal{U}\in \mathcal{C}_{X\times Y}$, we consider
$\mathcal{U}$ as a measurable field $\{\mathcal{U}_x\}_{x\in
X}\subseteq \mathcal{C}_Y$, where $\{x\}\times
\mathcal{U}_x=\mathcal{U}\cap (\{x\}\times Y)$.

\begin{de} For
$\mathcal{U}\in \mathcal{C}_{X\times Y}$, we define
\begin{align*}
h^-_\nu(\mathcal{S},\phi,\mathcal{U})= \int_X \frac{1}{\#
\mathcal{S}(x)} H_\nu \left(\bigvee_{y\in
\mathcal{S}(x)}\phi(x,y)\mathcal{U}_y\right) d\mu(x)\text{ and }
h_\nu(\mathcal{S},\phi,\mathcal{U})= \inf_{\alpha\in
\mathcal{P}_{X\times Y}: \alpha\succeq
\mathcal{U}}h_\nu^{-}(\mathcal{S},\phi,\alpha).
\end{align*}
Then we define the {\it $\nu^-$-entropy $h^-_\nu
(\phi,\mathcal{U})$} and the {\it $\nu$-entropy $h_\nu
(\phi,\mathcal{U})$} of $(\phi,\mathcal{U})$, respectively, by
\begin{equation*}
h^-_\nu (\phi,\mathcal{U})=\inf_{\mathcal{S}\in I (\mathcal{R})}
h_\nu^- (\mathcal{S},\phi,\mathcal{U})\text{ and
}h_\nu(\phi,\mathcal{U})= \inf_{\mathcal{S}\in I (\mathcal{R})}
h_\nu (\mathcal{S},\phi,\mathcal{U}).
\end{equation*}
\end{de}

It is clear that if $\beta\in \mathcal{P}_{X\times Y}$ and
$\mathcal{U}\in \mathcal{C}_{X\times Y}$ then
$h_\nu(\mathcal{S},\phi,\beta)=h_\nu^{-}(\mathcal{S},\phi,\beta)$,
$h_\nu(\phi,\beta)=h_\nu^{-}(\phi,\beta)$ and
$h_\nu(\phi,\mathcal{U})=\inf_{\alpha\in \mathcal{P}_{X\times Y}:
\alpha\succeq \mathcal{U}} h_\nu^{-}(\phi,\alpha)$. Moreover, if
$\mathcal{U}, \mathcal{V}\in \mathcal{C}_{X\times Y}$ satisfy
$\mathcal{U}\succeq \mathcal{V}$ then
$h_\nu(\mathcal{S},\phi,\mathcal{U})\ge
h_\nu(\mathcal{S},\phi,\mathcal{V})$ and
$h^-_\nu(\mathcal{S},\phi,\mathcal{U})\ge
h^-_\nu(\mathcal{S},\phi,\mathcal{V})$. It's not hard to obtain

\begin{prop} \label{pro-conclusion}
Let $(Z, \mathcal{B}_Z, \kappa)$ be a Lebesgue space,
$\mathcal{S}\in I(\mathcal{R})$, $\beta: \mathcal{S}\rightarrow
Aut (Z, \kappa)$ a cocycle and $\sigma: Z\times X\rightarrow
X\times Z, (z, x)\mapsto (x, z)$ the flip.
\begin{description}

\item[1] Let $\alpha':
\sigma^{- 1} \mathcal{S} (\beta) \sigma\rightarrow Aut (Y, \nu)$
and $\alpha: \mathcal{S}\rightarrow Aut (Y, \nu)$ be cocycles
satisfying $\alpha' ((z, x), (z', x'))= \alpha (x, x')$ when
$((z, x), (z', x'))\in \sigma^{- 1} \mathcal{S} (\beta) \sigma$.
Then $h_\nu^- (\sigma^{- 1} \mathcal{S} (\beta) \sigma, \alpha',
Z\times \mathcal{U})= h_\nu^- (\mathcal{S}, \alpha, \mathcal{U})$
for any $\mathcal{U}\in \mathcal{C}_{X\times Y}$.

\item[2] Let $\alpha'':
\mathcal{S} (\beta) \rightarrow Aut (Y, \nu)$ and $\alpha:
\mathcal{S}\rightarrow Aut (Y, \nu)$ be cocycles satisfying
$\alpha'' ((x, z), (x'', z''))$ $= \alpha (x, x'')$ when $((x, z),
(x'', z''))\in \mathcal{S} (\beta)$. Then if $\mathcal{U}''\in
\mathcal{C}_{X\times Z\times Y}$ and $\mathcal{U}\in
\mathcal{C}_{X\times Y}$ satisfies $\mathcal{U}''_{(x, z)}=
\mathcal{U}_x$ for each $(x, z)\in X\times Z$ then $h_\nu^-
(\mathcal{S} (\beta), \alpha'', \mathcal{U}'')= h_\nu^-
(\mathcal{S}, \alpha, \mathcal{U})$.
\end{description}
\end{prop}
\begin{proof}
As the proof is similar, we only present the proof for 1. Let
$\mathcal{U}\in \mathcal{C}_{X\times Y}$. Then
\begin{eqnarray*}
& & h_\nu^- (\sigma^{- 1} \mathcal{S} (\beta) \sigma, \alpha',
Z\times \mathcal{U}) \\
&= & \int_{Z\times X} \frac{1}{\# \sigma^{- 1} \mathcal{S} (\beta)
\sigma (z, x)} H_\nu \left(\bigvee_{(z', x')\in \sigma^{- 1}
\mathcal{S} (\beta) \sigma (z, x)} \alpha' ((z, x), (z', x'))
(Z\times \mathcal{U})_{(z',x')} \right) d \kappa\times \mu (z, x) \\
&= & \int_{Z\times X} \frac{1}{\# \mathcal{S} (x)} H_\nu
\left(\bigvee_{(x', z')\in \mathcal{S} (\beta) (x, z)} \alpha (x,
x')
\mathcal{U}_{x'}\right) d \kappa\times \mu (z, x) \\
&= & \int_X \frac{1}{\# \mathcal{S} (x)} H_\nu
\left(\bigvee_{x'\in \mathcal{S} (x)} \alpha (x, x')
\mathcal{U}_{x'}\right) d \mu (x)= h_\nu^- (\mathcal{S}, \alpha,
\mathcal{U}).
\end{eqnarray*}
\end{proof}

\begin{prop} \label{pro-c-5}
Let $\epsilon> 0$ and $\mathcal{T}, \mathcal{S}\in I
(\mathcal{R})$. If $\mathcal{T}\subseteq_{\epsilon} \mathcal{S}$
then
\begin{align*}
h_{\nu}^- (\mathcal{S},\phi,\mathcal{U})\le
h_\nu^-(\mathcal{T},\phi,\mathcal{U})+3 \epsilon \log
N(\mathcal{U}) \text{ and } h_\nu(\mathcal{S},\phi,\mathcal{U})\le
h_\nu(\mathcal{T},\phi,\mathcal{U})+3 \epsilon \log
N(\mathcal{U}).
\end{align*}
In particular, if $\mathcal{T}\subseteq \mathcal{S}$ then
$h_{\nu}^- (\mathcal{S},\phi,\mathcal{U})\le
h_\nu^-(\mathcal{T},\phi,\mathcal{U})$ and
$h_\nu(\mathcal{S},\phi,\mathcal{U})\le
h_\nu(\mathcal{T},\phi,\mathcal{U})$.
\end{prop}
\begin{proof} The proof follows the arguments of the proof of
\cite[Proposition 2.6]{D}. Let $A_0=\{ x\in A:
\mathcal{T}(x)\subseteq \mathcal{S}(x)\}$. Then
$\mu(A_0)>1-2\epsilon$ by Lemma \ref{lem-c-1}. We define the maps
$f, g: A_0\rightarrow \mathbb{R}$ by
\begin{equation*}
f(x)=\frac{1}{\# (\mathcal{S}(x)\cap A_0)}
H_\nu\left(\bigvee_{y\in \mathcal{S}(x)\cap A_0}
\phi(x,y)\mathcal{U}_y\right)\ \text{and}\ g(x)=\frac{1}{\#
\mathcal{T}(x)} H_\nu\left(\bigvee_{y\in \mathcal{T}(x)} \phi(x,y)
\mathcal{U}_y\right).
\end{equation*}
Since $A_0$ is $\mathcal{T}$-invariant, for each $x\in A_0$ there
are $x_1,\cdots,x_k\in X$ such that $\mathcal{S}(x)\cap
A_0=\bigsqcup_{i=1}^k \mathcal{T}(x_i)$, here the sign $\bigsqcup$
denotes the union of disjoint subsets. It follows that
\begin{eqnarray*}
f(x)&\le & \frac{1}{\# (\mathcal{S}(x)\cap A_0)} \sum_{i=1}^k
H_\nu\left(\phi(x, x_i) \bigvee_{y\in \mathcal{T}(x_i)}
\phi(x_i,y)
\mathcal{U}_y\right)\\
&= &\frac{1}{\# (\mathcal{S}(x)\cap A_0)}\sum_{i=1}^k \#
\mathcal{T}(x_i)\cdot g(x_i)=\frac{1}{\# (\mathcal{S}(x)\cap
A_0)}\sum_{i=1}^k \sum_{y\in
\mathcal{T}(x_i)} g(y)\\
&= &\frac{1}{\# (\mathcal{S}(x)\cap A_0)} \sum_{z\in
\mathcal{S}(x)\cap A_0} g(z) =\mathbb{E}(g|\mathcal{S}\cap
(A_0\times A_0))(x),
\end{eqnarray*}
where $\mathbb{E}(g|\mathcal{S}\cap (A_0\times A_0))$ denotes the
conditional expectation of $g$ w.r.t. $\mathcal{S}_{A_0}$, the
$\sigma$-algebra of all measurable $\mathcal{S}\cap (A_0\times
A_0)$-invariant subsets. Hence
\begin{eqnarray*}
h^-_\nu(\mathcal{S},\phi,\mathcal{U}) &= &\int_X \frac{1}{\#
\mathcal{S}(x)} H_\nu(\bigvee_{y\in \mathcal{S}(x)} \phi(x,y)
\mathcal{U}_y) d \mu(x)\\
&\le &\int_{A_0} \frac{1}{\# \mathcal{S}(x)} H_\nu(\bigvee_{y\in
\mathcal{S}(x)} \phi(x,y) \mathcal{U}_y) d \mu(x)+\int_{X\setminus
A_0}  \frac{1}{\# \mathcal{S}(x)}\sum_{y\in \mathcal{S}(x)}
H_\nu(\mathcal{U}_y) d \mu(x) \\
&\le &\int_{A_0} \left( f(x)+\frac{1}{\# \mathcal{S}(x)}H_\nu
(\bigvee_{y\in \mathcal{S}(x)\setminus A_0} \phi(x,y)\mathcal{U}_y) \right)d\mu(x)+\int_{X\setminus A_0}  \log N(\mathcal{U}) d\mu(x) \\
&\le &\int_{A_0} \left( \mathbb{E}(g|\mathcal{S}\cap (A_0\times
A_0))(x)+\frac{\#(\mathcal{S}(x)\setminus A_0)}{\#
\mathcal{S}(x)}\log N(\mathcal{U}) \right)
d \mu(x)+ 2\epsilon \log N(\mathcal{U})\\
&\le &\int_{A_0} \mathbb{E}(g|\mathcal{S}\cap (A_0\times A_0))(x) d
\mu(x)+ 3\epsilon \log N(\mathcal{U})\\
&=&\int_{A_0} g(x) d \mu(x)
+ 3\epsilon \log N(\mathcal{U})\le
h^-_\nu(\mathcal{T},\phi,\mathcal{U})+ 3\epsilon \log
N(\mathcal{U}).
\end{eqnarray*}
By the same reason, one has $h_\nu(\mathcal{S},\phi,\alpha)\le
h_\nu(\mathcal{T},\phi,\alpha)+ 3\epsilon \log N(\alpha)$ for
any $\alpha\in \mathcal{P}_{X\times Y}$. Thus
\begin{eqnarray*}
h_\nu(\mathcal{S},\phi,\mathcal{U})&=&\inf \{
h_\nu(\mathcal{S},\phi,\alpha):\alpha\in \mathcal{P}_{X\times Y}
\text{ with }\alpha\succeq \mathcal{U}, N(\alpha)\le N(\mathcal{U})\} \\
&\le &\inf \{ h_\nu(\mathcal{T},\phi,\alpha)+3\epsilon \log
N(\alpha):\alpha\in \mathcal{P}_{X\times Y} \text{ with
}\alpha\succeq \mathcal{U}, N(\alpha)\le N(\mathcal{U})\} \\
&\le &\inf \{ h_\nu(\mathcal{T},\phi,\alpha)+ 3\epsilon \log
N(\mathcal{U}): \alpha\in \mathcal{P}_{X\times Y}
\text{ with }\alpha\succeq \mathcal{U}, N(\alpha)\le N(\mathcal{U})\}\\
&= &h_\nu(\mathcal{T},\phi,\mathcal{U})+ 3\epsilon \log
N(\mathcal{U}).
\end{eqnarray*}

Now if $\mathcal{T}\subseteq \mathcal{S}$ then
$\mathcal{T}\subseteq_{\epsilon} \mathcal{S}$ for each $\epsilon>
0$, so letting $\epsilon\rightarrow 0+$ we have $h_{\nu}^-
(\mathcal{S},\phi,\mathcal{U})\le
h_\nu^-(\mathcal{T},\phi,\mathcal{U})$ and
$h_\nu(\mathcal{S},\phi,\mathcal{U})\le
h_\nu(\mathcal{T},\phi,\mathcal{U})$. This finishes the proof.
\end{proof}

As a direct application of Lemma \ref{lem-c-2} (2) and Proposition
\ref{pro-c-5} we have

\begin{cor} \label{pro-c-5-1}
Let $\mathcal{R}$ be hyperfinite with $\{\mathcal{R}_n\}_{n\in
\mathbb{N}}$ a filtration of $\mathcal{R}$. Then
\begin{equation*}
\lim_{n\rightarrow +\infty} h_\nu(\mathcal{R}_n,\phi,\mathcal{U})=
h_\nu(\phi,\mathcal{U}) \text{ and } \lim_{n\rightarrow +\infty}
h^-_\nu(\mathcal{R}_n,\phi,\mathcal{U})=
h^-_\nu(\phi,\mathcal{U}).
\end{equation*}
\end{cor}

\subsection{Two kinds virtual entropy of covers}

Everywhere below, $\mathcal{R}$ is generated by a free $G$-measure
preserving system $(X, \mathcal{B}_X, \mu, G)$. Then $\mathcal{R}$
is hyperfinite and conservative. Let $\mathcal{S}\in I
(\mathcal{R})$ with $B\subseteq X$ a $\mathcal{S}$-fundamental
domain. Then there is a measurable map $B \ni x\mapsto G_x\in
F(G)$ with $G_x x=\mathcal{S}(x)$ and hence $X=\bigsqcup_{x\in B}
G_x x$. Note that $F(G)$ is a countable set, we obtain that
$X=\bigsqcup_i \bigsqcup_{g\in G_i} gB_i$ for a countable family
$\{G_i\}_i\subseteq F(G)$ and a decomposition $B=\bigsqcup_{i}
B_i$ with $G_ix=\mathcal{S}(x)$ for each $x\in B_i$. We shall
write it as $\mathcal{S}\sim (B_i,G_i)$. Then
\begin{eqnarray} \label{eq-2858}
h^-_\nu(\mathcal{S},\phi,\mathcal{U}) &= & \sum_{i} \sum_{g\in
G_i} \int_{gB_i} \frac{1}{\# \mathcal{S} (x)}
H_\nu\left(\bigvee_{y\in \mathcal{S}(x)}
\phi(x,y) \mathcal{U}_y\right) d \mu(x)\ \nonumber \\
&= & \sum_{i} \sum_{g\in G_i} \int_{B_i} \frac{1}{|G_i|} H_\nu
\left(\bigvee_{g'\in G_i}
\phi(g x, g' x) \mathcal{U}_{g' x}\right) d \mu(x)\ \nonumber \\
&= & \sum_{i} \sum_{g\in G_i} \int_{B_i} \frac{1}{|G_i|} H_\nu
\left(\bigvee_{g'\in G_i}
\phi(x, g' x) \mathcal{U}_{g' x}\right) d \mu(x)\ \nonumber \\
&= & \sum_i \int_{B_i} H_\nu\left(\bigvee_{g\in G_i}
\phi(x,gx)\mathcal{U}_{gx}\right) d \mu(x).
\end{eqnarray}

\begin{de}\label{de-1943} Let $(Y,\mathcal{B}_Y,\nu,G)$ be a $G$-measure preserving
system, $\mathcal{U}\in \mathcal{C}_Y$, $\Pi_g\in Aut(Y,\nu)$ the
action of $g\in G$ on $(Y, \mathcal{B}_Y, \nu)$ and
$\phi_G:\mathcal{R}\rightarrow Aut(Y,\nu)$ a cocycle given by
$\phi_G(gx,x)=\Pi_g$ for any $x\in X, g\in G$. The {\it
$\nu^-$-virtual entropy} and {\it $\nu$-virtual entropy} of
$\mathcal{U}$ is defined respectively by
$$\widehat{h_\nu}^-(G,\mathcal{U})=h^-_\nu(\phi_G, X\times
\mathcal{U}) \text{ and
}\widehat{h_\nu}(G,\mathcal{U})=h_\nu(\phi_G, X\times
\mathcal{U}).$$
\end{de}

Clearly, if $\alpha\in \mathcal{P}_Y$  then
$\widehat{h_\nu}(G,\alpha)=\widehat{h_\nu}^-(G,\alpha)$. For
$\mathcal{U}\in \mathcal{C}_Y$,
$\widehat{h_\nu}(G,\mathcal{U})=\inf_{\alpha\in \mathcal{P}_Y:
\alpha\succeq \mathcal{U}} \widehat{h_\nu}^-(G, \alpha)$. Note
that there may exist plenty of free $G$-actions generating
$\mathcal{R}$, $\phi_G$ is not determined uniquely by $\Pi_g$.
Hence, we need to show that $\widehat{h_\nu}^-(G,\mathcal{U})$ and
$\widehat{h_\nu}(G,\mathcal{U})$ are well defined.

\begin{prop}
Let $\{ U_g\}_{g\in G}$ and $\{ U'_g\}_{g\in G}$ be two free
$G$-actions on $(X,\mathcal{B}_X,\mu)$ such that $$\{ U_g x:g\in
G\}=\{ U_g'x: g\in G\}=\mathcal{R}(x)$$ for $\mu$-a.e. $x\in X$.
Define cocycles $\phi,\phi':\mathcal{R}\rightarrow Aut(Y,\nu)$ by
$$\phi(U_g x,x)=\phi'(U_g' x,x)=\Pi_g \text{ for any } g\in G,x\in X.$$ Then for any
$\mathcal{U}\in \mathcal{C}_Y$, $h^-_\nu(\phi,X\times
\mathcal{U})=h^-_\nu(\phi',X\times \mathcal{U}) \text{ and }
h_\nu(\phi,X\times \mathcal{U})=h_\nu(\phi',X\times \mathcal{U})$.
\end{prop}
\begin{proof}
Denote by $\mathcal{S}$ the equivalence relation on $X\times X$
generated by the diagonal $G$-action $\{ U_g\times U_g'\}_{g\in
G}$. Clearly, $\mathcal{S}$ is measure preserving and hyperfinite.
Let $\varphi_{U},\varphi_{U'}:\mathcal{R}\rightarrow Aut(X,\mu)$
and $\phi_G:\mathcal{S}\rightarrow Aut(Y,\nu)$ be cocycles defined
by
$$\varphi_U(U_g'x,x)=U_g, \ \varphi_{U'}(U_gx,x)=U_g' \text{ and }
\phi_G((U_gx,U'_g x'),(x,x'))=\Pi_g $$
 for any $g\in G, x, x'\in
X$. Then $\mathcal{S}=\mathcal{R}(\varphi_{U'})= \sigma^{- 1}
\mathcal{R}(\varphi_U) \sigma$, where $\sigma: X\times
X\rightarrow X\times X$ is the flip map, that is, $\sigma(x,x')=
(x',x)$. Hence if $\{\mathcal{R}_n\}_{n\in \mathbb{N}}$ is a
filtration of $\mathcal{R}$ then
$\{\mathcal{R}_n(\varphi_{U'})\}_{n\in \mathbb{N}}$ and
$\{\sigma^{- 1} \mathcal{R}_n(\varphi_U)\sigma\}_{n\in
\mathbb{N}}$ are both filtrations of $\mathcal{S}$.

For each $n\in \mathbb{N}$, one has $\phi_G ((x, z), (x'', z''))=
\phi (x, x'')$ if $((x, z), (x'', z''))\in
\mathcal{R}_n(\varphi_{U'})$ and $\phi_G ((z, x), (z', x'))= \phi'
(x, x')$ if $((z, x), (z', x'))\in \sigma^{- 1}
\mathcal{R}_n(\varphi_U)\sigma$. Then by Proposition
\ref{pro-conclusion}, for any $\mathcal{U}\in \mathcal{C}_Y$ one
has
\begin{equation*}
h^-_\nu (\mathcal{R}_n(\varphi_{U'}),\phi_G,X\times X\times
\mathcal{U})= h_\nu^-(\mathcal{R}_n,\phi,X\times \mathcal{U}),
\end{equation*}
\begin{equation*}
h^-_\nu (\sigma^{- 1}
\mathcal{R}_n(\varphi_{U})\sigma,\phi_G,X\times X\times
\mathcal{U})= h_\nu^- (\mathcal{R}_n,\phi',X\times \mathcal{U}).
\end{equation*}
Let $n\rightarrow +\infty$ we obtain $h_\nu^-(\phi_G,X\times X
\times \mathcal{U})=h_\nu^-(\phi,X\times \mathcal{U})$ and
$h_\nu^-(\phi_G,X\times X \times
\mathcal{U})=h_\nu^-(\phi',X\times \mathcal{U})$ for any
$\mathcal{U}\in \mathcal{C}_Y$ (see Corollary \ref{pro-c-5-1}).
This implies that $h_\nu^-(\phi,X\times
\mathcal{U})=h_\nu^-(\phi',X\times \mathcal{U})$ for any
$\mathcal{U}\in \mathcal{C}_Y$. Moreover, for $\mathcal{U}\in
\mathcal{C}_Y$ we have
\begin{eqnarray*}
h_\nu(\phi,X\times \mathcal{U})&= & \inf_{\alpha\in
\mathcal{P}_{X\times Y} :\alpha \succeq X\times \mathcal{U}}
h_\nu^{-}(\phi,\alpha)=\inf_{\beta\in \mathcal{P}_{Y}
:\beta \succeq \mathcal{U}} h_\nu^{-}(\phi,X\times \beta)\\
&= & \inf_{\beta\in \mathcal{P}_{Y} :\beta \succeq \mathcal{U}}
h_\nu^{-}(\phi',X\times \beta)=h_\nu(\phi',X\times \mathcal{U}).
\end{eqnarray*}
This finishes the proof of the proposition.
\end{proof}

Before proceeding, we need the following result. Let $K\in F(G)$
and $\epsilon>0$. $F\in F(G)$ is called {\it
$[K,\epsilon]$-invariant} if $|\{ g\in F|Kg\subseteq
F\}|>(1-\epsilon)|F|$.

\begin{lem} \label{in-est}Let $(Y,\mathcal{B}_Y,\nu,G)$ be a $G$-measure
preserving system, $\mathcal{U}\in \mathcal{C}_Y$ and
$\epsilon>0$. Then there exist $K\in F(G)$ and
$0<\epsilon'<\epsilon$ such that if $F\in F(G)$ is
$[K,\epsilon']$-invariant then
$$|\frac{1}{|F|}H_\nu(\mathcal{U}_F)-h_\nu(G,\mathcal{U})|<\epsilon.$$
\end{lem}
\begin{proof}
Choose $e_G\in K_1\subseteq K_2\subseteq \cdots$ with
$\bigcup_{i\in \mathbb{N}} K_i=G$. For each $i\in \N$ set
$\delta_i=\frac{1}{2^i(|K_i|+1)}$. Now if the lemma is not true
then there exists $\epsilon>0$ such that for each $i\in
\mathbb{N}$ there exists $F_i\in F (G)$ such that it is
$[K_i^{-1}K_i,\delta_i]$-invariant and
\begin{align}\label{con-eq-1}
|\frac{1}{|F_i|}H_\nu(\mathcal{U}_{F_i})-h_\nu(G,\mathcal{U})|\ge
\epsilon.
\end{align}
Let $K\in F(G)$ with $e_G\in K$ and $\delta>0$. If $F\in F (G)$ is
$[K^{-1}K,\delta]$-invariant then
\begin{eqnarray*}
B(F,K)&= &\{ g\in G: Kg\cap F \neq \emptyset \text{ and } Kg\cap
(G\
\setminus F)\neq \emptyset\} \\
&= & K^{-1}F\setminus \{ g\in F:Kg\subseteq F\}= (K^{-1} F\setminus
F)\cup (F\setminus \{ g\in
F:Kg\subseteq F\}) \\
&\subseteq & K^{-1}(F\setminus \{g\in F: K^{-1}g\subseteq F\})\cup
(F\setminus \{ g\in F: Kg\subseteq F\})\\
&\subseteq & K^{-1}(F\setminus \{g\in F: K^{-1}Kg\subseteq F\})\cup
(F\setminus \{ g\in F: K^{-1}Kg\subseteq F\}),
\end{eqnarray*}
hence $|B(F,K)|\le (|K|+1)\cdot|F\setminus \{ g\in F:
K^{-1}Kg\subseteq F\}|\le \delta (|K|+1)|F|$ (as $F\in F (G)$ is
$[K^{-1}K,\delta]$-invariant), i.e. $F$ is a
$(K,(|K|+1)\delta)$-invariant set. Particularly, we have that $F_i$
is $(K_i,\frac{1}{2^i})$-invariant for each $i\in \mathbb{N}$.
Moreover, since $e_G\in K_1\subseteq K_2\subseteq \cdots$ and
$\bigcup_{i\in \mathbb{N}} K_i=G$, we have that $\{ F_i\}_{i\in
\mathbb{N}}$ is a F\o lner sequence of $G$. Hence
$\lim_{i\rightarrow +\infty}\frac{1}{|F_i|}H_\nu(\mathcal{U}_{F_i})=h_\nu^-(G,\mathcal{U})$, a contradiction with
\eqref{con-eq-1}.
\end{proof}

\begin{thm}\label{thm-c-equi} Let $(Y,\mathcal{B}_Y,\nu,G)$ be a $G$-measure
preserving system and $\mathcal{U}\in \mathcal{C}_Y$. Then
$$h^-_\nu(G,\mathcal{U})=\widehat{h_\nu}^-(G,\mathcal{U})\text{
and }h_\nu(G,\mathcal{U})=\widehat{h_\nu}(G,\mathcal{U}).$$
\end{thm}
\begin{proof}
By Lemma \ref{in-est} for each $\epsilon>0$ there exist $K\in F(G)$
and $0<\epsilon'<\epsilon$ such that if $F\in F(G)$ is
$[K,\epsilon']$-invariant then
$|\frac{1}{|F|}H_\nu(\mathcal{U}_F)-h_\nu(G,\mathcal{U})|<\epsilon$.
Let $\{\mathcal{R}_n\}_{n\in \mathbb{N}}$ be a filtration of
$\mathcal{R}$ with $\mathcal{R}_n\sim (B_i^{(n)},G_i^{(n)})$ for
each $n\in \N$. Thus by Lemma \ref{lem-c-2}~(1), for each
sufficiently large $n$ there is a measurable
$\mathcal{R}_n$-invariant subset $A_n\subseteq X$ such that
$\mu(A_n)>1-\epsilon'$ and
\begin{align}\label{eq-h-1}
\# \{ x'\in \mathcal{R}_n(x): Kx'\subseteq \mathcal{R}_n(x)\}> (1-
\epsilon') \# \mathcal{R}_n(x)\ \text{for each}\ x\in A_n.
\end{align}
Since $A_n$ is $\mathcal{R}_n$-invariant, $A_n=\bigsqcup_{i\in J}
G_i^{(n)} C_i^{(n)}$ for some subset $J\subseteq \mathbb{N}$ and a
family of measurable subsets $C_i^{(n)}\subseteq B_i^{(n)}$ with
$\mu(C_i^{(n)})>0$, $i\in J$. By \eqref{eq-h-1}, if $i\in J$,
$x\in C_i^{(n)}$ and $g'\in G_i^{(n)}$ then
\begin{eqnarray*}
(1- \epsilon') \# \mathcal{R}_n (g' x)< \# \{x'\in \mathcal{R}_n
(g' x): K x'\subseteq \mathcal{R}_n (g' x)\}= \# \{x'\in
\mathcal{R}_n (x): K x'\subseteq \mathcal{R}_n (x)\}.
\end{eqnarray*}
That is, $(1- \epsilon') |G_i^{(n)}|< |\{g\in G_i^{(n)}: K
g\subseteq G_i^{(n)}\}|$, i.e. $G_i^{(n)}$ is
$[K,\epsilon']$-invariant. Set
\begin{equation*}
f (x)= \frac{1}{\# \mathcal{R}_n(x)}H_\nu\left(\bigvee_{y\in
\mathcal{R}_n(x)} \phi_G (x,y)\mathcal{U}\right)\le \log N
(\mathcal{U})\ \text{for each}\ x\in X.
\end{equation*}
Then by similar reasoning of \eqref{eq-2858}, one has
$$\int\limits_{A_n} f (x) d \mu (x)= \sum\limits_{j\in J}
\int\limits_{C_j^{(n)}} H_\nu\left(\bigvee\limits_{g\in G_j^{(n)}}
\Pi_g^{- 1} \mathcal{U}\right)\, d\mu (x).$$ Hence
\begin{eqnarray*}
& & |h_\nu^- (\mathcal{R}_n, \phi_G, X\times \mathcal{U})- \mu
(A_n) h_\nu^- (G, \mathcal{U})| \\
&\le & |\int_{A_n} (f (x)- h_\nu^- (G, \mathcal{U}))\, d \mu (x)|+
|\int_{X\setminus A_n} f (x)\, d \mu (x)| \\
&\le & |\sum_{j\in J} \int_{C_j^{(n)}} |G_j^{(n)}|
\left(\frac{1}{|G_j^{(n)}|} H_\nu(\bigvee_{g\in G_j^{(n)}}
\Pi_g^{-1} \mathcal{U})- h_\nu^- (G, \mathcal{U})\right)\, d\mu
(x)|+ (1- \mu (A_n))\log
N (\mathcal{U}) \\
&\le & \left(\sum_{j\in J} |G_j^{(n)}| \mu (C_j^{(n)})\right)
\epsilon+ (1- \mu (A_n))\log N (\mathcal{U})\ (\text{by the
selection of}\ K\ \text{and}\ \epsilon').
\end{eqnarray*}
Note that $A_n=\bigsqcup_{i\in J} G_i^{(n)} C_i^{(n)}$ and
$\mu(A_n)>1-\epsilon'$ where $0< \epsilon'< \epsilon$, first let
$n\rightarrow +\infty$ and then let $\epsilon\rightarrow 0+$ we
have $\widehat{h_\nu}^-(G,\mathcal{U})=h_\nu^-(\phi_G,X\times
\mathcal{U})=h_\nu^-(G,\mathcal{U})$ (see Corollary
\ref{pro-c-5-1}). Moreover,
\begin{equation*}
\widehat{h_\nu}(G,\mathcal{U})=\inf_{\alpha\in \mathcal{P}_X:
\alpha\succeq \mathcal{U}}
\widehat{h_\nu}^-(G,\alpha)=\inf_{\alpha\in \mathcal{P}_X:
\alpha\succeq \mathcal{U}} h_\nu(G,\alpha)=h_\nu(G,\mathcal{U}).
\end{equation*}
This finishes the proof.
\end{proof}

Let $(Z, \mathcal{B}_Z, \kappa)$ be a Lebesgue space with $T$ an
invertible measure-preserving transformation, $\mathcal{W}\in
\mathcal{C}_Z$ and $\mathcal{D}\subseteq \mathcal{B}_Z$ a
$T$-invariant sub-$\sigma$-algebra, i.e. $T^{- 1}
\mathcal{D}= \mathcal{D}$. Set $\mathcal{W}_0^{n -1}= \bigvee_{i=
0}^{n- 1} T^{- i} \mathcal{W}$ for each $n\in \N$. It is clear
that the sequence $\{H_\kappa (\mathcal{W}_0^{n- 1}|
\mathcal{D})\}_{n\in \mathbb{N}}$ is non-negative and
sub-additive. So we may define
\begin{equation*}
h_\kappa (T, \mathcal{W}| \mathcal{D})= \inf_{\gamma\in
\mathcal{P}_Z:
 \gamma\succeq \mathcal{W}} h_\kappa^- (T, \gamma|
\mathcal{D}),
\end{equation*}
\begin{equation*}
h_\kappa^- (T, \mathcal{W}|
\mathcal{D})= \lim_{n\rightarrow +\infty} \frac{1}{n} H_\kappa
(\mathcal{W}_0^{n -1}| \mathcal{D})= \inf_{n\in
\mathbb{N}}\frac{1}{n} H_\kappa (\mathcal{W}_0^{n -1}|
\mathcal{D}).
\end{equation*}
Clearly $h_\kappa^- (T, \mathcal{W}| \mathcal{D})= h_\kappa (T,
\mathcal{W}| \mathcal{D})$ when $\mathcal{W}\in \mathcal{P}_Z$. We
shall write simply
$$h_\kappa^- (T, \mathcal{W})= h_\kappa^- (T,
\mathcal{W}| \{\emptyset, Z\}) \text{ and }h_\kappa (T,
\mathcal{W})= h_\kappa (T, \mathcal{W}| \{\emptyset, Z\}).$$

\begin{thm} \label{thm-c-d}
Let $\gamma$ be an invertible measure-preserving transformation on
$(X,\mathcal{B}_X,\mu)$ generating $\mathcal{R}$,
$\phi:\mathcal{R}\rightarrow Aut(Y,\nu)$ a cocyle and
$\gamma_\phi$ stand for the $\phi$-skew product extension of
$\gamma$. Then for each $\mathcal{U}\in \mathcal{C}_{X\times Y}$,
one has
$$h^-_\nu(\phi,\mathcal{U})=h^-_{\mu\times
\nu}(\gamma_\phi,\mathcal{U}|\mathcal{B}_X\otimes \{
\emptyset,Y\})\text{ and } h_\nu(\phi,\mathcal{U})=h_{\mu\times
\nu}(\gamma_\phi,\mathcal{U}|\mathcal{B}_X\otimes \{
\emptyset,Y\}).$$
\end{thm}
\begin{proof}
Let $\Sigma= \prod_{i= 1}^{+\infty} \{0, 1\}$ be the product space
of the discrete space $\{0, 1\}$. If $x= (x_1, x_2, \cdots), y=
(y_1, y_2, \cdots)\in \Sigma$ then the sum $x\oplus y= (z_1, z_2,
\cdots)$ is defined as follow. If $x_1+ y_1< 2$ then $z_1= x_1+
y_1$, if $x_1+ y_1\ge 2$ then $z_1= x_1+ y_1- 2$ and we carry $1$
to the next position. The other terms $z_2, \cdots$ are
successively determined in the same fashion. Let $\delta:
\Sigma\rightarrow \Sigma, z\mapsto z\oplus 1$ with $1= (1, 0, 0,
\cdots)$. It is known that $(\Sigma, \delta)$ is minimal, which is
called an {\it adding machine}. Let $\lambda$ be the Haar measure
on $(\Sigma, \oplus)$. Denote by $\mathcal{S}$ the $\delta\times
\gamma$-orbit equivalence relation on $\Sigma\times X$. Let
$\sigma:\Sigma\times X\rightarrow X\times \Sigma$ be the flip map.
We have $\mathcal{S}= \sigma^{- 1} \mathcal{R}(\varphi) \sigma$
for the cocycle $\varphi:\mathcal{R}\rightarrow
Aut(\Sigma,\lambda)$ given by $(\gamma^n x, x)\mapsto \delta^n,
n\in \mathbb{Z}$ (as $\mathcal{R}$ is conservative, $\gamma$ is
aperiodic and so $\varphi$ is well defined).

Now we define a cocycle $1\oplus \phi:\mathcal{S}\rightarrow
Aut(Y,\nu)$ by setting $((z, x), (z', x'))\mapsto \phi (x, x')$.
Let~ $\{\mathcal{R}_n\}_{n\in \mathbb{N}}$ be a filtration of
$\mathcal{R}$. Then
$\{\sigma^{-1}\mathcal{R}_n(\varphi)\sigma\}_{n\in \mathbb{N}}$ is
a filtration of $\mathcal{S}$ and so for each $\mathcal{U}\in
\mathcal{C}_{X\times Y}$
\begin{eqnarray} \label{07.02.25-1}
h^-_{\nu}(1\oplus \phi,\Sigma\times \mathcal{U})&= &
\lim_{n\rightarrow +\infty} h^-_{\nu}
(\sigma^{-1}\mathcal{R}_n(\varphi)\sigma,1\oplus \phi,\Sigma\times
\mathcal{U})\ (\text{by Corollary \ref{pro-c-5-1}})\ \nonumber \\
&= & \lim_{n\rightarrow +\infty}
h^-_\nu(\mathcal{R}_n,\phi,\mathcal{U})\ (\text{by Proposition
\ref{pro-conclusion}~(1)})\ \nonumber \\
&=&h^-_\nu(\phi,\mathcal{U})\ (\text{by Corollary
\ref{pro-c-5-1}}).
\end{eqnarray}
On the other hand, for each $n\in \N$ we let $A_n=\{ z\in \Sigma:
z_i=0 \text{ for }1\le i\le n\}$. Then $A_1\supseteq A_2\supseteq
\cdots$ is a sequence of measurable subsets of $\Sigma$ such that
$\Sigma= \bigsqcup_{i=0}^{2^n-1} \delta^i A_n$ and so
$\Sigma\times X=\bigsqcup_{i=0}^{2^n-1} (\delta\times
\gamma)^i(A_n\times X)$ for each $n\in \N$. Let $\mathcal{S}_n\in
I (\mathcal{S})$ with $\mathcal{S}_n\sim (A_n\times X,\{
(\delta\times \gamma)^i: i= 0, 1, \cdots, 2^n-1\})$. By
\eqref{eq-2858} we obtain that
\begin{eqnarray*}
& & h^-_\nu(\mathcal{S}_n,1\oplus \phi, \Sigma\times \mathcal{U})
\\
&= & \int_{A_n\times X} H_\nu\left(\bigvee_{i=0}^{2^n-1}\phi
(x,\gamma^ix)\mathcal{U}_{\gamma^i x}\right) d \lambda\times \mu (z,x) \\
&= & \frac{1}{2^n} \int_X H_\nu\left(\bigvee_{i=0}^{2^n-1}
\phi(x,\gamma^i
x)\mathcal{U}_{\gamma^i x}\right) d \mu(x)\ (\text{as}\ \lambda (A_n)= \frac{1}{2^n}) \\
&= & \frac{1}{2^n}\int_X H_\nu\left(\left(\bigvee_{i=0}^{2^n-1}
\gamma_{\phi}^{-i}\mathcal{U}\right)_{x}\right) d \mu(x)\
\left(\text{as}\ \bigvee_{i=0}^{2^n-1} \phi(x,\gamma^i
x)\mathcal{U}_{\gamma^i x}= \left(\bigvee_{i=0}^{2^n-1}
\gamma_{\phi}^{-i}\mathcal{U}\right)_{x}\right) \\
&= & \frac{1}{2^n}\int_X H_{\delta_x \times
\nu}\left(\bigvee_{i=0}^{2^n-1}
\gamma_{\phi}^{-i}\mathcal{U}\right) d \mu(x)=
\frac{1}{2^n}H_{\mu\times \nu} \left(\bigvee_{i=0}^{2^n-1}
\gamma_{\phi}^{-i}\mathcal{U}|\mathcal{B}_X\otimes \{
\emptyset,Y\}\right).
\end{eqnarray*}
Note that $\mathcal{S}_1\subseteq \mathcal{S}_2\subseteq \cdots$
and $\bigcup_{n\in \mathbb{N}} \mathcal{S}_n=\mathcal{S}$, then
\begin{eqnarray*}
h^-_\nu(1\oplus \phi,\Sigma\times \mathcal{U})&= &\lim_{n\rightarrow
+\infty} h^-_\nu(\mathcal{S}_n,1\oplus \phi, \Sigma\times
\mathcal{U}) \ \ \text{(by Corollary \ref{pro-c-5-1})}\\
&= &\lim_{n\rightarrow +\infty}\frac{1}{2^n}H_{\mu\times \nu}
\left(\bigvee_{i=0}^{2^n-1}
\gamma_{\phi}^{-i}\mathcal{U}|\mathcal{B}_X\otimes \{
\emptyset,Y\}\right)\\
&= &h_{\mu\times \nu}^-(\gamma_\phi,\mathcal{U}|\mathcal{B}_X\otimes
\{ \emptyset, Y\})
\end{eqnarray*}
 and so
$h_\nu^-(\phi,\mathcal{U})=h_{\mu\times
\nu}^-(\gamma_\phi,\mathcal{U}|\mathcal{B}_X \otimes \{ \emptyset,
Y\})$ for each $\mathcal{U}\in \mathcal{C}_{X\times Y}$ by
\eqref{07.02.25-1}. Finally,
\begin{eqnarray*}
h_\nu(\phi,\mathcal{U})&= & \inf_{\alpha\in \mathcal{P}_{X\times Y}:
\alpha\succeq \mathcal{U}} h^-_\nu(\phi,\alpha) =\inf_{\alpha\in
\mathcal{P}_{X\times Y}: \alpha\succeq \mathcal{U}} h^-_{\mu\times
\nu}(\gamma_\phi,\alpha|\mathcal{B}_X\otimes \{ \emptyset, Y\})\\
&= & h_{\mu\times \nu} (\gamma_\phi, \mathcal{U}|\mathcal{B}_X\otimes
\{ \emptyset,Y\})
\end{eqnarray*}
 for each $\mathcal{U}\in
\mathcal{C}_{X\times Y}$. This finishes the proof the theorem.
\end{proof}

\subsection{The proof of the equivalence of measure-theoretic entropy of covers}

The following result was proved by the same authors \cite[Theorem 6.4]{HYZ1} (see
also \cite{GW4, HMRY}).

\begin{lem}\label{lem-c-4-2}
Let $(X, T)$ be a TDS with $\mathcal{U}\in \mathcal{C}_X$ and
$\mu\in \mathcal{M} (X,T)$. Then $h_\mu^{-} (T,
\mathcal{U})= h_\mu (T, \mathcal{U}).$
\end{lem}

\begin{lem}\label{lem-c-4-3}
Let $(Z, \mathcal{B}_Z, \kappa)$ be a Lebesgue space with $T$ an
invertible measure-preserving transformation, $\mathcal{W}\in
\mathcal{C}_Z$ and $\mathcal{D}\subseteq \mathcal{B}_Z$ a
$T$-invariant sub-$\sigma$-algebra. Then $h^-_\kappa (T,
\mathcal{W}| \mathcal{D})= h_\kappa (T, \mathcal{W}|
\mathcal{D})$.
\end{lem}
\begin{proof}
First we claim the conclusion for the case $\mathcal{D}=
\{\emptyset, Z\}$. By the ergodic decomposition of $h^-_\kappa (T,
\mathcal{W})$ and $h_\kappa (T, \mathcal{W})$ (see \eqref{lab-eq}
in the case of $G=\Z$), it suffices to prove it when $\kappa$ is
ergodic. By the Jewett-Krieger Theorem (see \cite{DGS}),  $(Z, \kappa,
T)$ is measure theoretical isomorphic to a uniquely ergodic
zero-dimensional topological dynamical system $(\widehat{Z},
\widehat{\kappa}, \widehat{T})$. Let $\pi: (\widehat{Z},
\widehat{\kappa}, \widehat{T})\to (Z, \kappa, T)$ be such an
isomorphism. Then using Lemma \ref{lem-c-4-2} we have
\begin{equation*}
h^-_\kappa (T, \mathcal{W})= h^-_{\widehat{\kappa}} (\widehat{T},
\pi^{- 1} \mathcal{W})= h_{\widehat{\kappa}} (\widehat{T}, \pi^{-
1} \mathcal{W})= h_\kappa (T, \mathcal{W}).
\end{equation*}

In general case, let $\{\beta_j\}_{j\in \mathbb{N}}\subseteq
\mathcal{P}_Z$ with $\beta_j\nearrow \mathcal{D}$ (mod $\mu$).
For simplicity, we write $\mathcal{P}(\mathcal{V})=\{ \alpha\in
\mathcal{P}_Z: \alpha\succeq \mathcal{V}\}$ for $\mathcal{V}\in
\mathcal{C}_X$. Then
\begin{eqnarray} \label{kkk-eq-11}
h^-_\kappa (T, \mathcal{W}| \mathcal{D})&=&\inf_{n\ge 1} \frac{1}{n}
H_\kappa (\mathcal{W}_0^{n-1}| \mathcal{D})=\inf_{n\ge 1}
\frac{1}{n} \left(\inf_{\alpha\in \mathcal{P}(\mathcal{W}_0^{n- 1})}
H_\kappa (\alpha|
\mathcal{D})\right) \nonumber \\
&=&\inf_{n\ge 1} \frac{1}{n} \left(\inf_{\alpha\in
\mathcal{P}(\mathcal{W}_0^{n- 1})} \inf_{j\ge 1} H_\kappa (\alpha|
(\beta_j)_0^{n- 1})\right)\ (\text{as $\beta_j\nearrow
\mathcal{D}$ (mod $\mu$)}) \nonumber \\
&=& \inf_{j\ge 1} \inf_{n\ge 1} \frac{1}{n} \left(\inf_{\alpha\in
\mathcal{P}(\mathcal{W}_0^{n- 1})} H_\kappa (\alpha| (\beta_j)_0^{n-
1})\right) \nonumber \\
&=&\inf_{j\ge 1} \inf_{n\ge 1} \frac{1}{n} H_\kappa
(\mathcal{W}_0^{n-1}| (\beta_j)_0^{n- 1}).
\end{eqnarray}
Let $j\in \N$. Since for any $n,m\in \N$ and $\mathcal{V}\in \mathcal{C}_X$ one has
\begin{eqnarray*}
H_\kappa (\mathcal{V}_0^{n+m-1}| (\beta_j)_0^{n+m- 1})&\le &
H_\kappa(\mathcal{V}_0^{n-1}|(\beta_j)_0^{n+m-
1})+H_\kappa(T^{-n}\mathcal{V}_0^{m-1}|(\beta_j)_0^{n+m- 1})\\
&\le & H_\kappa(\mathcal{V}_0^{n-1}|(\beta_j)_0^{n-
1})+H_\kappa(T^{-n}\mathcal{V}_0^{m-1}|T^{-n}(\beta_j)_0^{m- 1})\\
&= & H_\kappa(\mathcal{V}_0^{n-1}|(\beta_j)_0^{n-
1})+H_\kappa(\mathcal{V}_0^{m-1}|(\beta_j)_0^{m- 1}).
\end{eqnarray*}
Hence
\begin{align}\label{eeee-1} \inf_{n\ge 1} \frac{1}{n} H_\kappa
(\mathcal{V}_0^{n-1}| (\beta_j)_0^{n- 1})=\lim_{n\rightarrow
+\infty}\frac{1}{n} H_\kappa (\mathcal{V}_0^{n-1}| (\beta_j)_0^{n-
1}).
\end{align}
Combing \eqref{eeee-1} for $\mathcal{V}=\mathcal{W}$ with
\eqref{kkk-eq-11}, one has
\begin{eqnarray*}
h^-_\kappa (T, \mathcal{W}| \mathcal{D}) &= & \inf_{j\ge 1}
\lim_{n\rightarrow +\infty}\frac{1}{n} H_\kappa
(\mathcal{W}_0^{n-1}| (\beta_j)_0^{n- 1}) \\
&= & \inf_{j\ge
1}\lim_{n\rightarrow +\infty}\frac{1}{n}\inf_{\alpha\in
\mathcal{P}(\mathcal{W}_0^{n- 1})} H_\kappa (\alpha| (\beta_j)_0^{n-
1}) \\
&=&\inf_{j\ge 1}\lim_{n\rightarrow +\infty} \frac{1}{n}\left(
\inf_{\alpha\in \mathcal{P}(\mathcal{W}_0^{n- 1})} H_\kappa
(\alpha\vee (\beta_j)_0^{n- 1})-H_\kappa ((\beta_j)_0^{n-
1}) \right)\\
&\ge & \inf_{j\ge 1} \lim_{n\rightarrow +\infty} \frac{1}{n}\left(
H_\kappa (\mathcal{W}_0^{n- 1}\vee (\beta_j)_0^{n- 1})-H_\kappa
((\beta_j)_0^{n- 1}) \right)\\
&=& \inf_{j\ge 1} \left( h^-_\kappa (T,\mathcal{W}\vee \beta_j)-
h_\kappa^- (T, \beta_j) \right)
\\
&=& \inf_{j\ge 1} (h_\kappa (T,\mathcal{W}\vee \beta_j)- h_\kappa
(T, \beta_j))
\ \ \ \ \ (\text{by the first part}) \\
&=& \inf_{j\ge 1} (\inf_{\alpha\in \mathcal{P}( \mathcal{W})}
h_\kappa (T, \alpha\vee \beta_j)- h_\kappa
(T, \beta_j)) \\
&=&\inf_{j\ge 1} \inf_{\alpha\in \mathcal{P}(\mathcal{W})}
\lim_{n\rightarrow +\infty} \frac{1}{n}\left( H_\kappa ((\alpha\vee
\beta_j)_0^{n- 1})- H_\kappa ((\beta_j)_0^{n- 1}) \right) \\
&\ge & \inf_{j\ge 1} \inf_{\alpha\in \mathcal{P}(\mathcal{W})}
\inf_{n\ge 1} \frac{1}{n} H_\kappa (\alpha_0^{n- 1}| (\beta_j)_0^{n- 1}) \ \ \ \text{(by \eqref{eeee-1}
for $\mathcal{V}=\alpha$)}\\
&= &  \inf_{\alpha\in \mathcal{P}(\mathcal{W})}
\inf_{n\ge 1} \inf_{j\ge 1}\frac{1}{n} H_\kappa (\alpha_0^{n- 1}| (\beta_j)_0^{n- 1}) \\
&=& \inf_{\alpha\in \mathcal{P}(\mathcal{W})} \inf_{n\ge 1}
\frac{1}{n} H_\kappa (\alpha_0^{n- 1}| \mathcal{D})\, (\text{as
$\beta_j\nearrow \mathcal{D}$ (mod $\mu$)})\\
&=&h_\kappa (T, \mathcal{W}| \mathcal{D}).
\end{eqnarray*}
As the inequality of $h^-_\kappa (T, \mathcal{W}| \mathcal{D})\le
h_\kappa (T, \mathcal{W}| \mathcal{D})$ is straightforward, this
finishes the proof.
\end{proof}

The following result is our main result in the  section.

\begin{thm} \label{equivalence}
Let $(Y, \mathcal{B}_Y, \nu, G)$ be a $G$-measure preserving
system with $(Y, \mathcal{B}_Y, \nu)$ a Lebesgue space and
$\mathcal{U}\in \mathcal{C}_Y$. Then $h_\nu (G, \mathcal{U})=
h_\nu^- (G, \mathcal{U})$.
\end{thm}
\begin{proof}
Let $(X,\mathcal{B}_X,\mu,G)$ be a free $G$-measure preserving
system with $\mathcal{R}\subseteq X\times X$ the $G$-orbit
equivalence relation and $\gamma$ an invertible measure-preserving
transformation on $(X,\mathcal{B}_X,\mu)$ generating
$\mathcal{R}$. The cocycle $\phi_G:\mathcal{R}\rightarrow
Aut(Y,\nu)$ is given by $\phi_G(gx,x)=\Pi_g$, where $\Pi_g\in
Aut(Y,\nu)$ is the action of $g\in G$ on $(Y,\mathcal{B}_Y,\nu)$.
By Definition \ref{de-1943} of vitual entropy and Theorem
\ref{thm-c-equi}, we have
\begin{align}\label{key-eq-1}
h_\nu^-(G,\mathcal{U})=h_\nu^-(\phi_G,X\times \mathcal{U}) \text{
and } h_\nu(G,\mathcal{U})=h_\nu(\phi_G,X\times \mathcal{U}).
\end{align}
Let $T=\gamma_{\phi_G}$ be the $\phi$-skew production extension of
$\gamma$. Using Theorem \ref{thm-c-d} one has
\begin{equation} \label{key-eq-2}
h_\nu^-(\phi_G,X\times \mathcal{U})=h_{\mu\times
\nu}^-(T,\mathcal{U}|\mathcal{B}_X \times \{\emptyset,Y\})\
\text{and}\ h_\nu(\phi_G,X\times \mathcal{U})=h_{\mu\times
\nu}(T,\mathcal{U}|\mathcal{B}_X \times \{\emptyset,Y\}).
\end{equation}
As $\mathcal{B}_X\times \{ \emptyset, Y\}$ is $T$-invariant,
$h_{\mu\times \nu}^-(T,\mathcal{U}|\mathcal{B}_X \times
\{\emptyset,Y\}) =h_{\mu\times \nu}(T,\mathcal{U}|\mathcal{B}_X
\times \{\emptyset,Y\})$ by Lemma \ref{lem-c-4-3}. Combining this
fact with \eqref{key-eq-1} and \eqref{key-eq-2}, we get
$h_\nu^-(G,\mathcal{U})=h_\nu(G,\mathcal{U})$. This finishes the
proof.
\end{proof}

\subsection{A local version of Katok's result}

At the end of this section, we shall give a local version of a
well-known result of Katok \cite[Theorem I.I]{Kat(1980.b)} for a
$G$-action. Let $(X, G)$ be a $G$-system, $\mu\in \mathcal{M} (X,
G)$ and $\mathcal{U}\in \mathcal{C}_X$. Let $a\in (0, 1)$ and
$F\in F (G)$. Set
$$b (F, a, \mathcal{U})= \min \{\#
(\mathcal{C}): \mathcal{C}\subseteq \mathcal{U}_F\ \text{and}\ \mu
(\cup \mathcal{C})\ge a\}.$$ The following simple fact is inspired
by \cite[Lemma 5.11]{We}.

\begin{lem} \label{weiss}
$H_\mu (\mathcal{U}_F)\le \log b (F, a, \mathcal{U})+ (1- a)
|F|\log N (\mathcal{U})+ \log 2$.
\end{lem}
\begin{proof}
Let $\mathcal{C}=\{ C_1,\cdots,C_\ell\}\subseteq
\mathcal{U}_F$ such that $\mu (\cup \mathcal{C})\ge a$ and $\ell=
b (F, a, \mathcal{U})$. Let $\alpha_1=\{ C_1,C_2\setminus
C_1,\cdots, C_\ell\setminus \bigcup_{j=1}^{\ell-1} C_j\}$. Then
$\alpha_1$ is a partition of $\bigcup_{i=1}^\ell C_i$ and $\#
\alpha_1=b (F, a, \mathcal{U})$. Similarly, we take $\alpha_2'\in
\mathcal{P}_X$ satisfying $\#\alpha_2'=N(\mathcal{U}_F)$. Then let
$\alpha_2=\{ A \cap (X\setminus \bigcup_{i=1}^\ell C_i): A\in
\alpha_2'\}$. Then $\#\alpha_2\le N(\mathcal{U}_F)$. Set $\alpha=
\alpha_1\cup \alpha_2$. Then $\alpha\in \mathcal{P}_X$ and $\alpha
\succeq \mathcal{U}_F$. Note that if $x_1, \cdots, x_m\ge 0$ then
\begin{equation} \label{27-6}
\sum_{i= 1}^m \phi (x_i)\le \left(\sum_{i= 1}^m x_i\right) \log m+
\phi \left(\sum_{i= 1}^m x_i\right),
\end{equation}
thus one has
\begin{eqnarray*}
& & H_\mu (\mathcal{U}_F)\le H_\mu (\alpha) \\
&\le & \mu \left(\bigcup_{i=1}^\ell C_i\right) \left(\log \# \alpha_1- \log \mu
\left(\bigcup_{i=1}^\ell C_i\right)\right)+ \\
& &\hskip2cm \left(1- \mu \left(\bigcup_{i=1}^\ell C_i\right)\right) \left(\log \#
\alpha_2- \log \left(1- \mu \left(\bigcup_{i=1}^\ell C_i\right)\right)\right)\ (\text{by \eqref{27-6}}) \\
&\le & \log b (F, a, \mathcal{U})+ (1- a) \log N(\mathcal{U}_F)-
\mu \left(\bigcup_{i=1}^\ell C_i\right) \log \mu
\left(\bigcup_{i=1}^\ell C_i\right) \\
& &\hskip2cm -\left(1- \mu \left(\bigcup_{i=1}^\ell C_i\right)\right) \log \left(1- \mu
\left(\bigcup_{i=1}^\ell C_i\right)\right)\\
 &\le & \log b (F, a, \mathcal{U})+ (1- a) |F|\log
N(\mathcal{U})+ \log 2.
\end{eqnarray*}
\end{proof}

As a direct application of Lemma \ref{weiss} by letting
$a\rightarrow 1-$ we have

\begin{prop} \label{le part}
Let $\{F_n\}_{n\in \mathbb{N}}$ be a F\o lner sequence of $G$.
Then
\begin{equation*}
h_\mu^- (G, \mathcal{U})\le \lim_{\epsilon\rightarrow 0+}
\liminf_{n\rightarrow +\infty} \frac{1}{|F_n|} \log b (F_n, 1-
\epsilon, \mathcal{U}).
\end{equation*}
\end{prop}

The following result is \cite[Theorem 1.3]{L}.

\begin{lem} \label{Lindenstrauss}
Let $\alpha\in \mathcal{P}_X$ and $\{F_n\}_{n\in \mathbb{N}}$ be a
F\o lner sequence of $G$ such that $\lim_{n\rightarrow +\infty}
\frac{|F_n|}{\log n}= +\infty$ and for some constant $C> 0$ one
has $|\bigcup_{k= 1}^{n- 1} F_k^{- 1} F_n|\le C |F_n|$ for each
$n\in \N$. If $\mu$ is ergodic then for $\mu$-a.e. $x\in X$ and in
the sense of $L^1 (X, \mathcal{B}_X, \mu)$-norm one has
\begin{equation*}
\lim_{n\rightarrow +\infty} - \frac{\log \mu (\alpha_{F_n}
(x))}{|F_n|}= h_\mu (G, \alpha).
\end{equation*}
\end{lem}

\begin{prop} \label{ge part}
Let $\{F_n\}_{n\in \mathbb{N}}$ be a F\o lner sequence of $G$. If
$\mu\in \mathcal{M}^e (X, G)$ then
\begin{equation*}
h_\mu (G, \mathcal{U})\ge \lim_{\epsilon\rightarrow 0+}
\limsup_{n\rightarrow +\infty} \frac{1}{|F_n|} \log b (F_n, 1-
\epsilon, \mathcal{U}).
\end{equation*}
\end{prop}
\begin{proof}
First for any $\mathcal{P}\in \mathcal{P}_X$ we claim the
conclusion by proving
\begin{equation} \label{partition case}
h_\mu (G, \alpha)\ge \limsup_{n\rightarrow +\infty}
\frac{1}{|F_n|} \log b (F_n, 1- \epsilon, \alpha)\ \text{for each
$\epsilon\in (0, 1)$}.
\end{equation}

\begin{proof}[Proof of Claim]
Fix $\epsilon\in (0, 1)$. In $\{F_n\}_{n\in \N}$ we can select a
sub-sequence $\{E_n\}_{n\in \mathbb{N}}$ satisfying
\begin{equation*}
\limsup_{n\rightarrow +\infty} \frac{1}{|F_n|} \log b (F_n, 1-
\epsilon, \alpha)= \lim_{n\rightarrow +\infty} \frac{1}{|E_n|}
\log b (E_n, 1- \epsilon, \alpha),
\end{equation*}
$\lim_{n\rightarrow +\infty} \frac{|E_n|}{\log n}= +\infty$ and
for some constant $C> 0$ one has $|\bigcup_{k= 1}^{n- 1} E_k^{- 1}
E_n|\le C |E_n|$ for each $n\in \N$. Now applying Lemma
\ref{Lindenstrauss} to $\{E_n\}_{n\in \mathbb{N}}$, for each
$\delta> 0$ there exists $N\in \mathbb{N}$ such that for each
$n\ge N$, $\mu (A_n)\ge 1- \epsilon$ where
\begin{eqnarray*}
A_n&= & \left\{x\in X: - \frac{\log \mu (\alpha_{E_n}
(x))}{|E_n|}\le h_\mu
(G, \alpha)+ \delta\right\} \\
&\supseteq & \left\{x\in X: - \frac{\log \mu (\alpha_{E_m}
(x))}{|E_m|}\le h_\mu (G, \alpha)+ \delta\ \text{if}\ m\ge
n\right\}.
\end{eqnarray*}
Note that $A_n$ must be a union of some atoms in $\alpha_{E_n}$,
where each atom has measure at least $e^{- |E_n| (h_\mu (G,
\alpha)+ \delta)}$, which implies $b (E_n, 1- \epsilon, \alpha)\le
(1- \epsilon) e^{|E_n| (h_\mu (G, \alpha)+ \delta)}$ when $n\ge
N$. So
\begin{equation*}
\limsup_{n\rightarrow +\infty} \frac{1}{|F_n|} \log b (F_n, 1-
\epsilon, \alpha)= \lim_{n\rightarrow +\infty} \frac{1}{|E_n|}
\log b (E_n, 1- \epsilon, \alpha)\le h_\mu (G, \alpha)+ \delta.
\end{equation*}
Since $\delta> 0$ is arbitrary, one claims \eqref{partition case}.
\end{proof}

Now for general case, by the above discussions we have
\begin{eqnarray*}
h_\mu (G, \mathcal{U})&= & \inf_{\alpha\in \mathcal{P}_X:
\alpha\succeq \mathcal{U}} h_\mu (G, \alpha) \\
&\ge & \inf_{\alpha\in \mathcal{P}_X: \alpha\succeq \mathcal{U}}
\lim_{\epsilon\rightarrow 0+} \limsup_{n\rightarrow +\infty}
\frac{1}{|F_n|} \log b (F_n, 1- \epsilon, \alpha)\
(\text{by}\ \eqref{partition case}) \\
&\ge & \lim_{\epsilon\rightarrow 0+} \limsup_{n\rightarrow
+\infty} \frac{1}{|F_n|} \log b (F_n, 1- \epsilon, \mathcal{U}).
\end{eqnarray*}
\end{proof}

Now combining Lemma \ref{equivalence} with Propositions \ref{le
part} and \ref{ge part} we obtain (when $G= \Z$, it can be viewed
as a local version of Katok's result \cite[Theorem
I.I]{Kat(1980.b)})

\begin{thm} \label{060111}
Let $\{F_n\}_{n\in \mathbb{N}}$ be a F\o lner sequence of $G$. If
$\mu\in \mathcal{M}^e (X, G)$ then
\begin{eqnarray*}
h_\mu (G, \mathcal{U})= \lim_{\epsilon\rightarrow 0+}
\limsup_{n\rightarrow +\infty} \frac{1}{|F_n|} \log b (F_n, 1-
\epsilon, \mathcal{U})= \lim_{\epsilon\rightarrow 0+}
\liminf_{n\rightarrow +\infty} \frac{1}{|F_n|} \log b (F_n, 1-
\epsilon, \mathcal{U}).
\end{eqnarray*}
\end{thm}

\section{A local variational principle of topological entropy}\label{sectionlvp}

The main result of this section is

\begin{thm}[Local Variational Principle of Topological Entropy]  \label{avpin}
Let $\mathcal{U}\in \mathcal{C}_X^{o}$. Then
$$h_{\text{top}}(G,\mathcal{U})= \max_{\mu \in
\mathcal{M}(X,G)} h_\mu(G,\mathcal{U})= \max_{\mu \in \mathcal{M}^e
(X,G)} h_\mu(G,\mathcal{U}).$$
\end{thm}

We remark that Theorem \ref{avpin} generalizes the results in
\cite{OP2, ST}:

\begin{thm}[Variational Principle of Topological Entropy, \cite{OP2,
ST}]
\label{vpp}
$$h_{\text{top}}(G, X)=\sup_{\mu \in
\mathcal{M}(X,G)} h_\mu(G, X)=\sup_{\mu\in
\mathcal{M}^e(X,G)}h_\mu(G, X).
$$
\end{thm}
\begin{proof}
It is  a direct corollary of Lemma \ref{basic} (3), Theorems
\ref{ope} and \ref{avpin}.
\end{proof}

Before proving Theorem \ref{avpin}, we need a key lemma.

\begin{lem} \label{key} Let $\mathcal{U}\in \mathcal{C}_X^{o}$ and $\alpha_l\in
\mathcal{P}_X$ with $\alpha_l\succeq \mathcal{U}$, $1\le l\le K$.
Then for each $F\in F (G)$ there exists a finite subset
$B_F\subseteq X$ such that each atom of $(\alpha_l)_F$ contains at
most one point of $B_F$, $l=1,\cdots,K$ and $\# B_F\ge
\frac{N(\mathcal{U}_F)}{K}$.
\end{lem}
\begin{proof} We follow the arguments in the proof of \cite[Lemma 3.5]{HYZ1}.
Let $F\in F (G)$. For each $l=1,\cdots, K$ and $x\in X$, let
$A_l(x)$ be the atom of $(\alpha_l)_F$ containing $x$, then for
$x_1,x_2\in X$, $x_1$ and $x_2$ are contained in the same atom of
$(\alpha_l)_F$ iff $A_l(x_1)=A_l(x_2)$.

 To construct the subset $B_F$ we first take any $x_1\in X$.
If $\bigcup_{l=1}^K A_l(x_1)=X$,  then we take $B_F=\{ x_1 \}$.
Otherwise, we take $X_1=X\setminus \bigcup_{l=1}^K A_l(x_1)\neq
\emptyset$ and take any $x_2\in X_1$. If $\bigcup_{l=1}^K
A_l(x_2)\supseteq X_1$, then we take $B_F=\{ x_1,x_2 \}$. Otherwise,
we take $X_2=X_1\setminus \bigcup_{l=1}^K A_l(x_2)\neq \emptyset$.
Since $\{ A_l(x):1\le l\le K,x\in X \}$ is a finite cover of $X$, we
can continue the above procedure inductively to obtain a finite
subset $B_F=\{ x_1,\cdots x_m \}$ and non-empty subsets $X_j$,
$j=1,\cdots,m-1$ such that
\begin{enumerate}

\item $X_1=X\setminus \bigcup_{l=1}^K A_l(x_1)$,

\item $X_{j+1}=X_j\setminus \bigcup_{l=1}^K A_l(x_{j+1})$ for
$j=1,\cdots,m-1$,

\item $\bigcup_{j=1}^m \bigcup_{l=1}^K A_l(x_j)=X$.
\end{enumerate}

From the construction of $B_F$, clearly each atom of
$(\alpha_l)_F$, $l=1,\cdots,K$, contains at most one point of
$B_F$. Since for any $1\le i\le m$ and $1\le l\le K$, $A_l(x_i)$
is an atom of $(\alpha_l)_F$, and thus is contained in some
element of $\mathcal{U}_F$, so $m K\ge N(\mathcal{U}_F)$ (using
(3)), that is, $\# B_F= m\ge \frac{N(\mathcal{U}_F)}{K}$.
\end{proof}

\begin{prop} \label{plecd} Let $\mathcal{U}\in \mathcal{C}^{o}_X$. If $X$
is zero-dimensional then there exists $\mu \in \mathcal{M}(X,G)$
satisfying
\begin{equation} \label{gcd}
h_\mu(G,\mathcal{U})\ge h_{\text{top}}(G,\mathcal{U}).
\end{equation}
\end{prop}
\begin{proof}
Let $\mathcal{U}=\{ U_1,\cdots,U_d \}$ and $\mathcal{U}^*=
\{\alpha= \{A_1,\cdots,A_d\}\in \mathcal{P}_X: A_m\subseteq U_m,
m=1, \cdots, d \}$. Since $X$ is zero-dimensional, the family of
partitions in $\mathcal{U}^*$ consisting of clopen (closed and
open) subsets, which are finer than $\mathcal{U}$, is countable.
We let $\{ \alpha_l: l\ge 1 \}$ denote an enumeration of this
family. Then $h_\nu (G, \mathcal{U})= \inf_{l\in \N} h_\nu (G,
\alpha_l)$ for each $\nu\in \mathcal{M} (X, G)$ by Lemma
\ref{07-01-02-01}.

Let $\{ F_n \}_{n\in \mathbb{N}}$ be a F\o lner sequence of $G$
satisfying $|F_n|\ge n$ for each $n\in \N$ (obviously, such a
sequence exists since $|G|=+\infty$). By Lemma \ref{key}, for each
$n\in \N$ there exists a finite subset $B_n\subseteq X$ such that
\begin{align}\label{Bn}
\# B_n\ge \frac{N(\mathcal{U}_{F_n})}{n},
\end{align}
 and each atom of
$(\alpha_l)_{F_n}$ contains at most one point of $B_n$, for each
$l=1,\cdots, n$. Let
\begin{eqnarray*}
\nu_n=\frac{1}{\# B_n}\sum_{x\in B_n} \delta_{x}\ \text{and}\
\mu_n=\frac{1}{|F_n|}\sum_{g\in F_n} g\nu_n.
\end{eqnarray*}
We can choose a sub-sequence $\{ n_j \}_{j\in \mathbb{N}}\subseteq
\N$ such that $\mu_{n_j}\rightarrow \mu$ in the weak$^*$-topology of
$\mathcal{M}(X)$ as $j\rightarrow +\infty$. It is not hard to check
the invariance of $\mu$, i.e. $\mu\in \mathcal{M} (X, G)$. Now we
aim to show that $\mu$ satisfies \eqref{gcd}. It suffices to show
that $h_{\text{top}}(G,\mathcal{U})\le h_\mu(G,\alpha_l)$ for each
$l\in \mathbb{N}$.

Fix a $l\in \mathbb{N}$ and each $n>l$. Using \eqref{Bn} we know
from the construction of $B_n$ that
\begin{equation} \label{n-v}
\log N(\mathcal{U}_{F_n})-\log n \le \log (\# B_n)=\sum_{x\in
B_n}- \nu_n(\{ x \}) \log \nu_n( \{ x
\})=H_{\nu_n}((\alpha_l)_{F_n}).
\end{equation}
On the other hand, for each $B\in F (G)$, using Lemma
\ref{lem-436} (3) one has
\begin{eqnarray} \label{n-v1}
\frac{1}{|F_n|}H_{\nu_n}((\alpha_l)_{F_n})&\le & \frac{1}{|F_n|}
\sum\limits_{g\in F_n} \frac{1}{|B|} H_{\nu_n}
((\alpha_l)_{Bg})+\frac{|F_n \setminus \{ g\in G:
B^{-1}g\subseteq F_n \}|}{|F_n|} \cdot \log \# \alpha_l\ \nonumber \\
&= & \frac{1}{|B|\cdot|F_n|} \sum\limits_{g\in F_n} H_{g\nu_n}
((\alpha_l)_{B}) +\frac{|F_n \setminus \{ g\in G: B^{-1}g\subseteq
F_n \}|}{|F_n|} \cdot
\log d\ \nonumber \\
&\le & \frac{1}{|B|}H_{\mu_n}((\alpha_l)_{B})+\frac{|F_n \setminus
\{ g\in G: B^{-1}g\subseteq F_n \}|}{|F_n|} \cdot \log d.
\end{eqnarray}
Now by dividing \eqref{n-v} on both sides by $|F_n|$, then combining
it with \eqref{n-v1} we obtain
\begin{equation} \label{n-v2}
\frac{1}{|F_n|} \log N(\mathcal{U}_{F_n})\le
\frac{1}{|B|}H_{\mu_n}((\alpha_l)_{B})+\frac{\log n}{|F_n|}
+\frac{|F_n \setminus \{ g\in G: B^{-1}g\subseteq F_n \}|}{|F_n|}
\cdot \log d.
\end{equation}
Note that $\lim_{j\rightarrow +
\infty}H_{\mu_{n_j}}((\alpha_l)_B)=H_\mu((\alpha_l)_B)$, by
substituting $n$ with $n_j$ in \eqref{n-v2} one has
$$h_{\text{top}}(G,\mathcal{U})\le \frac{1}{|B|}H_{\mu}((\alpha_l)_{B})\
(\text{using}\ \eqref{need}).$$ Now, taking the infimum over $B\in
F (G)$, we get $h_{\text{top}}(G,\mathcal{U})\le
h_\mu(G,\alpha_l)$. This ends the proof.
\end{proof}

A continuous map $\pi: (X, G)\rightarrow (Y, G)$ is called a {\it
homomorphism} or a {\it factor map} if it is onto and $\pi\circ
g=g\circ \pi$ for each $g\in G$. In this case, $(X, G)$ is called
an {\it extension of $(Y, G)$} and $(Y, G)$ is called a {\it
factor of $(X, G)$}. If $\pi$ is also injective then it is called
an {\it isomorphism}.

\begin{proof}[Proof of Theorem \ref{avpin}] First, by Lemmas
\ref{basic} (1) and \ref{equivalence}, it suffices to prove
$h_{\theta}(G,\mathcal{U})\ge h_{\text{top}}(G,\mathcal{U})$ for
some $\theta\in \mathcal{M}^e (X, G)$.
 It is well known that there
exists a surjective continuous map $\phi_1: C\rightarrow X$, where
$C$ is a cantor set. Let $C^G$ be the product space equipped with
the $G$-shift $G\times C^G\rightarrow C^G, (g', (z_g)_{g\in
G})\mapsto (z'_g)_{g\in G}$ where $z'_g= z_{g' g}, g', g\in G$.
Define
\begin{eqnarray*}
Z= \{\overline{z}= (z_g)_{g\in G}\in C^G: \phi_1 (z_{g_1 g_2})= g_1
\phi_1 (z_{g_2})\ \text{for each}\ g_1, g_2\in G\},
\end{eqnarray*}
and $\varphi: Z\rightarrow X, (z_g)_{g\in G}\mapsto \phi_1 (z_{e_G})$.
It's not hard to check that $Z\subseteq C^G$ is a closed invariant
subset under the $G$-shift. Moreover, $\varphi:(Z,G)\rightarrow
(X,G)$ becomes a factor map between $G$-systems.
Applying Proposition \ref{plecd} to the $G$-system $(Z, G)$, there
exists $\nu\in \mathcal{M}(Z,G)$ with
$h_\nu(G,\varphi^{-1}(\mathcal{U}))\ge h_{\text{top}} (G,
\varphi^{-1} (\mathcal{U}))= h_{\text{top}} (G,\mathcal{U})$. Let
$\mu=\varphi \nu\in \mathcal{M}(X,G)$. Then
$$h_\mu(G,\mathcal{U})=\inf_{\alpha\in \mathcal{P}_X:
\alpha\succeq \mathcal{U}} h_\mu(G,\alpha) =\inf_{\alpha\in
\mathcal{P}_X: \alpha\succeq \mathcal{U}}
h_\nu(G,\varphi^{-1}(\alpha))\ge
h_\nu(G,\varphi^{-1}(\mathcal{U}))\ge
h_{\text{top}}(G,\mathcal{U}).$$ Let
$\mu=\int_{\mathcal{M}^e(X,T)} \theta d m(\theta)$ be the ergodic
decomposition of $\mu$. Then by Theorem \ref{en-decom} one has
\begin{equation*}
\int_{\mathcal{M}^e(X,T)} h_{\theta}(G,\mathcal{U}) d m (\theta)
=h_\mu(G,\mathcal{U}).
\end{equation*}
Hence, $h_{\theta}(G,\mathcal{U})\ge
h_{\text{top}}(G,\mathcal{U})$ for some $\theta\in \mathcal{M}^e
(X, G)$. This ends the proof.
\end{proof}

At last, we ask an  open question.

\begin{ques}
In the proof of \cite[Proposition 7.10]{GW4} (or its relative
version \cite[Theorem A.3]{HYZ1}), a universal version of the
well-known Rohlin Lemma \cite[Proposition 7.9]{GW4} plays a key
role. Thus, a natural open question arises: for actions of a countable
discrete amenable group, are there a universal version of Rohlin
Lemma and a similar result to \cite[Proposition 7.10]{GW4} or
\cite[Theorem A.3]{HYZ1}? Whereas, up to now they still stand as
open questions.
\end{ques}

\section{Entropy tuples}

In this section we will firstly introduce entropy tuples in both
topological and measure-theoretic settings. Then we characterize the
set of entropy tuples for an invariant measure as the support of
some specific relative product measure. Finally by the lift property
of entropy tuples, we will establish the variational relation of
entropy tuples. At the same time, we also discuss entropy tuples of
a finite product. We need to mention that the proof of those
results in this section are similar to the proof of corresponding
results in \cite{HY, HYZ2} for the case $G=\mathbb{Z}$, but for
completion we provide the detailed proof.

\subsection{Topological entropy tuples}

First we are going to define the topological entropy tuples.

Let $n\ge 2$. Set $X^{(n)}=X\times \cdots \times X$ ($n$-times);
$\Delta_{n}(X)=\{(x_i)_{1}^n \in X^{(n)}| x_1= \cdots =x_n\}$, the
{\it $n$-th diagonal} of $X$. Let $(x_i)_1^n\in X^{(n)}\setminus
\Delta_{n}(X)$. We say $\mathcal{U}\in \mathcal{C}_X$ {\it
admissible w.r.t. $(x_i)_1^n$}, if for any $U\in \mathcal{U}$,
$\overline{U}\not \supseteq \{ x_1,\cdots,x_n \}$.

\begin{de}
Let $n\ge 2$. $(x_i)_1^n\in X^{(n)}\setminus \Delta_{n}(X)$ is
called a {\it topological entropy $n$-tuple} if $h_{\text{top}}(G,
\mathcal{U})>0$ when $\mathcal{U}\in \mathcal{C}_X$ is admissible
w.r.t. $(x_i)_1^n$.
\end{de}

\begin{rem} \label{thm108}
We may replace all admissible finite covers by admissible finite
open or closed covers in the definition. Moreover, we can choose all
covers to be of the forms $\mathcal{U}=\{U_1, \cdots, U_n\}$, where
$U_i^c$ is a neighborhood of $x_i$, $1\le i\le n$ such that if
$x_i\neq x_j, 1\le i< j\le n$ then $U_i^c\cap U_j^c =\emptyset$.
Thus, our definition of topological entropy $n$-tuples is the same
as the one defined by Kerr and Li in \cite{KL}.
\end{rem}

For each $n\ge 2$, denote by $E_n(X, G)$ the set of all topological
entropy $n$-tuples. Then following the ideas of \cite{B2} we obtain
directly

\begin{prop} \label{thm110}
Let $n\ge 2$.
\begin{description}

\item[1] If $\mathcal{U}=\{ U_1,\cdots,U_n \}\in
\mathcal{C}_X^{o}$ satisfies $h_{\text{top}}(G,\mathcal{U})>0$
then $E_n(X,G)\cap \bigcap_{i= 1}^n U_i^c\neq \emptyset$.

\item[2] If $h_{\text{top}}(G, X)>0$, then $\emptyset\neq \overline{E_n(X, G)}
\subseteq X^{(n)}$ is $G$-invariant. Moreover,
$\overline{E_n(X,G)}\setminus \Delta_n(X)= E_n(X,G)$.

\item[3] Let $\pi: (Z, G)\rightarrow (X, G)$ be a factor map between
$G$-systems. Then
$$E_n(X,G)\subseteq (\pi\times \cdots \times
\pi) E_n(Z,G)\subseteq E_n(X,G)\cup \Delta_n(X).$$

\item[4] Let $(W, G)$ be a sub-$G$-system of $(X, G)$. Then
$E_n(W,G)\subseteq E_n(X,G)$.
\end{description}
\end{prop}

The notion of disjointness of two TDSs was introduced in \cite{F}.
Blanchard proved that any u.p.e. TDS was disjoint from all minimal
TDSs with zero topological entropy (see \cite[Proposition 6]{B2}).
This is also true for actions of a countable discrete amenable
group. First we introduce

\begin{de}
Let $n\ge 2$. We say that
\begin{enumerate}

\item $(X, G)$ has {\it u.p.e. of order $n$},
if any cover of $X$ by $n$ non-dense open sets has positive
topological entropy. When $n=2$, we say simply that $(X, G)$ has
{\it u.p.e.}

\item $(X, G)$ has {\it u.p.e. of all orders} or {\it
topological $K$} if any cover of $X$ by finite non-dense open sets
has positive topological entropy, equivalently, it has u.p.e. of
order $m$ for any $m\ge 2$.
\end{enumerate}
\end{de}

Thus, for each $n\ge 2$, $(X,G)$ has u.p.e. of order $n$ iff
$E_n(X,G)= X^{(n)}\setminus \Delta_n(X)$.

We say $(X, G)$ {\it minimal} if it contains properly no other
sub-$G$-systems. Let $(X, G)$ and $(Y, G)$ be two $G$-systems and
$\pi_X: X\times Y\rightarrow X$, $\pi_Y: X\times Y\rightarrow Y$ the
natural projections. $J\subseteq X\times Y$ is called a {\it joining
of $(X, G)$ and $(Y, G)$} if $J$ is a $G$-invariant closed subset
satisfying $\pi_X(J)=X$ and $\pi_Y(J)=Y$. Clearly, $X\times Y$ is
always a joining of $(X, G)$ and $(Y, G)$. We say that $(X,G)$ and
$(Y,G)$ are {\it disjoint} if $X\times Y $ is the unique joining of
$(X, G)$ and $(Y,G)$. The proof of the following theorem is similar
to that of \cite[Proposition 6]{B2} or \cite[Theorem 2.5]{HYZ2}.

\begin{thm} \label{thm104}
Let $(X,G)$ be a $G$-system having u.p.e. and $(Y,G)$ a minimal
$G$-system with zero topological entropy. Then $(X,G)$ and $(Y,G)$
are disjoint.
\end{thm}

\subsection{Measure-theoretic entropy tuples}

Now we aim to define the measure-theoretic entropy tuples for an
invariant Borel probability measure.

Let $\mu\in \mathcal{M} (X,G)$. $A\subseteq X$ is called a {\it
$\mu$-set} if $A\in \mathcal{B}_X^{\mu}$. If $\alpha= \{A_1, \cdots,
A_k\}\subseteq \mathcal{B}_X^\mu$ satisfies $\bigcup_{i= 1}^k A_i=
X$ and $A_i\cap A_j= \emptyset$ when $1\le i< j\le k$ then we say
$\alpha$ a {\it finite $\mu$-measurable partition of $X$}. Denote by
$\mathcal{P}_X^\mu$ the set of all finite $\mu$-measurable
partitions of $X$. Similarly, we can introduce $\mathcal{C}_X^\mu$
and define $\alpha_1\succeq \alpha_2$ for $\alpha_1, \alpha_2\in
\mathcal{C}_X^\mu$ and so on.

\begin{de} \label{etfm-de-1} Let $n\ge 2$. $(x_i)_1^n\in
X^{(n)}\setminus \Delta_{n}(X)$ is called a {\it measure-theoretic
entropy $n$-tuple for $\mu$} if $h_{\mu}(G,\alpha)>0$ for any
admissible $\alpha\in \mathcal{P}_X$ w.r.t. $(x_i)_1^n$.
\end{de}

\begin{rem} \label{etfm-rem-2} We may replace all admissible $\alpha\in \mathcal{P}_X$
 by
all admissible $\alpha\in \mathcal{P}_X^\mu$ in the definition.
\end{rem}

For each $n\ge 2$, denote by $E_n^{\mu}(X,G)$ the set of all
measure-theoretic entropy $n$-tuples for $\mu\in \mathcal{M} (X,
G)$. In the following, we shall investigate the structure of
$E_n^{\mu}(X,G)$. To this purpose, let $P_{\mu}$ be the Pinsker
$\sigma$-algebra of $(X,\mathcal{B}_X^\mu,\mu,G)$, i.e. $P_\mu=
\{A\in \mathcal{B}_X^\mu: h_\mu (G, \{A, A^c\})= 0\}$. We define a
measure $\lambda_n(\mu)$ on $(X^{(n)},(\mathcal{B}_X^\mu)^{(n)}, G)$
by letting
$$
\lambda_n(\mu) \left(\prod_{i=1}^n A_i\right)= \int_X \prod_{i=1}^n
\E(1_{A_i}|P_{\mu}) d \mu,
$$
where $(\mathcal{B}_X^\mu)^{(n)}=\mathcal{B}_X^\mu\times
\cdots\times \mathcal{B}_X^\mu$ ($n$ times) and $A_i\in
\mathcal{B}_X^\mu, i=1, \cdots, n$. First we need

\begin{lem} \label{etfm-lem-3}
Let $\mathcal{U}=\{ U_1,\cdots,U_n \}\in \mathcal{C}_X$. Then
$\lambda_n(\mu)(\prod_{i=1}^n U_i^c)>0$ iff for any finite (or
$n$-set) $\mu$-measurable partition $\alpha$, finer than
$\mathcal{U}$ as a cover, one has $h_{\mu}(G,\alpha)>0$.
\end{lem}
\begin{proof}
First we assume that for any finite (or $n$-set) $\mu$-measurable
partition $\alpha$, finer than $\mathcal{U}$ as a cover, one has
$h_{\mu}(G,\alpha)>0$ and $\lambda_n(\mu)(\prod_{i=1}^n U_i^c)=0$.
For $i= 1, \cdots, n$. Let $C_i=\{ x\in X:
\E(1_{U_i^c}|P_{\mu})(x)>0 \}\in P_{\mu}$, and put $D_i=C_i\cup
(U_i^c\setminus C_i)$, $D_i(0)= D_i$ and $D_i(1)= D_i^c$, as
$$
0=\int_{X\setminus C_i} \E(1_{U_i^c}|P_{\mu})(x) d \mu=\mu
(U_i^c\cap (X\setminus C_i))= \mu(U_i^c\setminus C_i),
$$
then $D_i^c\subseteq U_i$ and $D_i (0), D_i (1)\in P_{\mu}$. For any
$s=(s(1),\cdots,s(n))\in \{ 0,1 \}^n$, let $D_s= \bigcap_{i=1}^{n}
D_i (s(i))$ and set $D^j_0=(\bigcap_{i=1}^{n}D_i) \cap (U_j\setminus
\bigcup_{k=1}^{j-1} U_k)$ for $j=1,\cdots,n$. We consider
$$
\alpha=\{ D_s: s\in \{ 0,1 \}^n\ \text {and} \ s\neq (0,\cdots,0) \}
\cup \{ D_0^1,\cdots,D_0^n\}.
$$
On one hand, for any $s\in {\{ 0,1 \}}^n $ with $s\neq (0,
\cdots,0)$, one has $s(i)=1$ for some $1\le i \le n$, then
$D_s\subseteq D_i^c \subseteq U_i$. Note that $D_0^j\subseteq U_j, \
j=1,\cdots,n$, thus $\alpha\succeq \mathcal{U}$ and so
$h_{\mu}(G,\alpha)>0$. On the other hand, obviously
$\mu(\bigcap_{i=1}^n D_i)=\mu(\bigcap_{i=1}^n C_i)$ and
\begin{equation*}
0= \lambda_n(\mu)\left(\prod_{i=1}^n U_i^c\right)=
\int_{\bigcap_{i=1}^n C_i} \prod_{i=1}^n \E(1_{U_i^c}|P_{\mu}) (x) d
\mu (x),
\end{equation*}
then $\mu (\bigcap_{i=1}^n C_i)= 0$, and so $D_0^1,\cdots,D_0^n\in
P_{\mu}$. As $D_1,\cdots,D_n\in P_{\mu}$, $D_s\in P_{\mu}$ for each
$s\in \{ 0,1 \}^n$, thus $\alpha\subseteq P_{\mu}$, one gets
$h_{\mu}(G,\alpha)= 0$, a contradiction.

Now we assume $\lambda_n(\mu)(\prod_{i=1}^n U_i^c)>0$. For any
finite (or $n$-set) $\mu$-measurable partition $\alpha$ which is
finer than $\mathcal{U}$, with no loss of generality we assume $\alpha=\{ A_1,\cdots,A_n
\}$ with $A_i\subseteq U_i, \ i=1,\cdots,n$. As
\begin{equation*}
\int_X \prod_{i=1}^n \E(1_{X\setminus A_i}|P_{\mu})(x) d \mu(x) \ge
\int_X \prod_{i=1}^n \E(1_{U_i^c}|P_{\mu})(x) d \mu(x)
=\lambda_n(\mu)(\prod_{i=1}^n U_i^c)>0,
\end{equation*}
therefore $A_j \notin P_{\mu}$ for every $1\le j \le n$, and so
$h_{\mu}(G,\alpha)>0$. This finishes the proof.
\end{proof}

Then we have (we remark that the case of $G=\Z$ is proved in
\cite{G} and \cite{HY}).

\begin{thm} \label{etfm-thm-4} Let $n\ge 2$ and
$\mu\in \mathcal{M} (X,G)$. Then
$E_n^{\mu}(X,G)=\text{supp}(\lambda_n(\mu))\setminus \Delta_n (X)$.
\end{thm}
\begin{proof} Let $(x_i)_1^n\in E_n^{\mu}(X,G)$. To show
$(x_i)_1^n\in \text{supp}(\lambda_n(\mu))\setminus \Delta_n (X)$, it
remains to prove that for any Borel neighborhood $\prod_{i=1}^n U_i$ of
$(x_i)_1^n$ in $X^{(n)}$, $\lambda_n(\mu)(\prod_{i=1}^n U_i)>0$. Set
$\mathcal{U}=\{ U_1^c,\cdots,U_n^c \}$. With no loss of generality we assume
$\mathcal{U}\in \mathcal{C}_X$ (if necessity we replace $U_i$ by a
smaller Borel neighborhood of $x_i$, $1\le i\le n$). It is clear that if
$\alpha\in \mathcal{P}_X^\mu$ is finer than $\mathcal{U}$ then
$\alpha$ is admissible w.r.t. $(x_i)_1^n$, and so
$h_{\mu}(G,\alpha)>0$. Using Lemma \ref{etfm-lem-3} one has
$\lambda_n(\mu)(\prod_{i=1}^n U_i)>0$.

Now let $(x_i)_1^n\in \text{supp}(\lambda_n(\mu)) \setminus
\Delta_n (X)$. We shall show that $h_{\mu}(G,\alpha)>0$ for any
admissible $\alpha=\{ A_1,\cdots,A_k \}\in \mathcal{P}_X$ w.r.t.
$(x_i)_1^n$. In fact, let $\alpha$ be such a partition. Then there
exists a neighborhood $U_l$ of $x_l$, $1\le l\le n$ such that for
each $i\in \{1,\cdots,k \}$ we find $j_i\in \{ 1,\cdots,n \}$ with
$A_i \subseteq U_{j_i}^c$, i.e. $\alpha\succeq \mathcal{U}=\{
U_1^c,\cdots,U_n^c \}$. As $(x_i)_1^n\in
\text{supp}(\lambda_n(\mu)) \setminus \Delta_n (X)$,
$\lambda_n(\mu)(\prod_{i=1}^n U_i)>0$ and so $h_{\mu}(G,\alpha)>0$
(see Lemma \ref{etfm-lem-3}). This ends the proof.
\end{proof}

Before proceeding we also need

\begin{thm}[{\cite[Theorem 0.1]{D}}] \label{Danil} Let
$\mu\in \mathcal{M} (X,G)$, $\alpha\in \mathcal{P}_X^\mu$ and
$\epsilon>0$. Then there exists $K\in F (G)$ such that if $F\in F
(G)$ satisfies $(F F^{- 1}\setminus \{e_G\})\cap K= \emptyset$ then
$$|\frac{1}{|F|}H_\mu(\alpha_F|P_\mu)-H_\mu(\alpha|P_\mu)|<\epsilon.$$
\end{thm}

The following theorem are crucial for this section of our paper, and
the methods of proving it may be useful in other settings as well.

\begin{thm} \label{etfm-thm-5}
Let $\mu\in \mathcal{M} (X,G)$ and $\mathcal{U}
 =\{U_1,\cdots, U_n \}\in \mathcal{C}_X^\mu, n\ge 2$.
 If $h_{\mu}(G,\alpha)>0$ for any finite
(or $n$-set) $\mu$-measurable partition $\alpha$, finer than
$\mathcal{U}$, then $h_\mu^-(G,\mathcal{U})>0$.
\end{thm}
\begin{proof}
For any $s=(s(1),\cdots,s(n))\in { \{ 0,1 \} }^n$, set
$A_s=\bigcap_{i=1}^{n}U_i(s(i))$, where $U_i(0)=U_i$ and
$U_i(1)=U_i^c$. Let $\alpha=\{ A_s: s\in { \{ 0,1 \} }^n \}$. Note
that $\lambda_n(\mu)(\prod_{i=1}^n U_i^c)= \int_X \prod_{i=1}^n
\E(1_{U_i^c}|P_{\mu}) d\mu >0$ (Lemma \ref{etfm-lem-3}), hence there
exists $M\in \N$ such that $\mu(D)>0$, where
$$D=\left\{ x\in X:\min_{1\le i \le n} \E(1_{U_i^c}|P_{\mu})(x) \ge \frac {1}{M} \right\}.$$

\noindent{\bf {Claim.}} If $\beta\in \mathcal{P}_X^\mu$ is finer
than $\mathcal{U}$ then $H_{\mu}(\alpha|\beta \vee P_{\mu}) \le
H_{\mu}(\alpha|P_{\mu})-\frac{\mu(D)}{M} \log (\frac {n}{n-1})$.

\begin{proof}[Proof of Claim] With no loss of generality we
assume $\beta=\{ B_1, \cdots, B_n \}$ with $B_i\subseteq U_i, \
i=1,\cdots,n$. Then
\begin{eqnarray} \label{etfm-eq-1}
H_{\mu}(\alpha|\beta \vee P_{\mu}) &= & H_{\mu}(\alpha \vee \beta|
P_{\mu})-
H_{\mu}(\beta| P_{\mu}) \nonumber \\
&= & \int_X \sum_{s\in { \{ 0,1 \} }^n} \sum_{i=1}^n \E(1_{B_i} |
P_{\mu}) \phi \left(\frac {\E(1_{A_s \cap B_i}| P_{\mu})}
{\E(1_{B_i} | P_{\mu})}\right) d \mu \nonumber \\
&= & \sum_{s\in { \{ 0,1 \} }^n} \int_X \sum_{1\le i\le n, s(i)=0}
\E(1_{B_i} | P_{\mu}) \phi\left(\frac {\E(1_{A_s \cap B_i}|
P_{\mu})} {\E(1_{B_i}| P_{\mu})}\right) d \mu,
\end{eqnarray}
where the last equality comes from the fact that, for any $s\in { \{
0,1 \} }^n$ and $1\le i \le n$, if $s(i)=1$ then $A_s\cap
B_i=\emptyset$ and so $\frac {\E(1_{A_s \cap B_i}
|P_{\mu})}{\E(1_{B_i} |P_{\mu})}(x)= 0$ for $\mu$-a.e. $x\in X$. Put
$c_s= \sum_{1\le k\le n,s(k)=0} \E(1_{B_k} |P_{\mu})$. As $\phi$ is
a concave fucntion, \eqref{etfm-eq-1}
\begin{eqnarray}
&\le & \sum\limits_{s\in { \{ 0,1 \} }^n} \int_X c_s \cdot
\phi\left(\sum\limits_{1\le i\le n,s(i)=0} \frac{\E(1_{B_i}
|P_{\mu})} {c_s}\cdot \frac {\E(1_{A_s \cap B_i} |P_{\mu})}
{\E(1_{B_i} |P_{\mu})}\right) d \mu \nonumber \\
&= & \sum\limits_{s\in { \{ 0,1 \} }^n} \int_X c_s \cdot
\phi\left(\frac {\E(1_{A_s} |P_{\mu})}
{c_s}\right) d \mu \nonumber \\
&= & \sum\limits_{s\in { \{ 0,1 \} }^n} \left(\int_X
\phi(\E(1_{A_s}|P_{\mu})) d \mu - \int_X \E(1_{A_s}|P_{\mu}) \log
\frac {1}
{c_s} d \mu\right) \nonumber \\
&= & H_{\mu}(\alpha|P_{\mu})-\sum\limits_{s\in { \{ 0,1 \} }^n}
\int_X \E(1_{A_s}|P_{\mu}) \log \frac {1} {c_s} d \mu.
\label{07-01-03-07}
\end{eqnarray}
Note that if $s(i)=1$, $1\le i\le n$ then $\sum_{1\le k\le n,
s(k)=0} \E(1_{B_k} |P_{\mu}) \le \E(1_{X\setminus B_i}|P_{\mu})$,
moreover, $(\frac{b_1+\cdots +b_n}{n})^n\ge b_1\cdots b_n$ and
$\sum_{i=1}^n b_i=\sum_{i=1}^n\sum_{1\le j\le n,
j\not=i}\E(1_{B_j}|P_\mu) =(n-1)\sum_{i=1}^n\E(1_{B_i}|P_\mu)=n-1$,
here $b_i=\E(1_{X\setminus B_i}|P_\mu)$, $i=1,\cdots, n$. Then we
have
\begin{eqnarray}
& & \sum\limits_{s\in { \{ 0,1 \}}^n} \int_X \E(1_{A_s}|P_{\mu})
\log \left(\frac {1} {\sum\limits_{1\le k\le n,s(k)=0}
\E(1_{B_k} |P_{\mu})}\right) d \mu \nonumber \\
&\ge & \frac{1}{n} \sum\limits_{i=1}^n \int_X
\left(\sum\limits_{s\in { \{ 0,1 \}}^n,s(i)=1}
\E(1_{A_s}|P_{\mu})\right)
\log \frac {1}{b_i} d \mu \nonumber \\
&= & \frac{1}{n} \sum\limits_{i=1}^n \int_X \E(1_{U_i^c}|P_{\mu})
\log \frac {1}{b_i}d\mu \ge \frac{1}{nM}
\sum\limits_{i=1}^n \int_D \log \frac{1}{b_i}d \mu \nonumber \\
&= &\frac{1}{nM} \int_D \log \frac {1} {\prod_{i=1}^n b_i} d\mu \ge
\frac{1}{M} \int_D \log\frac {n} {\sum\limits_{i=1}^n b_i} d
\mu=\frac{\mu(D)}{M} \log \left(\frac {n}{n-1}\right),
\label{07-01-03-08}
\end{eqnarray}
Hence, $H_{\mu}(\alpha|\beta \vee P_{\mu}) \le
H_{\mu}(\alpha|P_{\mu})-\frac {\mu(D)}{M} \log (\frac {n}{n-1})$
(using \eqref{07-01-03-07} and \eqref{07-01-03-08}).
\end{proof}

Set $\epsilon=\frac {\mu(D)}{M} \log (\frac {n}{n-1})>0$. By Theorem
\ref{Danil}, there exists $K\in F (G)$ such that
\begin{equation} \label{etfm-eq-2}
|\frac{1}{|F|}H_\mu(\alpha_F|P_\mu)-
H_\mu(\alpha|P_\mu)|<\frac{\epsilon}{2}
\end{equation}
when $F\in F (G)$ satisfies $(F F^{-1}\setminus \{e_G\})\cap K=
\emptyset$. Let $\{ F_m \}_{m\in \mathbb{N}}$ be a F\o lner sequence
of $G$. For each $m\in \N$, we can take $E_m\subseteq F_m$ such that
$(E_m E_m^{-1}\setminus \{e_G\})\cap K= \emptyset$ and $|E_m|\ge
\frac{|F_m|}{2|K|+ 1}$. Now if $\beta_m\in \mathcal{C}_X^\mu$ is finer
than $\mathcal{U}_{F_m}$ then $g\beta_m\succeq \mathcal{U}$ for each
$g\in F_m$, and so
\begin{eqnarray*}
H_{\mu}(\beta_m)&\ge & H_\mu(\beta_m \vee \alpha_{E_m}|P_\mu)-
H_\mu(\alpha_{E_m}|\beta_m \vee P_\mu) \\
&\ge & H_\mu(\alpha_{E_m}|P_\mu)-\sum_{g\in E_m}
H_\mu(\alpha|g\beta_m \vee P_\mu) \\
&\ge & H_\mu(\alpha_{E_m}|P_\mu)-|E_m|
(H_\mu(\alpha|P_\mu)-\epsilon) \ \text{(by Claim)} \\
&\ge &|E_m|\frac{\epsilon}{2} \ \text{(by the selection of $E_m$ and
applying \eqref{etfm-eq-2} to $E_m$)}.
\end{eqnarray*}
Hence, $H_\mu(\mathcal{U}_{F_m})\ge |E_m|\frac{\epsilon}{2}$ and so
$h_\mu^-(G,\mathcal{U})\ge \frac{\epsilon}{2 (2|K|+ 1)}$. This finishes the
proof of the theorem.
\end{proof}

An immediate consequence of Lemma \ref{etfm-lem-3} and Theorem
\ref{etfm-thm-5} is
\begin{cor} \label{etfm-cor-6}
Let $\mu\in \mathcal{M} (X,G)$ and $\mathcal{U}
 =\{U_1,\cdots, U_n \}\in \mathcal{C}_X^\mu$. Then the following
 statements are equivalent:
\begin{description}

\item[1] $h_{\mu}^- (G,\mathcal{U})>0$, equivalently, $h_{\mu} (G,\mathcal{U})>0$;

\item[2] $h_\mu (G, \alpha)> 0$ if $\alpha\in \mathcal{C}_X^\mu$ is
finer than $\mathcal{U}$;

\item[3]
$\lambda_n(\mu)(\prod_{i=1}^n U_i^c)>0$.
\end{description}
\end{cor}

Now with the help of Theorem \ref{en-decom} and Corollary
\ref{etfm-cor-6} we can obtain Theorem \ref{et-decom} which
discloses the relation of entropy tuples for an invariant measure
and entropy tuples for ergodic measures in its ergodic
decomposition, generalizing \cite[Theorem 4]{BGH} and \cite[Theorem
4.9]{HY}.

\begin{thm} \label{et-decom}
Let $\mu\in \mathcal{M} (X, G)$ with $\mu=\int_{\Omega} \mu_{\omega}
d m(\omega)$ the ergodic decomposition of $\mu$. Then
\begin{description}

\item[1] for $m$-a.e. $\omega\in \Omega$, $E_n^{\mu_{\omega}}(X,
G)\subseteq E_n^{\mu}(X, G)$ for each $n\ge 2$.

\item[2] if $(x_i)_1^n\in E_n^{\mu}(X, G)$, then for every
measurable neighborhood $V$ of $(x_i)_1^n$, $m (\{ \omega\in \Omega:
V\cap E_n^{\mu_{\omega}}(X, G)\not=\emptyset\})>0$. Thus for an
appropriate choice of $\Omega$, we can require
$$
\overline{\cup \{ E_n^{\mu_{\omega}}(X, G): \omega \in \Omega
\}}\setminus \Delta_n(X) =E_n^{\mu}(X, G).
$$
\end{description}
\end{thm}
\begin{proof}
1. It suffices to prove the conclusion for each given $n\ge 2$. Let
$n\ge 2$ be fixed.

Let $U_i,\ i=1,\cdots,n$ be open subsets of $X$ with
$\bigcap_{i=1}^n \overline{U_i}=\emptyset$ and $(\prod_{i=1}^n
\overline{U_i}) \cap E^\mu_n(X,G)=\emptyset$. Then
$\lambda_n(\mu)(\prod_{i=1}^n \overline{U_i})= 0$ by Theorem
\ref{etfm-thm-4}, and so $h_\mu(G,\mathcal{U})=0$ by Corollary
\ref{etfm-cor-6}, where $\mathcal{U}=\{ U_1^c,\cdots,U_n^c\}$. As
$\int_\Omega h_{\mu_\omega}(G,\mathcal{U}) d
m(\omega)=h_\mu(G,\mathcal{U})=0$ (see \eqref{lab-eq}), for
$m$-a.e. $\omega\in \Omega$, $h_{\mu_\omega}(G,\mathcal{U})=0$ and
so $\lambda_n(\mu_\omega)(\prod_{i=1}^n U_i)=0$ by Corollary
\ref{etfm-cor-6}, hence $(\prod_{i=1}^n U_i) \cap
E^{\mu_\omega}_n(X,G)=\emptyset$ (using Theorem \ref{etfm-thm-4}
and the assumption of $\bigcap_{i=1}^n \overline{U_i}=\emptyset$).

Since $E_n^\mu(X,G)\cup \Delta_n(X)\subseteq X^{(n)}$ is closed, its
complement can be written as a union of countable sets of the form
$\prod_{i=1}^n U_i$ with $U_i,i=1,\cdots,n$ open subsets satisfying
$\bigcap_{i=1}^n \overline{U_i}=\emptyset$. Then applying the above
procedure to each such a subset $\prod_{i=1}^n U_i$ one has that for
$m$-a.e. $\omega\in \Omega$, $E^{\mu_\omega}_n(X,G)\cap
{(E^\mu_n(X,T))}^c=\emptyset$, equivalently,
$E^{\mu_\omega}_n(X,G)\subseteq E^\mu_n(X,T)$.

2.  With no loss of generality we assume $V=\prod_{i=1}^n A_i$, where $A_i$ is a
closed neighborhood of $x_i$, $1\le i\le n$ and $\bigcap_{i=1}^n
A_i=\emptyset$. As $\lambda_n(\mu)(\prod_{i=1}^n A_i)>0$ by
Theorem \ref{etfm-thm-4}, one has
\begin{equation*}
\int_\Omega h_{\mu_\omega}(T,\{ A_1^c,\cdots,A_n^c \}) d m(\omega)
=h_\mu(T,\{ A_1^c,\cdots,A_n^c \})>0\text{ (using \eqref{lab-eq} and
Corollary \ref{etfm-cor-6})},
\end{equation*}
there exists $\Omega'\subseteq \Omega$ with $m(\Omega')>0$ such that
if $\omega\in \Omega'$ then
$$
h_{\mu_\omega}(G,\{ A_1^c,\cdots,A_n^c \})>0,\ \text{i.e.}\
\lambda_n(\mu_\omega)\left(\prod_{i=1}^nA_i\right)>0\text{ (see
Corollary \ref{etfm-cor-6})},
$$
and so $(\prod_{i=1}^nA_i) \cap E_n^{\mu_\omega}(X,G)\not=\emptyset$
 (see Theorem
\ref{etfm-thm-4}), i.e. $m(\{\omega\in \Omega: V\cap
E_n^{\mu_\omega}(X,G)\not=\emptyset\})>0$.
\end{proof}

\begin{lem} \label{lelift}
Let $\pi: (X,G)\rightarrow (Y,G)$
 be a factor map between
$G$-systems, $\mathcal{U}\in \mathcal{C}_Y$ and $\mu\in
\mathcal{M}(X,G)$. Then $h_\mu^- (G,\pi^{-1}\mathcal{U}) =h_{\pi
\mu}^- (G,\mathcal{U})$.
\end{lem}
\begin{proof}
Note that, for each $F\in F (G)$, $P ((\pi^{-1}\mathcal{U})_F)=
\pi^{-1} P (\mathcal{U}_F)$, using \eqref{keyeq1} we have
\begin{eqnarray}
\label{avv} H_{\pi\mu} (\mathcal{U}_F)&= & \inf_{\beta\in P
(\mathcal{U}_F)} H_{\pi\mu} (\beta) =\inf_{\beta\in P
(\mathcal{U}_F)} H_{\mu} (\pi^{-1}\beta)\ \nonumber
\\
&= & \inf_{\beta'\in P ((\pi^{-1}\mathcal{U})_F)} H_\mu(\beta') =
H_\mu((\pi^{-1}\mathcal{U})_F).
\end{eqnarray}
Then the lemma immediately follows when divide $|F|$ on  both sides
of \eqref{avv} and then let $F$ range over a fixed F\o lner sequence
of $G$.
\end{proof}

Then we have

\begin{thm} Let $\pi: (X,G)\rightarrow (Y,G)$ be a factor map
between $G$-systems, $\mu\in \mathcal{M} (X, G)$. Then
$$E_n^{\pi \mu} (Y, G)\subseteq (\pi\times \cdots \times
\pi) E_n^\mu (X, G)\subseteq E_n^{\pi \mu} (Y, G)\cup
\Delta_n(Y)\text{ for each $n\ge 2$}.$$
\end{thm}
\begin{proof} The second inclusion follows directly from the
definition. For the first inclusion, we assume
$(y_1,\cdots,y_n)\in E_n^{\pi\mu} (Y,G)$. For $m\in \N$, take a
closed neighborhood $V_i^m$ of $y_i, i=1,\cdots,n$ with diameter
at most $\frac{1}{m}$ such that $\bigcap_{i=1}^n V_i^m=\emptyset$.
Consider $\mathcal{U}_m=\{ (V_1^m)^c,\cdots,(V_n^m)^c \}\in
\mathcal{C}_Y^{o}$, then $h^-_\mu (G, \pi^{-1}\mathcal{U}_m)=
h^-_{\pi \mu} (G, \mathcal{U}_m)>0$ and so
$\lambda_n(\mu)(\prod_{i=1}^n \pi^{- 1} V_i^m)>0$ by Corollary
\ref{etfm-cor-6} and Lemma \ref{lelift}. Hence $\prod_{i=1}^n
\pi^{- 1} V_i^{m} \cap (\text{supp}(\lambda_n(\mu))\setminus
\Delta_n (X))\neq \emptyset$. Moreover, there exists
$(x_i^m)_1^n\in \prod_{i= 1}^n \pi^{-1} V_i^m\cap E_n^\mu (X, G)$
by Theorem \ref{etfm-thm-4}. We may assume
$(x_1^m,\cdots,x_n^m)\rightarrow (x_1,\cdots,x_n)$ (if necessity
we take a sub-sequence). Clearly, $x_i\in
\pi^{-1}(y_i),i=1,\cdots,n$ and $(x_1,\cdots, x_n)\in E_n^\mu (X,
G)$ by Proposition \ref{thm110}~(2). This finishes the proof of
the theorem.
\end{proof}

\subsection{A variational relation of entropy tuples}

Now we are to show the variational relation of topological and
measure-theoretic entropy tuples.

\begin{thm} \label{vrtm-thm-1}
Let $(X,G)$ be a $G$-system. Then
\begin{description}

\item[1] for each $\mu\in \mathcal{M} (X, G)$ and each $n\ge 2$,
$E_n(X,G)\supseteq E_n^{\mu}(X,G) =\text{supp}
(\lambda_n(\mu))\setminus \Delta_n(X)$.

\item[2] there exists $\mu\in \mathcal{M} (X, G)$ such that $E_n(X,G)=
E_n^{\mu}(X,G)$ for each $n\ge 2$.
\end{description}
\end{thm}
\begin{proof}
1. Let ${(x_i)}_{i=1}^n \in E_n^{\mu}(X,G)$ and $\mathcal{U}\in
\mathcal{C}_X^{o}$ admissible w.r.t. ${(x_i)}_{i=1}^n$. Then if
$\alpha\in \mathcal{P}_X$ is finer than $\mathcal{U}$ then it is
also admissible w.r.t. ${(x_i)}_{i=1}^n$, and so
$h_{\mu}(G,\alpha)>0$ (as $(x_i)_1^n\in E_n^{\mu}(X,G)$), thus
$h^-_\mu(G,\mathcal{U})>0$ by Theorem \ref{etfm-thm-5}. Moreover,
$h_{\text{top}}(G,\mathcal{U})\ge h_\mu^-(G,\mathcal{U})>0$. That
is, ${(x_i)}_{i=1}^n \in E_n(X,G)$, as $\mathcal{U}$ is arbitrary.

2. Let $n\ge 2$. First we have

\medskip

\noindent{\bf Claim.} If $(x_i)_1^n\in E_n(X,G)$ and $\prod_{i= 1}^n
U_i$ is a neighborhood of $(x_i)_1^n$ in $X^{(n)}$ then
$E_n^{\nu}(X,G) \cap \prod_{i= 1}^n U_i\not=\emptyset$ for some
$\nu\in \mathcal{M} (X,G)$.

\begin{proof}[Proof of Claim] With no loss of generality we
assume that $U_i$ is a closed neighborhood of $x_i$, $1\le i\le n$
such that $U_i\cap U_j=\emptyset$ if $x_i\neq x_j$ and $U_i=U_j$
if $x_i=x_j$, $1\le i <j \le n$. Let $\mathcal{U}=\{
U_1^c,\cdots,U_n^c \}$. Then $h_{\text {top}}(G,\mathcal{U})>0$
(as $(x_i)_1^n\in E_n(X,G)$). By Theorem \ref{avpin}, there exists
$\nu\in M(X,G)$ such that $h_{\nu}(G,\mathcal{U})= h_{\text
{top}}(G,\mathcal{U})$, then $\lambda_n(\nu)(\prod_{i=1}^n U_i)>0$
by Corollary \ref{etfm-cor-6}, i.e.
$\text{supp}(\lambda_n(\nu))\cap \prod_{i=1}^n U_i\not=\emptyset$.
As $\prod_{i=1}^n U_i\cap \Delta_n(X)=\emptyset$, one has
$E_n^{\nu}(X,G) \cap \prod_{i=1}^n U_i\not=\emptyset$ by Theorem
\ref{etfm-thm-4}. This ends the proof.
\end{proof}

By claim, for each $n\ge 2$, we can choose a dense sequence of
points $\{(x_1^m,\cdots,x_n^m)\}_{m\in \N}\subseteq E_n(X,G)$ with
$(x_1^m,\cdots,x_n^m )\in E_n^{\nu_n^m}(X,G)$ for some $\nu_n^m \in
\mathcal{M}(X,G)$. Let
\begin{equation*}
\mu=\sum_{n\ge 2} \frac{1}{2^{n-1}}\left(\sum_{m\ge 1}
\frac{1}{2^m}\nu_n^m\right).
\end{equation*}
As if $\alpha\in \mathcal{P}_X$ then
\begin{equation*}
h_{\mu}(G,\alpha)\ge \frac{1}{2^{m+n-1}}h_{\nu_n^m}(G,\alpha)\
(\text{using \eqref{lab-eq}})
\end{equation*}
and so $E_n^{\nu_n^m}(X,G)\subseteq E_n^\mu(X,G)$ for all $n\ge 2$
and $m\in \N$. Thus $(x_1^m,\cdots,x_n^m )\in E_n^\mu(X,G)$. Hence
$$
E_n^\mu(X,G)\supseteq \overline{\{ (x_1^m,\cdots,x_n^m ): m\in \N\}}
\setminus \Delta_n(X) =E_n(X,G),
$$
moreover, $E_n^\mu(X,G)=E_n(X,G)$ (using 1) for each $n\ge 2$.
\end{proof}

\subsection{Entropy tuples of a finite production}

At the end of this section, we shall provide a result about
topological entropy tuples of a finite product.

We say that $G$-measure preserving system $(X, \mathcal{B}, \mu, G)$
is {\it free}, if $g= e_G$ when $g\in G$ satisfies $g x= x$ for
$\mu$-a.e. $x\in X$, equivalently, for $\mu$-a.e. $x\in X$, the
mapping $G\rightarrow G x, g\mapsto g x$ is one-to-one. The
following is proved in \cite[Theorem 4]{GTW}.

\begin{lem} \label{product-Pinsker-factor}
Let $(X, \mathcal{B}, \mu, G)$ and $(Y, \mathcal{D}, \nu, G)$ both
be a free ergodic $G$-measure preserving system with a Lebesgue
space as its base space, with $P_\mu$ and $P_\nu$ Pinsker
$\sigma$-algebras, respectively. Then $P_\mu\times P_\nu$ is the
Pinsker $\sigma$-algebra of the product $G$-measure preserving
system $(X\times Y, \mathcal{B}\times \mathcal{D}, \mu\times \nu,
G)$.
\end{lem}

We say that $(X, G)$ is {\it free} if $g= e_G$ when $g\in G$
satisfies $g x= x$ for each $x\in X$. Let $n\ge 2$. Denote by
$\text{supp} (X, G)$ the {\it support of $(X, G)$}, i.e.
$\text{supp} (X, G)= \bigcup_{\mu\in \mathcal{M} (X, G)} \text{supp}
(\mu)$. $(X, G)$ is called {\it fully supported} if there is an
invariant measure $\mu\in \mathcal{M} (X, G)$ with full support
(i.e. $\text{supp} (\mu)= X$), equivalently, $\text{supp} (X, G)=
X$. Set $\Delta_n^S (X)= \Delta_n (X)\cap (\text{supp} (X,
G))^{(n)}$. Then

\begin{thm} \label{product-entropy-tuple}
Let $(X_i, G), i= 1, 2$ be two $G$-systems and $n\ge 2$. Then
\begin{equation} \label{05120102} E_n (X_1\times X_2, G)= E_n
(X_1, G)\times (E_n (X_2, G)\cup \Delta_n^S (X_2))\cup \Delta_n^S
(X_1)\times E_n (X_2, G).
\end{equation}
\end{thm}
\begin{proof}
Obviously, $E_n (X_1\times X_2, G)\subseteq (\text{supp} (X_1,
G)\times \text{supp}  (X_2, G))^{(n)}$ by Theorem
\ref{vrtm-thm-1}~(2), and so the inclusion of "$\subseteq$"
follows directly from Proposition \ref{thm110}~(3). Now let's turn
to the proof of "$\supseteq$".

First we claim this direction if the actions are both free. Let
\begin{equation*}
((x_i^1, x_i^2))_1^n\in E_n (X_1, G)\times (E_n (X_2, G)\cup
\Delta_n^S (X_2))\cup \Delta_n^S (X_1)\times E_n (X_2, G)
\end{equation*}
 and let
$U_1$ (resp. $U_2$) be any open neighborhood of $(x_i^1)_1^n$ in
$X_1^{(n)}$ (resp. $(x_i^2)_1^n$ in $X_2^{(n)}$). With no loss of generality we assume
$(x_i^1)_1^n\in E_n (X_1, G)$ and $U_1\cap \Delta_n (X_1)=
\emptyset$. Note that $\text{supp} (\lambda_n (\mu))\supseteq
(\text{supp} (\mu))^{(n)}\cap \Delta_n (X_2)$ for each $\mu\in
\mathcal{M} (X_2, G)$, by Theorems \ref{etfm-thm-4} and
\ref{et-decom} we cam choose $\mu_i\in \mathcal{M}^e (X_i, G)$ such
that $U_i\cap (\text{supp} (\mu_i))^{(n)}\neq \emptyset$, $i= 1, 2$.
As the actions are both free, we have

\medskip

\noindent{\bf Claim.} $U_1\times U_2\cap E_n^{\mu_1\times \mu_2}
(X_1\times X_2, G)\neq \emptyset$, and so $U_1\times U_2\cap E_n
(X_1\times X_2, G)\neq \emptyset$, which implies $((x_i^1,
x_i^2))_1^n\in E_n (X_1\times X_2, G)$ from the arbitrariness of
$U_1$ and $U_2$ (using Proposition \ref{thm110} (2)).

\begin{proof}[Proof of Claim]
Let $P_{\mu_i}$ be the Pinsker $\sigma$-algebra of $(X_i,
\mathcal{B}_{X_i}, \mu_i, G), i= 1, 2$. Then $P_{\mu_1}\times
P_{\mu_2}$ forms the Pinsker $\sigma$-algebra of $(X_1\times X_2,
\mathcal{B}_{X_1}\times \mathcal{B}_{X_2}, \mu_1\times \mu_2, G)$
by Lemma \ref{product-Pinsker-factor}. Say $\mu_i= \int_{X_i}
\mu_{i, x_i} d \mu_i (x)$ to be the disintegration of $\mu_i$ over
$P_{\mu_i}$, $i= 1, 2$. Then the disintegration of $\mu_1\times
\mu_2$ over $P_{\mu_1}\times P_{\mu_2}$ is
\begin{equation*}
\mu_1\times \mu_2= \int_{X_1\times X_2} \mu_{1, x_1}\times \mu_{2,
x_2} d \mu_1\times \mu_2(x_1,x_2)
\end{equation*}
Moreover, $\lambda_n (\mu_i)= \int_{X_i} \mu_{i, x_i}^{(n)} d
\mu_i (x_i)$, $i= 1, 2$, which implies
\begin{equation*}
\lambda_n (\mu_1\times \mu_2)= \int_{X_1\times X_2} \mu_{1,
x_1}^{(n)}\times \mu_{2, x_2}^{(n)} d \mu_1\times \mu_2 (x_1,x_2)=
\lambda_n (\mu_1)\times \lambda_n (\mu_2).
\end{equation*}
Then $\text{supp} (\lambda_n (\mu_1\times \mu_2))= \text{supp}
(\lambda_n (\mu_1))\times \text{supp} (\lambda_n (\mu_2))$. So
$U_1\times U_2\cap \text{supp} (\lambda_n (\mu_1\times \mu_2))\neq
\emptyset$ and $U_1\times U_2\cap E_n^{\mu_1\times \mu_2} (X_1\times
X_2, G)\neq \emptyset$ (as $U_1\cap \Delta_n (X_1)= \emptyset$).
This ends the proof of the claim.
\end{proof}

Now let's turn to the proof of general case. Let $(Z, G)$ be any
free $G$-system. Then $G$-systems $(X_i', G)\doteq (X_i\times Z,
G)$, $i= 1, 2$ are both free. Applying the first part to $(X_i',
G)$, $i= 1, 2$ we obtain
\begin{equation} \label{05120101}
E_n (X_1'\times X_2', G)= E_n (X_1', G)\times (E_n (X_2', G)\cup
\Delta_n^S (X_2'))\cup \Delta_n^S (X_1')\times E_n (X_2', G).
\end{equation}
Then applying Proposition \ref{thm110} (3) to the projection
factor maps $(X_1'\times X_2', G)\rightarrow (X_1\times X_2, G)$,
$(X_1', G)\rightarrow (X_1, G)$ and $(X_2', G)\rightarrow (X_2,
G)$ respectively we claim the relation \eqref{05120102}.
\end{proof}

\section{An amenable group action with u.p.e. and c.p.e.}

In this section, we discuss two special classes of an amenable group
action with u.p.e. and c.p.e. We will show that both u.p.e. and c.p.e. are preserved under a
finite product; u.p.e. implies c.p.e. and actions with c.p.e. are fully supported; u.p.e. implies mild mixing; minimal topological
$K$ implies strong mixing if the group considered is commutative.

Let $(X,G)$ be a $G$-system and $\alpha\in \mathcal{P}_X$. We say
that $\alpha$ is {\it topological non-trivial} if
$\overline{A}\subsetneq X$ for each $A\in \alpha$. It is easy to obtain

\begin{lem} \label{chtpk} Let $n\ge 2$
and $\mu\in \mathcal{M}(X,G)$. Then $E_n^{\mu}(X,G)=X^{(n)}\setminus
\Delta_n(X)$ iff $h_{\mu}(G,\alpha)> 0$ for any topological
non-trivial $\alpha= \{A_1, \cdots, A_n\}\in \mathcal{P}_X$.
\end{lem}
\begin{proof} First assume $E_n^{\mu}(X,G)=X^{(n)}\setminus\Delta_n(X)$.
If $\alpha=\{ A_1,\cdots,A_n \}\in \mathcal{P}_X$ is topological
non-trivial, we choose $x_i\in X\setminus \overline{A_i},
i=1,\cdots,n$, then $(x_i)_1^n\in X^{(n)}\setminus \Delta_n(X)$ and
$\alpha$ is admissible w.r.t. $(x_i)_1^n$. Thus
$h_{\mu}(G,\alpha)>0$.

Conversely, we assume $h_{\mu}(G,\alpha)>0$ for any topological
non-trivial $\alpha= \{A_1, \cdots, A_n\}\in \mathcal{P}_X$. Let
$(x_i)_1^n\in X^{(n)}\setminus \Delta_n(X)$. If $\alpha=\{
A_1,\cdots,A_n \}\in \mathcal{P}_X$ is admissible w.r.t.
$(x_i)_1^n$, then it is topological non-trivial and so
$h_{\mu}(G,\alpha)>0$. Thus $(x_i)_1^n\in E_n^{\mu}(X,G)$. This
completes the proof.
\end{proof}

As a direct consequence of Theorem \ref{vrtm-thm-1} and Lemma
\ref{chtpk} one has

\begin{thm} \label{u.p.e.-characterization}
Let $n\ge 2$. Then
\begin{description}

\item[1] $(X,G)$ has u.p.e. of order $n$ iff there exists
$\mu\in \mathcal{M} (X, G)$ such that $h_{\mu}(G,\alpha)> 0$ for any
topological non-trivial $\alpha= \{A_1, \cdots, A_n\}\in
\mathcal{P}_X$.

\item[2] $(X,G)$ has topological $K$ iff there is
$\mu\in \mathcal{M} (X, G)$ such that $h_{\mu}(G,\alpha)>0$ for any
topological non-trivial $\alpha\in \mathcal{P}_X$.
\end{description}
\end{thm}

\begin{de}
We say that $(X, G)$ has {\it c.p.e.} if any non-trivial topological
factor of $(X, G)$ has positive topological entropy.
\end{de}

Blanchard proved that any c.p.e. TDS is fully supported
\cite[Corollary 7]{B1}. As an application of Proposition
\ref{thm110} (3) and Theorem \ref{vrtm-thm-1} we have a similar
result.

\begin{prop} \label{0911152016}
$(X, G)$ has c.p.e. iff $X^{(2)}$ is the closed invariant
equivalence relation generated by $E_2 (X, G)$. Moreover, each
c.p.e. $G$-system is fully supported and each u.p.e. $G$-system has c.p.e. (hence is also fully supported).
\end{prop}
\begin{proof}
It is easy to complete the proof of the first part. Moreover, note
that $(\text{supp} (X, G))^{(2)}\cup \Delta_2 (X)$ is a closed
invariant equivalence relation containing $E_2 (X, G)$ (Theorem
\ref{vrtm-thm-1}). In particular, if $(X, G)$ has c.p.e. then it is
fully supported. Now assume that $(X, G)$ has u.p.e., thus $E_2 (X, G)= X^{(2)}\setminus \Delta_2 (X)$ and so  $X^{(2)}$ is the closed invariant
equivalence relation generated by $E_2 (X, G)$, particularly, $(X, G)$ has c.p.e. This finishes our proof.
\end{proof}

The following lemma is well known, in the case of $\Z$ see for
example \cite[Lemma 1]{PS}.

\begin{lem} \label{generated relation}
Let $(X_i, G)$ be a $G$-system and $\Delta_2 (X_i)\subseteq
A_i\subseteq X_i\times X_i$ with $<A_i>$ the closed invariant
equivalence relation generated by $A_i$, $i= 1, 2$. Then
$<A_1>\times <A_2>$ is the closed invariant equivalence relation
generated by $A_1\times A_2$.
\end{lem}

Thus we have

\begin{cor} \label{easy-2}
Let $(X_1, G)$ and $(X_2, G)$ be two $G$-systems and $n\ge 2$.
\begin{enumerate}
\item If $(X_1,G)$ and $X_2,G)$ both have u.p.e. of order $n$ then so does $(X_1\times X_2, G)$.

\item If $(X_1,G)$ and $(X_2,G)$ both have topological $K$
then so does $(X_1\times X_2,G)$.

\item If $(X_1,G)$ and $(X_2,G)$ both have c.p.e.
then so does $(X_1\times X_2,G)$.
\end{enumerate}
\end{cor}
\begin{proof}
By Proposition \ref{0911152016}, any
$G$-system having u.p.e. is full supported, then (1) and (2) follow from Theorem \ref{product-entropy-tuple} directly. Using Theorem
\ref{product-entropy-tuple} and Lemma \ref{generated relation}, we can obtain (3) similarly.
\end{proof}

In the following several sub-sections, we shall discuss more properties of an amenable group action with u.p.e.

\subsection{U.p.e. implies weak mixing of all orders}

Following the idea of the proof of \cite[Proposition 2]{B1}, it is easy to obtain
the following result.

\begin{lem} \label{another observation}
Let $\{U_1^c, U_2^c\}\in \mathcal{C}_X$. If
\begin{equation} \label{assumption}
\limsup_{n\rightarrow +\infty} \frac{1}{n} \log N \left(\bigvee_{i=
1}^n g_i^{- 1} \{U_1^c, U_2^c\}\right)> 0
\end{equation}
for some sequence $\{g_i: i\in \mathbb{N}\}\subseteq G$ then there
exist $1\le j_1< j_2$ with $U_1\cap g_{j_1} g_{j_2}^{- 1} U_2\neq
\emptyset$.
\end{lem}
\begin{proof}
Assume the contrary that for each $1\le j_1< j_2$, $U_1\cap
g_{j_1} g_{j_2}^{- 1} U_2= \emptyset$ and so $g_{j_1}^{- 1}
U_1\subseteq g_{j_2}^{- 1} U_2^c$. That is, for each $i\in
\mathbb{N}$ one has $g_i^{- 1} U_1\subseteq \bigcap_{j\ge i}
g_j^{- 1} U_2^c$.

Let $n\in \mathbb{N}$. Now for each $x\in X$ consider the first
$i\in \{1, \cdots, n\}$ such that $g_i x\in U_1$, when there exists
such an $i$. We get that the Borel cover $\bigvee_{j= 1}^n g_j^{- 1}
\{U_1^c, U_2^c\}$ admits a sub-cover
\begin{eqnarray*}
\left\{\bigcap_{s= 1}^{i- 1} g_s^{- 1} U_1^c\cap \bigcap_{t= i}^n
g_t^{- 1} U_2^c: i= 1, \cdots, n\right\}\cup \left\{\bigcap_{s=
1}^{n} g_s^{- 1} U_1^c\right\}.
\end{eqnarray*}
Moreover, $N(\bigvee_{j= 1}^n g_j^{- 1} \{U_1^c, U_2^c\})\le n+1$, a
contradiction with the assumption.
\end{proof}

We say that $(X, G)$ is {\it transitive} if for each non-empty
open subsets $U$ and $V$, the {\it return time set}, $N (U,
V)\doteq \{g\in G: U\cap g^{- 1} V\neq \emptyset\}$, is non-empty.
It is not hard to see that if $X$ has no isolated point then the
transitivity of $(X, G)$ is equivalent to that $N (U, V)$ is
infinite for each non-empty open subsets $U$ and $V$. Let $n\ge
2$. We say that $(X, G)$ is {\it weakly mixing of order $n$} if
the product $G$-system $(X^{(n)}, G)$ is transitive; if $n= 2$ we
call it simply {\it weakly mixing}. We say that $(X, G)$ is called
{\it weakly mixing of all orders} if for each $n\ge 2$ it is
weakly mixing of order $n$, equivalently, the product $G$-system
$(X^{\N},G)$ is transitive. It's well known that for
$\mathbb{Z}$-actions u.p.e. implies weakly mixing of all orders
\cite{B1}. In fact, this result holds for a general countable
discrete amenable group action by applying Corollary \ref{easy-2}
and Lemma \ref{another observation} to a u.p.e. $G$-system as many times as required.

\begin{thm} \label{weak-mixing}
Each u.p.e. $G$-system is weakly mixing of all orders.
\end{thm}

\subsection{U.p.e. implies mild mixing}

We say that $(X, G)$ is {\it mildly mixing} if the product
$G$-system $(X\times Y, G)$ is transitive for each transitive
$G$-system $(Y, G)$ containing no isolated points. We shall prove
that each u.p.e. $G$-system is mildly mixing. Note that similar to
the proof of Lemma \ref{another observation}, it is easy to show
that each non-trivial u.p.e. $G$-system contains no any isolated
point, thus the result in this sub-section strengthens Theorem
\ref{weak-mixing}. Before proceeding first we need

\begin{lem} \label{easy observation}
Let $\mu\in \mathcal{M} (X, G)$, $\mathcal{U} =\{U_1,\cdots, U_n
\}\in \mathcal{C}_X^{o}$, $\alpha\in \mathcal{P}_X$ and
$\{g_i\}_{i\in \mathbb{N}}\subseteq G$ be a sequence of pairwise
distinct elements. Then
\begin{description}

\item[1] $\limsup_{n\rightarrow
+\infty} \frac{1}{n} \log N(\bigvee_{i= 1}^n g_i^{- 1} \alpha)\ge
h_\mu (G, \alpha)$.

\item[2] if
$h_{\text{top}} (G, \mathcal{U})> 0$ then $\limsup_{n\rightarrow
+\infty} \frac{1}{n} \log N(\bigvee_{i= 1}^n g_i^{- 1} \mathcal{U})>
0$.
\end{description}
\end{lem}
\begin{proof}
1 follows directly from Lemma \ref{lem-436}~(4). Now let's turn to
the proof of 2.

By Theorem \ref{avpin} there exists $\mu\in \mathcal{M}^e (X, G)$
such that $h_\mu (G, \mathcal{U})= h_{\text{top}} (G,
\mathcal{U})> 0$. Let $P_\mu$ be the Pinsker $\sigma$-algebra of
$(X, \mathcal{B}_X^\mu, \mu, G)$. As $\lambda_n(\mu)(\prod_{i=1}^n
U_i^c)= \int_X \prod_{i=1}^n \E(1_{U_i^c}|P_{\mu}) d\mu >0$ (see
Corollary \ref{etfm-cor-6}), repeating the same procedure of the
proof of Theorem \ref{etfm-thm-5} we can obtain some $M\in \N,
D\in P_\mu$ and $\alpha\in \mathcal{P}_X$ such that $\mu(D)>0$ and
if $\beta\in \mathcal{P}_X^\mu$ is finer than $\mathcal{U}$ then
$H_{\mu}(\alpha|\beta \vee P_{\mu}) \le
H_{\mu}(\alpha|P_{\mu})-\epsilon$, here $\epsilon=\frac
{\mu(D)}{M} \log (\frac {n}{n-1})>0$. Note that there exists $K\in
F (G)$ such that if $F\in F (G)$ satisfies $(F F^{-1}\setminus
\{e_G\})\cap K= \emptyset$ then
$|\frac{1}{|F|}H_\mu(\alpha_F|P_\mu)-
H_\mu(\alpha|P_\mu)|<\frac{\epsilon}{2}$ (see Theorem
\ref{Danil}). Obviously, there exists a sub-sequence $\{s_1< s_2<
\cdots\}\subseteq \mathbb{N}$ such that $\frac{i}{s_i}\ge
\frac{1}{2 |K|+1}$ for each $i\in \mathbb{N}$ and $g_{s_i}
g_{s_j}^{- 1}\notin K$ when $i\neq j$. Then for each $n\in
\mathbb{N}$ one has
\begin{align}\label{guanj-eq-1}
|\frac{1}{n} H_\mu \left(\bigvee_{i= 1}^n g_{s_i}^{-1} \alpha
|P_\mu\right)- H_\mu (\alpha |P_\mu)|< \frac{\epsilon}{2}.
\end{align}
Now let $n\in \mathbb{N}$. If $\beta_n\in \mathcal{P}_X^\mu$ is
finer than $\bigvee_{i=1}^n g_{s_i}^{-1} \mathcal{U}$, then
$g_{s_i}\beta_n\succeq \mathcal{U}$ for each $i=1,\cdots,n$, and so
\begin{eqnarray*}
H_{\mu}(\beta_n)&\ge & H_\mu\left(\beta_n \vee \bigvee_{i=1}^n
g_{s_i}^{-1}\alpha|P_\mu\right)-
H_\mu\left(\bigvee_{i=1}^n g_{s_i}^{-1}\alpha|\beta_n \vee P_\mu\right) \\
&\ge & H_\mu\left(\bigvee_{i=1}^n
g_{s_i}^{-1}\alpha|P_\mu\right)-\sum_{i=1}^n
H_\mu(\alpha|g_{s_i} \beta_n \vee P_\mu) \\
&\ge & H_\mu\left(\bigvee_{i=1}^n g_{s_i}^{-1}\alpha|P_\mu\right)-
n(H_\mu(\alpha|P_\mu)-\epsilon) \\
&\ge &
n\left(H_\mu(\alpha|P_\mu)-\frac{\epsilon}{2}\right)-n(H_\mu(\alpha|P_\mu)-\epsilon)
\ \ \ \  \text{(by \eqref{guanj-eq-1})}\\
&= & \frac{n \epsilon}{2}.
\end{eqnarray*}
Hence, $\frac{1}{n}H_\mu(\bigvee_{i=1}^n g_{s_i}^{-1}\mathcal{U})\ge
\frac{\epsilon}{2}$. Which implies
\begin{eqnarray*}
 \limsup_{n\rightarrow +\infty} \frac{1}{n} \log
N \left(\bigvee_{i= 1}^n g_i^{- 1} \mathcal{U}\right)&\ge&
\limsup_{m\rightarrow +\infty}
\frac{1}{s_m} H_\mu \left(\bigvee_{i= 1}^{s_m} g_i^{-1} \mathcal{U}\right)\\
&\ge &\limsup_{m\rightarrow +\infty} \frac{m}{s_m}\cdot
\frac{1}{m} H_\mu \left(\bigvee_{i= 1}^{m} g_{s_i}^{-1}
\mathcal{U}\right)\ge \frac{\epsilon}{2 (2|K|+ 1)}> 0.
\end{eqnarray*}
This ends the proof of the lemma.
\end{proof}





Now we claim that u.p.e. implies mild mixing.

\begin{thm}
Let $(X, G)$ be a u.p.e. $G$-system. Then $(X, G)$ is mildly mixing.
\end{thm}
\begin{proof}
Let $(Y, G)$ be any transitive $G$-system containing no isolated
points and $(U_Y, V_Y)$ any pair of non-empty open subsets of $Y$.
It remains to show that $N (\overline{U_X}\times U_Y,
\overline{V_X}\times V_Y)\neq \emptyset$ for each pair of non-empty
open subsets $(U_X, V_X)$ of $X$. As $(Y, G)$ is transitive, there
is $g\in G$ with $U_Y\cap g^{- 1} V_Y\neq \emptyset$. Set $W_Y=
U_Y\cap g^{- 1} V_Y$. Then
\begin{equation*}
N(\overline{U_X}\times U_Y, \overline{V_X}\times V_Y)\supseteq g
N(\overline{U_X}\times W_Y, \overline{g^{- 1} V_X}\times W_Y).
\end{equation*}
Now it suffices to show that $N(\overline{U_X}\times W_Y,
\overline{g^{- 1} V_X}\times W_Y)\neq \emptyset$.

If $\overline{U_X}\cap \overline{g^{- 1} V_X}\neq \emptyset$ then
the proof is finished, so we assume $\overline{U_X}\cap
\overline{g^{- 1} V_X}= \emptyset$. As $(Y, G)$ is a transitive
$G$-system containing no isolated points, there exists $g_1'\in
G\setminus \{e_G\}$ with $g_1' W_Y\cap W_Y\neq \emptyset$. Now find
$g_2'\in G\setminus \{e_G, (g_1')^{- 1}\}$ with $g_2' (g_1' W_Y\cap
W_Y)\cap (g_1' W_Y\cap W_Y)\neq \emptyset$. By induction, similarly
there exists a sequence $\{g_n'\}_{n\ge 1}\subseteq G$ such that for
each $j\ge 1$ one has $g_j'\in G\setminus \{e_G, (g_{j- 1}')^{- 1},
(g_{j- 1}' g_{j- 2}')^{- 1}, \cdots, (g_{j- 1}' g_{j- 2}'\cdots
g_1')^{- 1}\}$ and for each $n\in \mathbb{N}$ it holds that
\begin{eqnarray} \label{Ip-set}
W_Y\cap \bigcap_{1\le i\le j\le n} (g_j' g_{j- 1}'\cdots g_i'
W_Y)\neq \emptyset.
\end{eqnarray}
Set $g_n= g'_n g'_{n- 1}\cdots g'_1$ for each $n\in \mathbb{N}$.
Then $g_i\neq g_j$ if $1\le i\neq j$. Note that $\overline{U_X}\cap
\overline{g^{- 1} V_X}= \emptyset$ and $(X, G)$ is u.p.e., then
$h_{\text{top}} (G, \{\overline{U_X}^c, \overline{g^{- 1} V_X}^c\})>
0$ and so by Lemma \ref{easy observation} one has
\begin{equation*}
\limsup_{n\rightarrow +\infty} \frac{1}{n} \log N \left(\bigvee_{i=
1}^n g_i^{- 1} \{\overline{U_X}^c, \overline{g^{- 1}
V_X}^c\}\right)> 0.
\end{equation*}
Then by Lemma \ref{another observation}, there exists $1\le i< j$
such that
\begin{eqnarray*}
\emptyset\neq \overline{U_X}\cap g_{i} g_j^{- 1} \overline{g^{- 1}
V_X}= \overline{U_X}\cap (g_j' g_{j- 1}'\cdots g_{i+ 1}')^{- 1}
\overline{g^{- 1} V_X},
\end{eqnarray*}
which implies (using \eqref{Ip-set})
\begin{equation*}
g_j' g_{j- 1}'\cdots g_{i+ 1}'\in N (\overline{U_X}, \overline{g^{-
1} V_X})\cap N (W_Y, W_Y)= N (\overline{U_X}\times W_Y,
\overline{g^{- 1} V_X}\times W_Y)\neq \emptyset.
\end{equation*}
This finishes the proof of the theorem.
\end{proof}

\subsection{Minimal topological $K$-actions of an amenable group}

We say that $(X, G)$ is {\it strongly mixing} if $N (U, V)$ is
cofinite (i.e. $G\setminus N (U, V)$ is finite) for each pair of
non-empty open subsets $(U, V)$ of $X$. It's proved in \cite{HSY}
that any topological $K$ minimal  $\Z$-system is strongly mixing.
In fact, this result holds again in general case of considering a
commutative countable discrete amenable group. In the remaining part
of this sub-section we are to show it.

Denote by $\mathcal{F}_{inf} (G)$ the family of all infinite
subsets of $G$. Let $d$ be the compatible metric on $(X, G)$, $S=
\{g_1, g_2, \cdots\}\in \mathcal{F}_{inf} (G)$ and $n\ge 2$.
$RP_S^n (X, G)\subseteq X^{(n)}$ is defined by $(x_i)_1^n\in
RP_S^n (X, G)$ iff for each neighborhood $U_{x_i}$ of $x_i, 1\le
i\le n$ and $\epsilon> 0$ there exists $x_i'\in U_{x_i}, 1\le i\le
n$ and $m\in \mathbb{N}$ with $\max_{1\le k, l\le n} d (g_m^{- 1}
x_k', g_m^{- 1} x_l')\le \epsilon$. Obviously, the definition of
$RP_S^n (X, G)$ is independent of the selection of compatible
metrics. As a direct corollary of Lemma \ref{easy observation} we
have

\begin{lem} \label{media lemma}
Let $n\ge 2$ and $S\in \mathcal{F}_{inf} (G)$. If $(X, G)$ is u.p.e.
of order $n$ then $RP_S^n (X, G)= X^{(n)}$.
\end{lem}
\begin{proof} Assume the contrary that there is
$S= \{g_1, g_2, \cdots\}\in \mathcal{F}_{inf} (G)$ such that
$RP_S^n (X, G)\subsetneq X^{(n)}$. Fix such an $S$ and take
$(x_i)_1^n\in X^{(n)}\setminus RP_S^n (X, G)$. Then we can find a
closed neighborhood $U_i$ of $x_i, 1\le i\le n$ and $\epsilon> 0$
such that if $x_i'\in U_i, 1\le i\le n$ and $m\in \mathbb{N}$ then
$\max_{1\le k, l\le n} d (g_m^{- 1} x_k', g_m^{- 1} x_l')>
\epsilon$. Now let $\{C_1, \cdots, C_k\}$ ($k\ge n$) be a closed
cover of $X$ such that the diameter of each $C_i$, $1\le i\le k$,
is at most $\epsilon$ and if $i\in \{ 1,\cdots, n\}$ then $x_i\in
(C_i)^{\text{o}}\subseteq C_i\subseteq U_i$. Clearly
$(x_i)_1^n\notin \Delta_n (X)$, we may assume that $\{C_1^c,
\cdots, C_n^c\}$ forms an admissible open cover of $X$ w.r.t.
$(x_i)_1^n$, and so $h_{\text{top}} (G, \{C_1^c, \cdots, C_n^c\})>
0$. Moreover,
\begin{equation} \label{later estimate2}
\limsup_{m\rightarrow +\infty} \frac{1}{m} \log N\left(\bigvee_{i= 1}^m
g_i^{- 1} \{C_1^c, \cdots, C_n^c\}\right)> 0\ (\text{by Lemma \ref{easy
observation}}).
\end{equation}
Whereas, it's not hard to claim that for each $i\in \{1, \cdots,
k\}$ and $m\in \mathbb{N}$ there exists $j_i^m\in \{1, \cdots,
n\}$ such that $g_m C_i\cap C_{j_i^m}= \emptyset$. Otherwise, for
some $i_0\in \{1, \cdots, k\}$ and $m_0\in \N$, it holds that for
each $i\in \{1,\cdots,n\}$, $g_{m_0} C_{i_0}\cap C_i\neq
\emptyset$, let $y_i\in g_{m_0} C_{i_0}\cap C_i\subseteq U_i$.
Thus $\max_{1\le k, l\le n} d (g_{m_0}^{- 1} y_k, g_{m_0}^{- 1}
y_l)$ is at most the diameter of $C_{i_0}$, which is at most
$\epsilon$, a contradiction with the selection of $y_1, \cdots,
y_n$. Therefore, $C_i\subseteq \bigcap_{m\in \mathbb{N}} g_m^{- 1}
C_{j_i^m}^c$ for each $i\in \{ 1,\cdots,k\}$, which implies
$N(\bigvee_{i= 1}^m g_i^{- 1} \{C_1^c, \cdots, C_n^c\})\le k$ for
each $m\in \mathbb{N}$, a contradiction with \eqref{later
estimate2}. This finishes the proof of the lemma.
\end{proof}

Then we have

\begin{thm}
Let $U$ and $V$ be non-empty open subsets of $X$. If $(X, G)$ is
minimal and topological $K$ then there exists $g_1, \cdots, g_l\in
G$ ($l\in \N$) such that $\bigcup_{i= 1}^l g_i N (U, V) g_i^{-
1}\subseteq G$ is cofinite. In particular, if $G$ is commutative
then $(X, G)$ is strongly mixing.
\end{thm}
\begin{proof}
As $(X, G)$ is a minimal $G$-system, there exist distinct elements
$g_1, \cdots, g_N\in G$ such that $\bigcup_{i= 1}^N g_i U= X$. Let
$\delta> 0$ be a Lebesgue number of $\{g_1 U, \cdots, g_N U\}\in
\mathcal{C}_X^{o}$ and set
\begin{equation*}
B= \left\{g\in G: \exists x_i\in g_i V (1\le i\le N)\ \text{s.t.}\
\max_{1\le k, l\le N} d (g^{- 1} x_k, g^{- 1} x_l)<
\frac{\delta}{2}\right\}.
\end{equation*}
As $(X, G)$ is topological $K$, $(g_i x)_1^N\in RP_S^n (X, G)$ for
each $S\in \mathcal{F}_{inf} (G)$ and $x\in X$ by Lemma \ref{media
lemma}. This implies $B\cap S\neq \emptyset$ for each $S\in
\mathcal{F}_{inf} (G)$. Hence, $G\setminus B$ is a finite subset,
i.e. $B\subseteq G$ is cofinite. Now if $g\in B$ then for each
$i\in \{1,\cdots, N\}$ there exists $x_i\in g_i V$ such that
$\max_{1\le k, l\le N} d (g^{- 1} x_k, g^{- 1} x_l)<
\frac{\delta}{2}$. Moreover, the diameter of
$\{g^{_1}x_1,\cdots,g^{-1}x_N\}$ is less than $\delta$. So by
the selection of $\delta$, for some $1\le k\le N$, $g^{- 1} x_1,
\cdots, g^{- 1} x_N\in g_k U$, in particular, $g_k U\cap g^{- 1}
g_k V\neq \emptyset$. That is, for each $g\in B$ there exists
$k\in \{1, \cdots, N\}$ such that $g_k^{- 1} g g_k\in N (U, V)$,
i.e. $B\subseteq \bigcup_{k= 1}^N g_k N (U, V) g_k^{- 1}$.
\end{proof}

\section*{Acknowledgements}

The first author is supported by NSFC (10911120388, 11071231), Fok Ying Tung Education Foundation and the Fundamental Research Funds for the Central Universities, the second author is supported by NSFC (11071231) and the third author is supported by NSFC (10801035), FANEDD (201018) and a grant from Chinese Ministry of Education (200802461004).

\end{document}